\documentclass[a4paper,11pt,reqno]{amsart}
\usepackage{amsthm}
\usepackage{amsmath}
\usepackage{enumerate}
\usepackage{amsfonts}
\usepackage{amssymb}
\usepackage{fullpage}
\usepackage{amsmath,amscd}
\usepackage{graphicx}
\usepackage{array}
\usepackage{amsmath,amssymb,amsfonts}
\usepackage[utf8]{inputenc} 
\usepackage[T1]{fontenc}
\usepackage[english,french]{babel}

   \usepackage{gfsartemisia}

 \usepackage{xspace}
\usepackage{xparse}
\usepackage{graphicx}
\usepackage{float}
\usepackage{pdfpages}
\usepackage[a4paper, margin = 3cm, bottom = 3cm]{geometry}
\usepackage{ifpdf}
\usepackage{marginnote}
\usepackage{hyperref}
\usepackage{tikz}
\usepackage{csquotes}
\usetikzlibrary{calc,matrix,arrows,shapes,decorations.pathmorphing,decorations.markings,decorations.pathreplacing}
\hypersetup{pdfborder=0 0 0, 
	    colorlinks=true,
	    citecolor=black,
	    linkcolor=blue,
	    urlcolor=black,
	    pdfauthor={Guillaume Tahar,Quentin Gendron}
	   }

\usepackage[
	    style=alphabetic,
            isbn=false,
            doi=false,
            url=false,
            backend=bibtex, 
            maxnames = 5,
            sorting=nty,
           ]{biblatex}
\defbibheading{bibliography}[\bibname]{%
  \section*{Références}%
}

\DeclareFieldFormat[article]{title}{{\it #1}}
\DeclareFieldFormat{journaltitle}{{\rm #1}}
\renewbibmacro{in:}{%
  \ifentrytype{article}{}{\printtext{\bibstring{in}\intitlepunct}}}
  \bibliography{res_biblio}

\newcommand{\RR}{\mathbf{R}}

\newcommand{\CC}{\mathbf{C}}
\newcommand{\NN}{\mathbf{N}}

\newcommand{\PP}{{\mathbb{P}}}

\newcommand{\pgcd}{{\rm pgcd}}

\newcommand{\ord}{\operatorname{ord}\nolimits}

\renewcommand{\tilde}{\widetilde}

\newcommand{\rot}{{\rm rot}} 
\newcommand{\ind}{{\rm Ind}}

\newcommand\Res{\operatorname{Res}}
\newcommand{\Resk}[1][k]{\Res^{#1}}
\newcommand{\Resquad}{\Res^{2}}
\DeclareDocumentCommand{\rec}{ O{a} O{n}}{\lbrace #1 \rbrace^{#2}}

\DeclareDocumentCommand{\appresk}{ O{g} O{k}}{\mathfrak{R}_{#1}^{#2}}
\DeclareDocumentCommand{\appresquad}{ O{g}}{\mathfrak{R}_{#1}^{2}}
\DeclareDocumentCommand{\appresquadrho}{ O{g} O{\rho}}{\mathfrak{R}_{#1}^{2,#2}}

\DeclareDocumentCommand{\espresk}{O{g} O{k}}{\mathcal{R}_{#1}^{#2}}
\DeclareDocumentCommand{\espresquad}{O{g}}{\mathcal{R}_{#1}^{2}}

\newcommand{\moduli}[1][g]{{\mathcal M}_{#1}}
\newcommand{\quadomoduli}[1][g]{{\Omega^2\mathcal M}_{#1}}
\newcommand{\omoduli}[1][g]{{\Omega\mathcal M}_{#1}}

\newcommand{\whX}{\widehat{X}}
\newcommand{\whomega}{\widehat{\omega}}

\newcommand{\whz}{\widehat{z}}

\newcommand{\noeuds}{n\oe{}uds\xspace}

\def\={\;=\;} 
\newcommand{\npa}{\mathfrak{p}}
\newcommand{\nim}{\mathfrak{i}}

\newcommand{\Quentin}[1]{{\color{purple}{Quentin: #1}}}
\newtheorem{thm}{Théorème}[section]
\newtheorem{cor}[thm]{Corollaire}
\newtheorem{prop}[thm]{Proposition}
\newtheorem{lem}[thm]{Lemme}

\theoremstyle{definition}
\newtheorem{defn}[thm]{Définition}
\theoremstyle{remark}
\newtheorem{rem}[thm]{Remarque}
\theoremstyle{definition}

\theoremstyle{definition}

\theoremstyle{definition}

\theoremstyle{definition}

\numberwithin{equation}{section} 
  \usepackage{gfsartemisia}

\title{Différentielles quadratiques à singularités prescrites}
\author{Quentin Gendron et Guillaume Tahar}

\address[Quentin Gendron]{Instituto de Matem\'{a}ticas de la UNAM Ciudad Universitaria, CDMX, 04510, M\'{e}xico}
\email{quentin.gendron@im.unam.mx}

\address[Guillaume Tahar]{Beijing Institute of Mathematical Sciences and Applications, Huairou District, Beijing, China}
\email{guillaume.tahar@bimsa.cn}

\date{\today}
\keywords{Quadratic differential, Flat surface, Strata, Residue}
\begin{document}

\selectlanguage{english}

\begin{abstract}
The local invariants of a meromorphic quadratic differential on a compact Riemann surface are the orders of zeros and poles, and the residues at the poles of even orders. The main result of this paper is that with few exceptions, every pattern of local invariants can be obtained by a quadratic differential on some Riemann surface. The exceptions are completely classified and only occur in genera zero and one. Moreover, in the case of a nonconnected stratum, we show that, with three exceptions in genus one, each configuration of invariants can be realized in each non-hyperelliptic connected component of the stratum. In the hyperelliptic components with two poles the residues at both poles coincide. These results are obtained using the flat metric induced by the differentials. We give an application by bounding the number of disjoint cylinders on a primitive quadratic differential.
\end{abstract}

\selectlanguage{french}

\maketitle
\setcounter{tocdepth}{1}
\tableofcontents

\section{Introduction}

Sur une surface de Riemann $X$ de genre $g$, une {\em différentielle quadratique} est une section méromorphe du fibré $K_{X}^{\otimes 2}$, où $K_{X}$ est le fibré en droites canonique de $X$. Localement, une différentielle quadratique s'écrit $f(z)(dz)^{2}$, où $f$ est une fonction méromorphe. Une différentielle quadratique est {\em primitive} si elle n'est pas le carré d'une différentielle abélienne sur~$X$.
Les invariants locaux d'une différentielle quadratique $\xi$ en un point~$P$ sont l'{\em ordre} de la différentielle en~$P$ et le {\em résidu quadratique} $\Resquad_{P}(\xi)$ si $P$ est un pôle d'ordre pair de $\xi$ (voir par exemple \cite{strebel}).

Dans cet article, nous nous proposons de répondre à un problème de type Riemann-Hilbert, consistant à déterminer sous quelles conditions une configuration d'invariants locaux est réalisable par une différentielle quadratique sur une surface de Riemann, c'est-à-dire un objet global. Le problème correspondant pour les différentielles abéliennes a été résolu dans \cite{getaab}. Nous nous concentrerons donc dans cet article sur le cas des différentielles quadratiques primitives.
Plus précisément, nous nous proposons de répondre à la question suivante:
\begin{center}
{\em \'Etant donnés les ordres des zéros et des pôles ainsi que les résidus aux pôles, existe-t-il une différentielle quadratique primitive ayant ces invariants locaux?}
\end{center}

Il s'agit donc de déterminer les obstructions globales. Le degré du fibré canonique est un invariant topologique, ce qui implique que la somme des ordres des zéros et des pôles d'une différentielle quadratique est égale à $4g-4$. Notons que contrairement au cas des différentielles abéliennes, la somme des résidus des différentielles quadratiques ne s'annule pas forcément.

Jusqu'ici, les résultats d'existence les plus fins sont ceux de \cite{masm,Diaz-Marin} sur l'existence de différentielles quadratiques, respectivement holomorphes et méromorphes, dont les ordres des singularités sont prescrits, mais ils ne disent rien des résidus dans le cas méromorphe.

\smallskip
\par
\subsection{Définitions.}

Nous désignons par
$$\mu:=(a_{1},\dots,a_{n};-b_{1},\dots,-b_{p};-c_{1},\dots,-c_{r};\underbrace{-2,\dots,-2}_{s}) \,,$$
une décomposition de $4g-4$ en une somme d'entiers relatifs où les $a_{i}$ sont des entiers supérieurs ou égaux à $-1$, les $b_{i}$ sont des entiers positifs pairs, les~$c_{i}$ sont des entiers positifs impairs supérieurs ou égaux à $3$ et qui contient~$s$ fois~$-2$. Afin de simplifier les notations, si un ordre $a$ apparaît $k$ fois, on le notera~$\rec[a][k]$. Par exemple, la décomposition $\mu$ sera notée $(a_{1},\dots,a_{n};-b_{1},\dots,-b_{p};-c_{1},\dots,-c_{r};\rec[-2][s])$. Plus généralement, une suite $(a,\dots,a)$ de~$k$ nombres complexes tous identiques pourra être notée $\rec[a][k]$.

Nous appelons {\em zéro} d'une différentielle quadratique une singularité d'ordre supérieur ou égal à $-1$ et {\em pôle} une singularité d'ordre inférieur ou égal à $-2$. 
Enfin, nous posons $n=\npa+\nim$ où $\nim$ désigne le nombre de zéros d'ordre impair et $\npa$ celui des zéros d'ordre pair. Dans l'ordre de la séquence $(a_{1},\dots,a_{n})$, les ordres impairs seront toujours écrits avant les pairs. Ainsi, $(a_{1},\dots,a_{\nim})$ désigne les zéros d'ordres impairs.

La {\em strate primitive} $\quadomoduli(\mu)$ paramètre les paires $(X,\xi)$ formées d'une surface de Riemann~$X$ de genre $g$ et d'une différentielle quadratique {\em primitive} $\xi$ de type~$\mu$ sur $X$. Les strates primitives (non vides) de différentielles quadratiques sont des variétés orbifoldes de dimension $2g-2+n+p+r+s$.
 
Rappelons maintenant que pour une différentielle $\xi$ et un point $P$ de $X$, il existe une coordonnée $z$ au voisinage de $P$  telle que $\xi$ s'écrit
\begin{equation}\label{eq:standard_coordinates}
    \begin{cases}
      z^m\, (dz)^{2} &\text{si $m\geq -1$ ou $m$ est impair,}\\
      \left(\frac{r}{z}\right)^{2}(dz)^{2} &\text{ avec $r\in \CC^{\ast}$ si $m = -2$,}\\
        \left(z^{m/2} + \frac{r}{z}\right)^{2}(dz)^{2} &\text{avec $r \in \CC$ si $m < -2$ et $m$ pair.}
    \end{cases}
\end{equation}
Le {\em résidu quadratique} $\Resquad_{P}(\xi)$ de $\xi$ en $P$ est le carré de~$r$. Ainsi, le résidu quadratique est toujours non nul dans le cas des pôles doubles. Toutefois il n'existe pas de théorème des résidus pour les différentielles quadratiques. Ainsi, étant donnée une strate $\quadomoduli(\mu)$, nous définissons l'{\em espace résiduel de type $\mu$} par\begin{equation}
\espresquad(\mu) := \CC^{p}\times(\CC^{\ast})^{s}.
\end{equation}
Cet espace paramètre les configurations de résidus quadratiques que pourrait prendre une différentielle de $\quadomoduli(\mu)$. Afin d'obtenir des constructions plus agréables via la géométrie plate, sans incidence sur les énoncés des résultats, nous multiplions le résidu quadratique par $-4\pi^{2}$, que nous continuerons d'appeler {\em résidu quadratique} ou plus simplement {\em résidu}.

L'{\em application résiduelle} est donnée par
\begin{equation}
\appresquad(\mu) \colon \quadomoduli(\mu) \to \espresquad(\mu):\ (X,\xi) \mapsto (\Resquad_{P_{i}}(\xi)),
\end{equation}
où les $P_{i}$ sont les pôles  d'ordre pair de $\xi$. Insistons sur le fait que par définition, les différentielles quadratiques de $\quadomoduli(\mu)$ sont {\em primitives}. 
Résoudre la question centrale de cet article revient à déterminer l'image de cette application pour chaque strate et chaque composante connexe de la strate quand celle-ci n'est pas connexe.

\smallskip
\par
\subsection{Genre supérieur ou égal à un.}

Nous sommes maintenant en mesure d'énoncer les résultats de cet article. Rappelons que $\quadomoduli(\mu)$ paramètre les différentielles quadratiques {\em primitives} de type $\mu$. En général, les strates de différentielles quadratiques méromorphes ne sont pas connexes. Leurs composantes connexes ont été classifiées dans le théorème~1.3 de \cite{chge} dans le cas où la différentielle possède des pôles d'ordre inférieur ou égal à $-2$.  Les deux résultats suivants décrivent l'image de la restriction de l'application résiduelle à chaque composante connexe pour les strates de différentielles sur les surfaces de Riemann de genre supérieur ou égal à $1$.
\smallskip
\par
Si le genre $g$ est supérieur ou égal à $2$, rappelons qu'il existe quatre familles pour lesquelles il existe une composante hyperelliptique et une composante non-hyperelliptique. Toutefois, nous montrons dans le lemme~\ref{lem:nocomphyp} que seulement deux familles apparaissent si au moins un pôle est d'ordre pair. Dans ce cas, la différentielle possède un unique pôle. Les autres strates sont connexes. 
\begin{thm}\label{thm:ggeq2}
Pour tout $g\geq2$ la restriction à chaque composante connexe de la strate $\quadomoduli(\mu)$ de l'application résiduelle $\appresquad(\mu) \colon \quadomoduli(\mu) \to \espresquad(\mu)$ est surjective.
%
\end{thm}

Le cas des strates de genre $g=1$ est plus subtil. En effet, il existe des composantes où la restriction de l'application résiduelle n'est pas surjective. Rappelons, voir la section~\ref{sec:rapquad}, que les composantes connexes sont classifiées par le nombre de rotation $\rho$. On note $\quadomoduli[1]^{ \rho}(\mu)$ la composante connexe de nombre de rotation $\rho$ de la strate $\quadomoduli[1](\mu)$.
\begin{thm}\label{thm:geq1}
Soit $\mu$ une décomposition de $0$ avec au moins un élément inférieur ou égal à~$-2$.
\begin{itemize}
\item[i)] Si $\mu=(4a;\rec[-4][a])$ ou $\mu=(2a-1,2a+1;\rec[-4][a])$ pour $a\in\NN^{\ast}$, alors l'image de $\appresk[1][2](\mu)$ est égale à $\espresk[1][2](\mu)\setminus\left\{(0,\dots,0)\right\}$.
\item[ii)] Si $\mu=(2s;\rec[-2][s])$ ou $\mu=(s-1,s+1;\rec[-2][s])$ avec $s$ un entier pair non nul, alors l'image de $\appresk[1][2](\mu)$ est égale à $\espresk[1][2](\mu)\setminus \CC^{\ast}\cdot(1,\dots,1)$.
\item[iii)]  L'image de l'application résiduelle des composantes $\Omega^{2}\moduli[1]^{1}(6;-6)$, $\Omega^{2}\moduli[1]^{1}(3,3;-6)$ et $\Omega^{2}\moduli[1]^{3}(12;-6,-6)$ est le complémentaire de l'origine.
\item[iv)]Dans tout autre cas, l'application résiduelle $\appresquadrho[1](\mu)$ est surjective. 
\end{itemize} 
\end{thm}

\smallskip
\par
\subsection{Différentielles quadratiques sur la sphère de Riemann.}

Le cas des différentielles quadratiques sur la sphère de Riemann présente de nombreuses difficultés. En particulier, il existe des strates pour lesquelles le complémentaire de l'image de l'application résiduelle est de dimension complexe $1$ ou $2$.

Rappelons qu'étant donnée une décomposition $\mu$ de $-4$, l'espace $\quadomoduli[0](\mu)$ paramètre les différentielles quadratiques primitives de type $\mu$. On commence par remarquer (cf lemme~\ref{lem:puissk}) que ces strates sont non vides si et seulement si $\mu$ contient un nombre impair. De plus, le nombre de singularités d'ordres impairs est pair.

Le cas des strates paramétrant des différentielles quadratiques ayant au moins quatre singularités d'ordres impairs est donné par le théorème suivant.

\begin{thm}\label{thm:surjimp4}
L'application résiduelle d'une strate de genre zéro paramétrisant des différentielles quadratiques avec au moins quatre singularités d'ordres impairs est surjective.
\end{thm}

Le cas des strates qui possèdent deux singularités d'ordres impairs est plus subtil et les résultats suivants donnent la réponse complète suivant la décomposition considérée. 

Nous commençons par décrire le cas des strates où les deux singularités d'ordres impairs sont des pôles.
\begin{thm}\label{thm:g=0gen1ter}
Soit $\mu = (a_{1},\dots,a_{n};-b_{1},\dots,-b_{p};-c_{1},-c_{2};\rec[-2][s])$ une décomposition de $-4$ telle que $c_{1},c_{2} \geq 3$. L'application résiduelle de  $\quadomoduli[0](\mu)$ est surjective.
\end{thm}

Nous décrivons maintenant le cas où l'une des deux singularités d'ordres impairs est un pôle tandis que l'autre est un zéro.
\begin{thm}\label{thm:g=0gen1bis}
 Soit $\quadomoduli[0](a_{1},\dots,a_{n};-b_{1},\dots,-b_{p};-c;\rec[-2][s])$ une strate de genre zéro avec un unique zéro $a_{1}$ impair et $c \geq3$. L'image de l'application résiduelle de cette strate est:
 \begin{itemize}
 \item[i)] l'espace résiduel privé de l'origine $\espresquad[0](\mu)\setminus\left\{(0,\dots,0)\right\}$ si la somme des ordres des zéros d'ordres pairs est strictement inférieure à~$2p$,
 \item[ii)] l'espace résiduel $\espresquad[0](\mu)$ sinon.
\end{itemize}  
\end{thm}

Nous donnons maintenant la description de l'application résiduelle lorsque les deux singularités d'ordres impairs sont des zéros. 
Si les différentielles possèdent des pôles doubles et des pôles d'ordres inférieurs, on a la description suivante.
\begin{thm}\label{thm:r=0sneq0}
L'application résiduelle des strates $\quadomoduli[0](a_{1},\dots,a_{n};-b_{1},\dots,-b_{p};\rec[-2][s])$ avec $p\neq0$, $s\neq0$ et $\nim=2$ zéros impairs est surjective sauf dans les cas exceptionnels suivants :
\begin{enumerate}[i)]
\item L'image de $\appresk[0][2](2s'-1;2s'+1;-4;\rec[-2][2s'])$ avec $s'\geq1$ est $\espresk[0][2](\mu)\setminus\CC^{\ast}\cdot(0;1,\dots,1)$.

\item L'image de $\appresk[0][2](2a-1;2a+1;\rec[-4][a];-2,-2)$ avec $a\geq 1$ est $\espresk[0][2](\mu)\setminus\CC^{\ast}\cdot(0,\dots,0;1,1)$.

\item L'image de $\appresk[0][2](2s'+1;2s'+1;-4;\rec[-2][2s'+1])$ avec $s'\geq0$ est $\espresk[0][2](\mu)\setminus\CC^{\ast}\cdot(1;1,\dots,1)$.

\item L'image de $\appresk[0][2](2a-1,2a-1;\rec[-4][a];-2)$ avec $a \geq 1$ est $\espresk[0][2](\mu)\setminus\CC^{\ast}\cdot(1,0,\dots,0;1) $.
\end{enumerate}
\end{thm}

Dans le cas des différentielles dont tous les pôles sont d'ordres pairs inférieurs ou égaux à~$-4$, l'image de l'application résiduelle est donnée par le résultat suivant.
\begin{thm}\label{thm:r=0s=0}
Soit $\mu =(a_{1},\ldots,a_{n};-b_{1},\dots,-b_{p})$ une décomposition de $-4$ avec $\nim=2$ zéros d'ordres impairs. Notons $\tilde a_{\npa} = \sum_{i=3}^{n}a_{i}$ la somme des zéros d'ordres pairs.
\begin{enumerate}[i)]
 \item L'image de l'application résiduelle des strates $\Omega^{2}\moduli[0](2p+b-5,2p+b-5;-b,-b-2,\rec[-4][p-2])$ et $\Omega^{2}\moduli[0](2p+b-7,2p+b-5;-b,-b,\rec[-4][p-2])$ avec $p\geq2$ et $b\geq4$ pair est égale à $\espresquad[0](\mu)\setminus\CC \cdot(1,1,\rec[0][p-2])$.
 \item L'image de l'application résiduelle des strates $\quadomoduli[0](\mu)$ telles que $\tilde a_{\npa} < 2p$ (non considérées dans le point~i)) est le complémentaire de l'origine $\espresquad[0](\mu)\setminus\lbrace(0,\dots,0)\rbrace$.
 \item L'image de l'application résiduelle des strates avec $\tilde a_{\npa} \geq 2p$ est l'espace résiduel.
\end{enumerate}
\end{thm}

Enfin nous donnons la description de l'image par l'application résiduelle des strates dont tous les pôles sont d'ordre~$-2$. 
Dans ces strates, le concept suivant est important pour décrire l'application résiduelle.

\begin{defn}\label{def:triangulaire}
 Des nombres $R_{1},R_{2},R_{3}$ sont {\em triangulaires} s'il existe des racines carrées $r_{1},r_{2},r_{3}$ de ces nombres telles que $r_{1}+r_{2}+r_{3}=0$.
\end{defn} 

La description de l'image de l'application résiduelle des strates dont tous les pôles sont doubles et qui possède deux zéros d'ordres impairs est donnée dans le résultat suivant.
\begin{thm}\label{thm:geq0quad2}
L'application résiduelle de $\Omega^{2}\moduli[0](a_{1},\dots,a_{n};\rec[-2][s])$ avec $\nim=2$ zéros d'ordres impairs $a_{1},a_{2}$ est surjective sauf dans les quatre cas suivants.
\begin{enumerate}[i)]
\item L'image de l'application résiduelle des strates $\Omega^{2}\moduli[0](2s-1,2s+1;\rec[-2][2s'+2])$ ne contient pas les configurations de la forme $(R,\dots,R,R',R')$ pour tout $R,R'\in\CC^{\ast}$.

\item L'image de l'application résiduelle des strates $\Omega^{2}\moduli[0](2s'-1,2s'-1;\rec[-2][2s'+1])$ ne contient pas les configurations de la forme $(R_{1},R_{2},R_{3},\dots,R_{3})$ où les $R_{i}$ sont triangulaires.

\item L'image de l'application résiduelle des strates $\Omega^{2}\moduli[0](a_{1},\dots,a_{n};\rec[-2][s])$ ne contient pas les configurations de la forme $\lambda(r_{1}^{2},\dots,r_{s}^{2})$ où $\lambda \in \CC^{\ast}$, $s \geq 2$ et $r_{1},\dots,r_{n}$ entiers premiers entre eux avec $\sum r_{i}$ impair telle que
\begin{equation}\label{eq:sumimpintro}
 \sum r_{i} < \max(a_{1},a_{2})+2 \,.
\end{equation}
\item L'image de l'application résiduelle des strates $\Omega^{2}\moduli[0](a_{1},\dots,a_{n};\rec[-2][s])$ ne contient pas les configurations de la forme $\lambda(r_{1}^{2},\dots,r_{s}^{2})$ où $\lambda \in \CC^{\ast}$, $s \geq 2$ et $r_{1},\dots,r_{n}$ entiers premiers entre eux avec $\sum r_{i}$ pair telle que
\begin{equation}\label{eq:sumpairintro}
 \sum r_{i} < a_{1}+a_{2}+4 \,.
\end{equation}
\end{enumerate}
\end{thm}

La preuve du théorème~\ref{thm:geq0quad2} est obtenue en utilisant explicitement la géométrie plate associée à ces différentielles. Toutefois, cette dichotomie entre une borne en $\max(a_{1},a_{2})$ ou en $a_{1}+a_{2}$ selon si $\sum r_{i}$ est impair ou pair peut s'expliquer géométriquement de la façon suivante. La condition arithmétique sur les résidus implique que la différentielle quadratique est un revêtement ramifié de $\frac{\lambda }{z(z-1)}dz^{2}$. La somme $\sum r_{i}$ est le degré de ce revêtement. Les deux zéros d'ordre impair s'envoient impérativement sur les deux pôles simples de $\frac{\lambda }{z(z-1)}dz^{2}$.
Nous obtenons alors une inégalité reliant l'ordre des zéros d'ordre impair et le degré du revêtement. Lorsque le degré est pair, les zéros s'envoient sur le même pôle simple, la contrainte porte sur la somme de leurs degrés. Au contraire, lorsque le degré est impair, les zéros s'envoient sur des pôles différents et nous obtenons deux inégalités qui se combinent en une seule portant sur l'ordre maximum des deux zéros d'ordre impair.

\smallskip
\par
\subsection{Applications}
\label{sec:appliintro}
Nous donnons deux applications de nos résultats. Tout d'abord nous reprouvons que les strates de différentielles quadratiques d'aire finie sont non vides, à l'exception de quelques cas. Ensuite nous donnons la borne optimale du nombre de cylindres disjoints dans une différentielle quadratique d'une strate donnée.
\smallskip
\par
\paragraph{\bf Différentielles quadratiques d'aire finie.}
Le problème de savoir si les strates de différentielles quadratiques primitives $\quadomoduli(a_{1},\dots,a_{n})$ sont vides a été résolu par \cite{masm}. Le résultat est le suivant.
\begin{prop}\label{prop:stratesvides}
Soit $\mu = (a_{1},\dots,a_{n})$ une décomposition de $4g-4$ dont tous les éléments sont  supérieurs ou égaux à $-1$.
 La strate primitive $\quadomoduli(\mu)$ est vide si et seulement si
 \begin{enumerate}[i)]
  \item $g = 1$ et soit  $\mu=(-1,1)$ ou bien $\mu=\emptyset$;
  \item $g=2$ et soit $\mu=(4)$ ou bien $\mu=(1,3)$.
 \end{enumerate}
\end{prop}
 Nous donnons une nouvelle preuve de ce résultat dans la section~\ref{sec:plurifini}.

\smallskip
\par
\paragraph{\bf Cylindres.}
Naveh a montré dans \cite{Na} que le nombre maximal de cylindres disjoints dans une différentielle abélienne holomorphe de la strate $\omoduli(a_{1},\dots,a_{n})$  est $g+n-1$ et que cette borne est toujours atteinte. Nous généralisons ce résultat au cas des différentielles quadratiques. Rappelons que $\nim$ est le nombre de zéros impairs et pôles simples, et $\npa$ celui de zéros pairs.
\begin{prop}\label{prop:cylindresquad}
 Une différentielle de $\quadomoduli(a_{1},\dots,a_{n})$  possède au plus $g+\npa+\tfrac{\nim}{2}-1$ cylindres disjoints. De plus, cette borne est atteinte  dans chaque strate.
\end{prop}

\smallskip
\par
\subsection{Liens avec d'autres travaux}
\label{sec:travauxintro}

Cet article est le deuxième d'une série de trois articles consacrés à l'étude des obstructions globales à l'existence de différentielles dont les invariants locaux sont prescrits. Ces trois articles résultent du scindage de l'article \cite{geta} (non destiné à la publication) afin de pouvoir mieux expliquer les spécificités de chaque cas. Ce travail suit l'article \cite{getaab} relatif aux différentielles abéliennes et est suivi du troisième article \cite{getakdiff} traitant les cas des différentielles d'ordre supérieur (différentielles cubiques et au-delà). 
\smallskip
\par
Le travail \cite{SongSpher} associe aux différentielles de Strebel des métriques sphériques à singularités coniques avec une monodromie spéciale. Les angles coniques de ces métriques sont déterminés par les invariants locaux de la différentielle quadratique. Nos résultats nous permettent de caractériser  les distributions d'angles qui peuvent être réalisées par une telle métrique sphérique: voir~\cite{getasphere}.
\smallskip
\par
Le problème d'existence de différentielles quadratiques avec certains comportements est classique. En particulier, il est intéressant de fixer certains invariants de la géométrie plate associée. On pourra par exemple consulter \cite{DGTQuad} pour une avancée récente dans cette direction.
\smallskip
\par
\subsection{Organisation de cet article.}

Le périmètre des théorèmes ne correspond pas nécessairement aux sections de l'article. Celles-ci ont été organisées selon l'enchaînement des démonstrations plutôt qu'en vue de l'usage des résultats. Ainsi, chacune des sections de~\ref{sec:avecnondiv} jusqu'à~\ref{sec:ggeq1} couvre un certain type de strates.

Les preuves reposent sur la philosophie suivante. Dans un premier temps nous utilisons la correspondance entre les différentielles quadratiques et certaines classes de surfaces plates introduites par \cite{strebel}. Cette correspondance nous permet de construire explicitement des différentielles ayant les propriétés souhaitées lorsque le genre et le nombre de singularités sont petits. 

Dans un second temps, nous déduisons de ces résultats les autres cas grâce à deux opérations introduites par \cite{lanneauquad} : l'{\em éclatement d'une singularité} et la {\em couture d'anse}. La première de ces opérations permet d'augmenter le nombre de singularités sans changer le genre d'une différentielle. La seconde préserve le nombre de singularités mais augmente le genre de la surface sous-jacente.

Enfin, dans les cas où l'application résiduelle n'est pas surjective, nous développons des méthodes ad hoc afin de montrer la non-existence de différentielles ayant certains invariants locaux. 
\smallskip
\par
L'article s'organise comme suit. Dans la section~\ref{sec:bao} nous faisons les rappels nécessaires sur les représentations plates des différentielles quadratiques et sur les deux opérations précédemment citées. De plus, nous introduisons dans cette section les briques élémentaires qui nous permettrons de construire les différentielles avec les propriétés souhaitées. Enfin, nous introduisons dans la section~\ref{sec:arrhyp} la stratification résiduelle qui permet dans de nombreux cas de se ramener à des résidus réels positifs.

La section~\ref{sec:EXCEPT} établit de nombreuses obstructions qui sont spécifiques au cas des différentielles quadratiques et dont on ne retrouve pas l'équivalent chez les différentielles abéliennes ou d'ordre supérieur.

Les sections~\ref{sec:avecnondiv} jusqu'à~\ref{sec:juste-k} sont dédiées au cas des strates de genre zéro. La section~\ref{sec:avecnondiv} au cas des différentielles avec au moins un pôle d'ordre impair, montrant les théorèmes~\ref{thm:g=0gen1ter} et~\ref{thm:g=0gen1bis}. La section~\ref{sec:NTD} couvre le cas des différentielles dont les pôles sont pairs mais ne sont pas tous doubles, établissant les théorèmes~\ref{thm:r=0sneq0} et~\ref{thm:r=0s=0}. Enfin, la section~\ref{sec:juste-k} donne la preuve du théorème~\ref{thm:geq0quad2} traitant des différentielles dont tous les pôles sont doubles.
La preuve du théorème~\ref{thm:surjimp4} (que nous avons isolé en tant que résultat tant son énoncé est simple) est répartie dans les sections~\ref{sec:avecnondiv} jusqu'à~\ref{sec:juste-k} en fonction des types de singularités de la strate.

La section~\ref{sec:ggeq1} est dédiée aux cas des différentielles en genre supérieur ou égal à~$1$. Dans cette section, nous montrons les théorèmes~\ref{thm:ggeq2} et~\ref{thm:geq1}. Pour finir, nous donnons les preuves des applications énoncées dans l'introduction dans la section~\ref{sec:appli}.

\smallskip
\par
\subsection{Remerciements.} Nous remercions Corentin Boissy pour des discussions enrichissantes liées à cette série d'articles ainsi que les rapporteurs du texte qui ont amélioré la présentation de celui-ci. Le premier auteur remercie l'{\em Institut für algebraische Geometrie} de la {\em Leibniz Universität Hannover} et le {\em Centro de Ciencias Matemáticas} de la {\em Universidad Nacional Autonoma de México} où il a élaboré une partie de ce texte.

\section{Rappels et boîte à outils}
\label{sec:bao}

Dans cette section, nous introduisons les objets et les opérations de base pour nos constructions. Nous commençons par quelques rappels sur les différentielles quadratiques dans la section~\ref{sec:pluridiffbao}. Ensuite nous introduisons dans la section~\ref{sec:briques} les briques élémentaires de nos surfaces plates. Nous poursuivons par un rappel sur les différentielles entrelacées et les opérations de scindage de zéro et de couture d'anse dans la section~\ref{sec:pluridiffentre}. Enfin, nous discutons un cas spécial de surfaces plates dans la section~\ref{sec:coeur} et expliquons dans la section~\ref{sec:arrhyp} comment s'y ramener dans de nombreux cas.

\subsection{Différentielles quadratiques méromorphes}
\label{sec:pluridiffbao}

Soit $X$ une surface de Riemann compacte de genre~$g$ et $\omega$ une section méromorphe du fibré canonique $K_{X}$. On notera $Z$ les zéros et $P$ les pôles de $\omega$. L'intégration de $\omega$ sur $X\setminus P$ induit une structure plate sur $X\setminus P$. Chaque zéro de $\omega$ d'ordre $a$ correspond à une singularité conique d'angle $2(a+1)\pi$ de la structure plate. Les pôles simples de $\omega$ correspondent à des demi-cylindres infinis. Les pôles d'ordre $-b\leq -2$ correspondent à un revêtement de degré $b-1$ du plan dans lequel on a éventuellement fait une entaille correspondant au résidu.
Inversement, une surface plate obtenue en attachant par translation un nombre fini de demi-cylindres infinis, des revêtements d'ordre $b-1$ du plan entaillé et des polygones correspond à une différentielle abélienne méromorphe sur une surface de Riemann.

Une théorie similaire a été développée dans le cas des sections méromorphes $\xi$ de  $K_{X}^{2}$ (les détails se trouvent par exemple dans  \cite{BCGGM3} et \cite{chge}.). En effet, on peut passer au revêtement canonique $\pi\colon\whX\to X$ et choisir une racine carrée $\whomega$ de $\pi^{\ast}\xi$ sur $\whX$. L'intégration de $\whomega$ le long d'un chemin de $\whX$ nous fournit une structure plate sur $\whX$. La surface plate ainsi obtenue possède une symétrie cyclique d'ordre $2$ provenant de la structure de revêtement. Le quotient de cette surface par ce groupe est une surface plate où l'on autorise les identifications par des translations et rotations d'angle $\pi$.
On dit qu'une différentielle quadratique est  {\em primitive} si de manière équivalente son revêtement canonique est connexe, ou le groupe d'holonomie est non trivial.
Les pôles d'ordre~$-2$ correspondent à des demi-cylindres infinis et les pôles d'ordre $-b<-2$ à un revêtement d'ordre $b-2$ d'un domaine angulaire d'angle $\pi$. Les pôles simples et zéros d'ordres $a\geq-1$ correspondent aux singularités coniques d'angle $(a+2)\pi$.   

Si une différentielle quadratique $\xi$ est le carré d'une différentielle holomorphe $\omega$, on a pour tout pôle~$P$ 
\begin{equation}\label{eq:multiplires}
 \Resk[2]_{P}(\xi)=\left( \Res_{P}(\omega) \right)^{2}\,.
\end{equation}

En genre zéro, toutes les strates sont connexes. Il est facile de caractériser les strates qui ne contiennent que des différentielles quadratiques primitives. 

\begin{lem}\label{lem:puissk}
Soit $\mu=(m_{1},\dots,m_{t})$ un $t$-uplet de nombres pairs $m_{i}$ tel que $\sum m_{i}=-4$. Toutes les différentielles de type $\mu$ sont le carré d'une différentielle abélienne de $\Omega\moduli[0](\mu/2)$.
\end{lem}

\begin{proof}
Une différentielle $\xi$ sur $\PP^{1}$ de type $\mu$ est donnée par la formule 
$$\xi=\prod_{i=1}^{t}(z-z_{i})^{m_{i}}dz^{2}=\left( \prod_{i=1}^{t}(z-z_{i})^{m_{i}/2}dz\right)^{2}. $$
\end{proof}

\subsection{Briques élémentaires}
\label{sec:briques}

Dans ce paragraphe, nous introduisons des surfaces plates à bord qui nous serviront de briques pour construire les différentielles quadratiques ayant les propriétés locales souhaitées.
\smallskip
\par

Étant donnés des vecteurs $(v_{1},\dots,v_{l})$ dans $(\CC^{\ast})^{l}$, nous considérons la ligne brisée $L$ dans~$\CC$ donnée par la concaténation d'une demi-droite correspondant à $\RR_{-}$, des $v_{i}$ pour $i$ croissant et d'une demi-droite correspondant à $\RR_{+}$.
Nous supposerons que les $v_{i}$ sont tels que $L$ ne possède pas de point d'auto-intersection.  Nous donnerons une condition suffisante pour que cela soit possible dans le lemme~\ref{lem:noninter}.

Le {\em domaine basique positif (resp. négatif)} $D^{+}(v_{1},\dots,v_{l})$ (resp. $D^{-}(v_{1},\dots,v_{l})$) est l'adhérence de la composante connexe de $\CC\setminus L$ contenant les nombres complexes au-dessus (resp. en-dessous) de~$L$.
Étant donné un domaine positif $D^{+}(v_{1},\dots,v_{l})$ et un négatif $D^{-}(w_{1},\dots,w_{l'})$, on construit le {\em domaine basique ouvert à gauche (resp. droite)} $D_{g}(v_{1},\dots,v_{l};w_{1},\dots,w_{l'})$ (resp. $D_{d}(v_{1},\dots,v_{l};w_{1},\dots,w_{l'})$) en collant par translation les deux demi-droites correspondant à~$\RR_{+}$ (resp. $\RR_{-}$).

On se donne $b\geq4$ pair et $\tau\in\left\{1,\dots,\tfrac{b}{2}-1\right\}$.
Soient $(v_{1},\dots,v_{l};w_{1},\dots,w_{l'})$ des vecteurs de~$\CC^{\ast}$ tels que la partie réelle de leurs sommes est positive et que l'argument (pris dans $\left]-\pi,\pi\right]$) des $v_{i}$ est décroissant, des $w_{j}$ est croissant.
La partie polaire d'ordre $b$ et de type~$\tau$ associée à $(v_{1},\dots,v_{l};w_{1},\dots,w_{l'})$ est la surface plate à bord obtenue de la façon suivante. Prenons l'union disjointe de $\tau-1$ domaines basiques ouverts à gauche associés à la suite vide, $\tfrac{b}{2}-\tau-1$ domaines basiques ouverts à droite associés à la suite vide. Enfin prenons le domaine positif associé aux $v_{i}$ et le domaine négatif associé aux $w_{j}$. On colle alors par translation la demi-droite  inférieure du $i$-ième domaine polaire ouvert à gauche à la demi-droite supérieure du $(i+1)$-ième. La demi-droite inférieure du domaine $\tau-1$ est identifiée à la demi-droite de gauche du domaine positif. La demi-droite de gauche du domaine négatif est identifiée à la positive du premier domaine ouvert à gauche. On procède de même à droite. La figure~\ref{fig:ordreplusmoins} illustre cette construction.

\begin{figure}[htb]
\center
\begin{tikzpicture}[scale=1]

\begin{scope}[xshift=-6cm]
\fill[fill=black!10] (0,0)  circle (1cm);

\draw[] (0,0) coordinate (Q) -- (-1,0) coordinate[pos=.5](a);

\node[above] at (a) {$\tau$};
\node[below] at (a) {$1$};

\fill (Q)  circle (2pt);

\node at (1.5,0) {$\dots$};
\end{scope}

\begin{scope}[xshift=-3cm]
\fill[fill=black!10] (0,0)  circle (1cm);

\draw[] (0,0) coordinate (Q) -- (-1,0) coordinate[pos=.5](a);

\node[above] at (a) {$2$};
\node[below] at (a) {$3$};

\fill (Q)  circle (2pt);
\end{scope}
\begin{scope}[xshift=0cm]
\fill[fill=black!10] (0,0)  circle (1.5cm);
      \fill[color=white]
      (-.4,0.02) -- (.4,0.02) -- (.4,-0.02) -- (-.4,-0.02) --cycle;

\draw[] (-.4,0.02)  -- (.4,0.02) coordinate[pos=.5](a);
\draw[] (-.4,-0.02)  -- (.4,-0.02) coordinate[pos=.5](b);

\node[above] at (a) {$v$};
\node[below] at (b) {$v$};

\draw[] (-.4,0) coordinate (q1) -- (-1.5,0) coordinate[pos=.5](d);
\draw[] (.4,0) coordinate (q2) -- (1.5,0) coordinate[pos=.5](e);

\node[above] at (d) {$1$};
\node[below] at (d) {$2$};
\node[above] at (e) {$\tau+1$};
\node[below] at (e) {$\tau+2$};

\fill (q1) circle (2pt);
\fill[white] (q2) circle (2pt);
\draw (q2) circle (2pt);

\node at (4.5,0) {$\dots$};
\end{scope}

\begin{scope}[xshift=3cm]
\fill[fill=black!10] (0,0) coordinate (Q) circle (1cm);

\draw[] (0,0) coordinate (Q) -- (1,0) coordinate[pos=.5](a);

\node[above] at (a) {$\tau +2$};
\node[below] at (a) {$\tau+3$};

\fill[white] (Q)  circle (2pt);
\draw (Q)  circle (2pt);
\end{scope}
\begin{scope}[xshift=6cm]
\fill[fill=black!10] (0,0) coordinate (Q) circle (1cm);

\draw[] (0,0) coordinate (Q) -- (1,0) coordinate[pos=.5](a);

\node[above] at (a) {$b/2$};
\node[below] at (a) {$\tau+1$};

\fill[white] (Q)  circle (2pt);
\draw (Q)  circle (2pt);
\end{scope}

\end{tikzpicture}
\caption{Une partie polaire d'ordre $b$ de type $\tau$ associée à $(v;v)$. Les demi-droites dont les labels coïncident sont identifiés par translation} \label{fig:ordreplusmoins}
\end{figure}
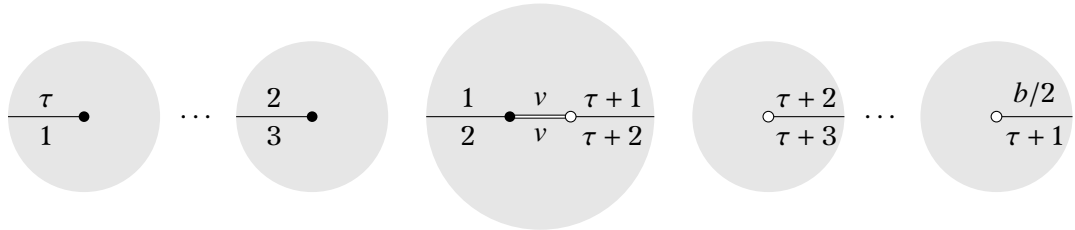

Si $\sum v_{i} =\sum w_{j}$ nous dirons que cette partie polaire est {\em triviale}. Dans le cas contraire, nous dirons que la partie polaire est {\em non triviale}. Le résidu quadratique du pôle d'ordre $-2b$ correspondant est donné par le carré de la somme $\sum v_{i}-\sum w_{j}$. Sur la figure~\ref{fig:partiesnontrivial}, le dessin de gauche illustre une partie polaire non triviale.

On se donne maintenant des vecteurs $(v_{1},\dots,v_{l})$ avec $l\geq1$ tels que la concaténation $V$ de ces vecteurs dans cet ordre n'a pas de points d'auto-intersection. De plus, on suppose qu'il existe deux demi-droites parallèles $L_{D}$ et $L_{F}$ de vecteur directeur $\overrightarrow{l}$, issues respectivement du point de départ $D$ et final $F$ de $V$, ne rencontrant pas $V$ et telles que $(\overrightarrow{DF},\overrightarrow{l})$ est une base positive de $\RR^{2}$. On définit la partie polaire $C(v_{1},\dots,v_{l})$ d'ordre $2$ associée aux $v_{i}$ comme le quotient du sous-ensemble de $\CC$ entre $V$ et les demi-droites $L_{D}$ et~$L_{F}$ par l'identification de $L_{D}$ à~$L_{F}$ par translation. Le résidu quadratique du pôle double correspondant est donné par le carré de la somme $F-D$ des~$v_{i}$. Une partie polaire d'ordre $2$ est donnée à droite de la figure~\ref{fig:partiesnontrivial}.

\begin{figure}[htb]
\center
\begin{tikzpicture}[scale=1.2]

\begin{scope}[xshift=-7cm]
\fill[fill=black!10] (0,0) coordinate (Q) circle (1.5cm);

\coordinate (a) at (-.5,0);
\coordinate (b) at (.5,0);
\coordinate (c) at (0,.2);

\fill (a)  circle (2pt);
\fill[] (b) circle (2pt);
    \fill[white] (a) -- (c)coordinate[pos=.5](f) -- (b)coordinate[pos=.5](g) -- ++(0,-2) --++(-1,0) -- cycle;
 \draw  (a) -- (c) coordinate () -- (b);
 \draw (a) -- ++(0,-1.1) coordinate (d)coordinate[pos=.5] (h);
 \draw (b) -- ++(0,-1.1) coordinate (e)coordinate[pos=.5] (i);
 \draw[dotted] (d) -- ++(0,-.3);
 \draw[dotted] (e) -- ++(0,-.3);
\node[below] at (f) {$v_{1}$};
\node[below] at (g) {$v_{2}$};
\node[left] at (h) {$1$};
\node[right] at (i) {$1$};

\draw (b)-- ++ (1,0)coordinate[pos=.6] (j);
\node[below] at (j) {$2$};
\node[above] at (j) {$3$};
    \end{scope}

\begin{scope}[xshift=-3.5cm]
\fill[fill=black!10] (0,0) coordinate (Q) circle (1.5cm);

\draw[] (0,0) -- (1.5,0) coordinate[pos=.5](a);

\node[above] at (a) {$2$};
\node[below] at (a) {$3$};
\fill[] (Q) circle (2pt);
\end{scope}

\begin{scope}[xshift=2.5cm,yshift=-1cm]
\coordinate (a) at (-1,0);
\coordinate (b) at (1,0);
\coordinate (c) at (0,.2);

    \fill[fill=black!10] (a) -- (c)coordinate[pos=.5](f) -- (b)coordinate[pos=.5](g) -- ++(0,1.5) --++(-2,0) -- cycle;
    \fill (a)  circle (2pt);
\fill[] (b) circle (2pt);
 \draw  (a) -- (c) coordinate () -- (b);
 \draw (a) -- ++(0,1.3) coordinate (d)coordinate[pos=.5](h);
 \draw (b) -- ++(0,1.3) coordinate (e)coordinate[pos=.5](i);
 \draw[dotted] (d) -- ++(0,.3);
 \draw[dotted] (e) -- ++(0,.3);
\node[below] at (f) {$v_{1}$};
\node[below] at (g) {$v_{2}$};
\node[left] at (h) {$3$};
\node[right] at (i) {$3$};

    \end{scope}
\end{tikzpicture}
\caption{Une partie polaire non triviale associée à $(v_{1},v_{2};\emptyset)$ d'ordre $6$ (de type $1$) à gauche et d'ordre $2$ à droite} \label{fig:partiesnontrivial}
\end{figure}
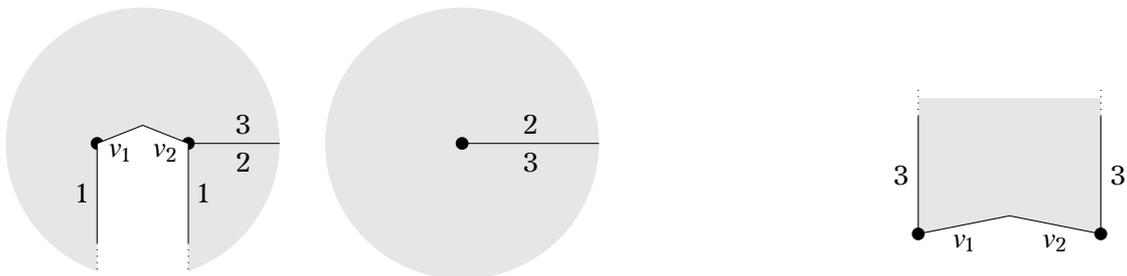

Nous traitons maintenant le cas des pôles d'ordres impairs. Soit $c=2\ell+1$ avec $\ell\geq 1$. On se donne des vecteurs $v_{i}$ de $\CC^{\ast}$ de partie réelle positive. La partie polaire d'ordre $c$ associée aux $(v_{1},\dots,v_{l};\emptyset)$ est donnée par la construction suivante. Nous considérons le domaine basique positif associé à $(v_{1},\dots,v_{l};\emptyset)$.
Ensuite nous prenons $\ell-1$ domaines basiques ouverts dans la direction de $L_{1}$ associés à la suite vide. Puis nous identifions les demi-droites cycliquement par translation, à l'exception de la dernière qui est identifiée par translation et rotation d'angle~$\pi$ à la demi-droite $L_{2}$. Cette construction est illustrée dans la figure~\ref{fig:partiepolairekdiff}. 
\begin{figure}[htb]
\center
 \begin{tikzpicture}

\begin{scope}[xshift=-2.5cm,yshift=-.5cm]
\fill[black!10] (-1,0)coordinate (a) -- (1.5,0)-- (a)+(2.5,0) arc (0:180:2)--(a)+(180:1.5) -- cycle;

   \draw (a)  -- node [below] {$v_{1}$} (0,0) coordinate (b);
 \draw (0,0) -- (1,0) coordinate[pos=.5] (c);
 \draw[dotted] (1,0) --coordinate (p1) (1.5,0);
 \fill (a)  circle (2pt);
\fill[] (b) circle (2pt);
\node[below] at (c) {$1$};

 \draw (a) -- node [above,rotate=180] {$2$} +(180:1) coordinate (d);
 \draw[dotted] (d) -- coordinate (p2) +(180:.5);
    \end{scope}

\begin{scope}[xshift=1cm,yshift=.5cm]
\fill[fill=black!10] (0,0) coordinate (Q) circle (1.2cm);

\draw[] (0,00) -- (1.2,0) coordinate[pos=.5](a);

\node[below] at (a) {$2$};
\node[above] at (a) {$1$};
\fill[] (Q) circle (2pt);
\end{scope}

\end{tikzpicture}
\caption{La partie polaire associée à $(v_{1};\emptyset)$ d'ordre~$3$} \label{fig:partiepolairekdiff}
\end{figure}
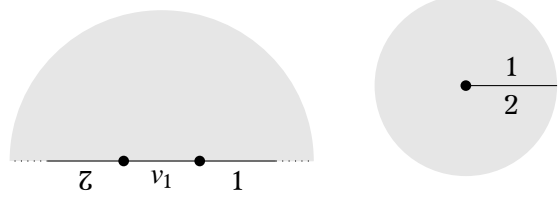
\par
Pour simplifier certaines constructions, il est utile de considérer la partie polaire d'ordre~$c$ associée à $(\emptyset;v_{1},\dots,v_{l})$. Elle est définie de manière similaire à la partie polaire précédente en considérant le domaine basique négatif associé à ces vecteurs.
\par
L'importance de ces constructions est résumée dans le lemme suivant.
\begin{lem}\label{lm:kresidu}
Soient $(v_{1},\dots,v_{l};w_{1},\dots,w_{l'})$ des nombres complexes, le pôle associé à la partie polaire d'ordre~$b$ pair et de type $\tau$ associée à $(v_{1},\dots,v_{l};w_{1},\dots,w_{l'})$ est d'ordre $-b$ et de résidu quadratique égal à $(\sum_{i=1}^{l} v_{i}-\sum_{j=1}^{l'} w_{j})^{2}$. 

Soit $(v_{1},\dots,v_{l})$ avec $l\geq1$, le pôle associé au domaine polaire d'ordre~$1$ associé à $v_{i}$ est d'ordre~$-2$ et possède un résidu égal à $(\sum_{i=1}^{l} v_{i})^{2}$.

Le pôle associé à un domaine polaire d'ordre $c\geq3$ impair est d'ordre $-c$.
\end{lem}
\smallskip
\par
Nous donnons maintenant une condition suffisante pour que la ligne brisée décrite aux paragraphes précédents soit sans point d'intersection.
\begin{lem}\label{lem:noninter}
Si les $v_{i}$ sont soit de partie réelle strictement positive, soit de partie réelle nulle et de partie imaginaire strictement positive, alors, quitte à permuter les $v_{i}$, la concaténation des $v_{i}$ avec les demi-droites horizontales $L_{1}$ et $L_{2}$ est sans point d'intersection.
\end{lem}
\begin{proof}
 Quitte à permuter les $v_{i}$, on peut supposer que les arguments des vecteurs~$v_{i}$, appartenant à l'ensemble $\left] -\tfrac{\pi}{2};\tfrac{\pi}{2}\right]$ sont décroissants. Notons que pour tout $\theta\in \left] \tfrac{\pi}{2}; \tfrac{3\pi}{2}\right]$ la demi-droite de pente $\theta$ partant du point initial de la concaténation n'a pas d'autres points d'intersection avec celle-ci. On a le même résultat pour les demi-droites d'angle $\phi$ partant du point final pour tout  $\phi\in \left] \tfrac{-\pi}{2}; \tfrac{\pi}{2}\right]$. Cela implique que l'on peut trouver des droites $L_{1}$ et~$L_{2}$ sans point d'intersection avec le reste de la construction qui forment n'importe quel angle strictement  compris entre $0$ et $2\pi$, en particulier avec les droites horizontales.
\end{proof}

\subsection{Différentielles quadratiques entrelacées, éclatement de zéros et couture d'anses.}
\label{sec:pluridiffentre}

Dans ce paragraphe, nous rappelons dans le cas quadratique certains cas particuliers des résultats de \cite{BCGGM3} sur les $k$-différentielles entrelacées. Cela nous permet de rappeler les constructions de l'{\em éclatement des zéros} et de la {\em couture d'anse} discutée en détail dans \cite{chge}. 

Tout d'abord, nous rappelons la définition d'une différentielle quadratique entrelacée.
\'Etant donnée une décomposition~$\mu:=(m_{1},\dots,m_{t})$ telle que $\sum_{i=1}^t m_i = 4g-4$, une {\em différentielle quadratique entrelacée $\eta$ de type~$\mu$}
sur une courbe stable $n$-marquée $(X,z_1,\ldots,z_t)$
est une collection de différentielles quadratiques non nulles~$\eta_v$ sur les composantes irréductibles~$X_v$ de~$X$ satisfaisant aux conditions suivantes:
\begin{itemize}
\item[(0)] {\bf (Annulation comme prescrit)} Chaque différentielle $\eta_v$ est méromorphe et le support de son diviseur est inclus dans l'ensemble des nœuds et des points marqués de $X_v$. De plus, si un point marqué $z_i$ se trouve sur~$X_v$, alors $\ord_{z_i} \eta_v=m_i$.
\item[(1)] {\bf (Ordres assortis)} Pour chaque nœud de $X$ qui identifie $q_1 \in X_{v_1}$ à $q_2 \in X_{v_2}$,
$$\ord_{q_1} \eta_{v_1}+\ord_{q_2} \eta_{v_2} = -4 \,. $$
\item[(2)] {\bf (Résidus assortis aux pôles doubles)} Si à un nœud de $X$
qui identifie $q_1 \in X_{v_1}$ avec $q_2 \in X_{v_2}$ on a $\ord_{q_1}\eta_{v_1}=
\ord_{q_2} \eta_{v_2}=-2$, alors
$$\Resk[2]_{q_1}\eta_{v_1} = \Resk[2]_{q_2}\eta_{v_2}\,.$$
\end{itemize}

Ce n'est que dans des cas très particuliers que nous aurons besoin de savoir quand une différentielle entrelacée est lissable. Nous rappelons ici uniquement les cas qui nous intéressent. Le premier cas est celui où les différentielles quadratiques ont un pôle double à tous les nœuds.
\begin{lem}\label{lem:lisspolessimples}
Soit $\eta=\left\{\eta_{v}\right\}$ une différentielle quadratique entrelacée. Si l'ordre des différentielles~$\eta_v$ aux nœuds est $-2$, alors $\eta$ est lissable sans modifier les invariants locaux aux points lisses.
\end{lem}
 
Maintenant nous regardons le cas des différentielles entrelacées à deux composantes.
\begin{lem}\label{lem:lissdeuxcomp}
 Supposons que $X$ possède exactement deux composantes $X_{1}$ et $X_{2}$ reliées par un unique nœud  qui identifie $q_1 \in X_{1}$ à $q_2 \in X_{2}$. Si $\ord_{q_1} \eta_{1}>-2>\ord_{q_2} \eta_{2}$, alors la différentielle quadratique entrelacée est lissable si et seulement si l'une des deux conditions suivantes est vérifiée.
 \begin{enumerate}[i)]
  \item $\Resk[2]_{q_2}\eta_{2}=0$
  \item $\eta_{1}$ n'est pas le carré d'une différentielle abélienne holomorphe.
 \end{enumerate}
 De plus, si i) est satisfait ou $\eta_{1}$ n'est pas le carré d'une différentielle abélienne méromorphe, alors il existe un lissage qui ne modifie pas les résidus de $\eta_{1}$.
\end{lem}
Remarquons que la deuxième partie du lemme n'est pas explicitement prouvée dans \cite{BCGGM3}. Toutefois, cela peut se montrer sans problèmes en combinant la preuve du théorème~1.5 et le lemme~4.4 de \cite{BCGGM3}.

Nous donnons maintenant deux applications cruciales du lemme~\ref{lem:lissdeuxcomp}.
\begin{prop}[Éclatement d'un zéro]\label{prop:eclatZero}
Soient $(X,\xi)$ une différentielle quadratique de type~$\mu$ et $z_{0}\in X$ un zéro d'ordre $a_{0}\geq -1$ de $\xi$. Soit $(\alpha_{1},\dots,\alpha_{t})$ un $t$-uplet d'entiers supérieurs où égaux à $-1$ tel que $\sum_{i}\alpha_{i}=a_{0}$. 

Il existe une opération sur $(X,\xi)$ en $z_{0}$ qui produit une  différentielle quadratique $(X',\xi')$ de type $(\alpha_{0},\dots,\alpha_{t},\mu\setminus\lbrace a_{0}\rbrace)$ et qui ne modifie pas les résidus de $\xi$ si et seulement si l'une des  deux conditions suivantes est vérifiée.
  \begin{enumerate}[i)]
  \item $\xi$ n'est pas le carré d'une différentielle abélienne (méromorphe).
  \item Il existe une différentielle quadratique de genre zéro et de type $(\alpha_{1},\dots,\alpha_{t};-a_{0}-4)$ dont le résidu au pôle d'ordre $-a_{0}-4$ est nul. 
 \end{enumerate} 
De plus, si $\xi=\omega^{2}$ avec $\omega$  une différentielle abélienne méromorphe, alors la différentielle~$\xi'$ est primitive si et seulement si au moins un $\alpha_{i}$ est impair. 
\end{prop}

\begin{proof}
 Partons d'une différentielle quadratique $(X,\xi)$. On forme une différentielle entrelacée en attachant au point~$z_{0}$ d'ordre $a_{0}$ une droite projective avec une différentielle de type $(\alpha_{1},\dots,\alpha_{t};-a_{0}-4)$. Le résultat est alors une conséquence directe du lemme~\ref{lem:lissdeuxcomp}.
\end{proof}

La seconde construction nous permettra en particulier de faire une récurrence sur le genre des surfaces de Riemann. 

\begin{prop}[Couture d'anse]\label{prop:attachanse}
 Soient $(X,\xi)$ une différentielle quadratique primitive dans la strate $\quadomoduli(\mu)$ et $z_{0}\in X$ un zéro d'ordre $a_{0}$ de $\xi$. 
 Il existe une opération locale à $z_{0}$ qui produit une différentielle $(X',\xi')$ dans la strate $\quadomoduli[g+1](a_{0}+2k,\mu\setminus\left\{a_{0}\right\})$. 
\end{prop}

\begin{proof}
 Partons d'une différentielle quadratique $(X,\xi)$. On forme une différentielle entrelacée en attachant au point $z_{0}$ une courbe elliptique avec une différentielle de type $(a_{0}+2k;-a_{0}-2k)$. Le lemme~\ref{lem:lissdeuxcomp} permet de conclure.
\end{proof}

\subsection{Différentielles à liens-selles horizontaux.}
\label{sec:coeur}

Pour certaines différentielles, beaucoup de problèmes géométriques se simplifient en des problèmes combinatoires. Cela nous permettra d'établir des obstructions en nous ramenant à ces objets.

Notons que l'application résiduelle est équivariante pour l'action de ${\rm GL}^{+}(2,\mathbb{R})$ sur les strates de différentielles quadratiques.
Le lemme~2.2 de \cite{tahar} établit le résultat suivant.
\begin{prop}\label{prop:coeurdege}
Dans une strate donnée de différentielles quadratiques d'aire infinie, chaque adhérence de ${\rm GL}^{+}(2,\mathbb{R})$-orbite contient une surface dont tous les liens-selles sont horizontaux.
\end{prop}

Dans une telle surface $S$ ayant $n$ zéros (et pôles simples) et $\tilde{p}$ pôles d'ordre au moins deux, les liens-selles sont tous disjoints car ils sont tous horizontaux. Le graphe des liens-selles découpe $S$ en $\tilde{p}$ domaines polaires. Chacun d'eux est un disque topologique avec un des $\tilde{p}$ pôles à l'intérieur (voir la section~4 de \cite{tahar} pour plus de détails). Un calcul de caractéristique d'Euler-Poincaré montre alors que les liens-selles du graphe sont nécessairement au nombre de $2g+n+\tilde{p}-2$. De plus, les angles entre les liens-selles sont des multiples entiers de $\pi$.
\par
En coupant le long de ces liens-selles, nous obtenons~$\tilde{p}$ domaines polaires. Le {\em graphe d'incidence} d'une telle surface est le graphe (plongé) dont les sommets sont les domaines polaires et deux sommets sont reliés par autant d'arêtes qu'il y a de liens-selles séparant ces deux domaines polaires.
De plus, le {\em graphe d'incidence simplifié} est obtenu en enlevant tous les sommets de valence $2$ au graphe d'incidence. Les sommets du graphe d'incidence qui sont de valence supérieure ou égale à $3$ sont dits {\em spéciaux}.

\subsection{Stratification de résonance et prolongement plat}\label{sec:arrhyp}

Fixons dans ce paragraphe une strate $S$ de la forme $\quadomoduli[0](a_{1},a_{2};-b_{1},\dots,-b_{p})$. Les longueurs des liens-selles d'une différentielle de $S$ sont proportionnelles à des normes de sommes partielles de racines de résidus quadratiques. En effet, la longueur d'un lien-selle fermé est égale à la norme d'une somme partielle et celle d'un lien-selle entre les deux zéros est égale à une moitié de somme totale de racines. La dégénérescence de ces liens-selles correspond donc à l'annulation de certaines de ces sommes.

Nous encodons ces conditions dans un arrangement d'hyperplans complexes appelé \textit{arrangement de résonance} dont la projection sur l'espace résiduel $\espresquad[0](a_{1},a_{2};-b_{1},\dots,-b_{p})$ définit la \textit{stratification de résonance}.
Dans $\CC^{p}$ un \textit{hyperplan de résonance} est l'ensemble des points $(r_{1},\dots,r_{p})$ dans la fermeture du lieu satisfaisant exactement une équation de la forme $\sum\limits_{j=1}^{p} w_{j}r_{j}=0$, avec $w_{j}\in \lbrace{ -1,0,1 \rbrace}$ non tous nuls. L'ensemble $H_{p}$ des hyperplans de résonance définit un arrangement d'hyperplans complexes dans $\CC^{p}$.

\begin{defn}\label{defn:stratres}
Soit $R=(R_{1},\dots,R_{p}) \in \espresquad[0](a_{1},a_{2};-b_{1},\dots,-b_{p})$ et $r=(r_{1},\dots,r_{p})$ des racines des $R_{i}$. L'ensemble des hyperplans de résonance contenant $r$ est $H_{p}(R) \subset H_{p}$ (et ne dépend que de $R$).

La \textit{stratification de résonance de $\espresquad[0](a_{1},a_{2};-b_{1},\dots,-b_{p})$} est la projection de l'union des hyperplans de résonance par l'application $(r_{1},\dots,r_{p}) \mapsto (r_{1}^{2},\dots,r_{p}^{2})$ et une \emph{strate} est l'ensemble des résidus $R$ qui possèdent le même $H_{p}(R)$.
\end{defn}

Dans chaque strate de résonance, nous définissons une systole résiduelle qui va nous permettre de contrôler les déformations de différentielles quadratiques.

\begin{defn}\label{defn:systole}
Soit $R=(R_{1},\dots,R_{p})$ une configuration de résidus quadratiques, la {\em systole résiduelle} de $R$ est
 \[ \sigma = \min \left\{ \left| \sum_{i\in I} r_{i} \right| : I\subset \lbrace1,\dots,p \rbrace, (r_{i})^{2}=R_{i} \text{ et } \sum_{i\in I} r_{i} \neq 0 \right\}\,.\]
\end{defn}

\begin{prop}\label{prop:systole}
La systole résiduelle varie continûment dans chaque strate de résonance.
\end{prop}

\begin{proof}
Il y a un nombre fini de sommes pondérées de racines et chacune d'elles varie continûment en fonction des résidus quadratiques. 
\end{proof}

Un prolongement plat permet d'établir l'existence de différentielles  par déformation.
\begin{cor}\label{cor:defplate}
S'il existe une différentielle de $\quadomoduli[0](a_{1},a_{2};-b_{1},\dots,-b_{p})$ dont les résidus appartiennent à une strate de résonance, alors toutes les configurations de cette strate de résonance sont réalisables par une différentielle de $\quadomoduli[0](a_{1},a_{2};-b_{1},\dots,-b_{p})$.
\end{cor}

\begin{proof}
Soit $\xi$ une différentielle quadratique dont les résidus sont $R_{1},\dots,R_{p}$. Chaque lien-selle fermé découpe la surface (de genre zéro) en deux composantes connexes dont une ne compte que des singularités d'ordre pair (en effet, la surface ne compte que deux singularités d'ordre impair et l'une d'entre-elles est incidente au lien-selle fermé). Cette composante est donc une surface de translation avec un lien-selle au bord. La longueur de ce lien-selle est donc la norme d'une somme de racines des résidus quadratiques des pôles contenus dans cette composante. Si au contraire le lien-selle relie les deux zéros, alors en découpant ce lien-selle on obtient une surface de translation avec deux bords dont les périodes sont identiques (sans quoi $\xi$ ne serait pas primitive). Il s'ensuit qu'il existe un choix de racines $r_{1},\dots,r_{p}$ de $R_{1},\dots,R_{p}$ tel que la longueur du lien-selle est la norme de $\frac{1}{2} \sum r_{i}$.

La longueur des liens-selles d'une différentielle de  $\quadomoduli[0](a_{1},a_{2};-b_{1},\dots,-b_{p})$ est minorée par la moitié de la systole résiduelle. Comme celle-ci varie continûment et que la dégénérescence d'une différentielle implique que la longueur d'un lien-selle tende vers $0$, on peut déformer $\xi$ dans un voisinage contenu dans la strate de résonance. Comme les strates de résonance sont connexes, on peut déformer $\xi$ en une différentielle quadratique ayant n'importe quelle configuration de résidus quadratiques de la même strate de résonance.
\end{proof}

\begin{rem}\label{rem:redux}
Les équations de résonance ont des coefficients réels. Il est donc facile de vérifier que chaque strate de résonance contient une configuration de résidus quadratiques qui sont tous des réels positifs. En conséquence, pour les strates telles que $g=0$, $r=0$ et $n=2$, la caractérisation des configurations réalisables peut s'établir en se restreignant au cas des configurations de résidus réels positifs. Pour de telles différentielles, tous les liens-selles sont alors horizontaux.
\end{rem}

\section{Obstructions exceptionnelles}\label{sec:EXCEPT}

La réalisation de configurations de résidus dans les strates de différentielles quadratiques présente une série d'obstructions spécifiques. Elles découlent de deux obstructions principales établies dans les propositions~\ref{prop:excep1} et~\ref{prop:excep2} par une série de chirurgies.

\begin{prop}\label{prop:excep1}
Pour tout $s\geq 2$ pair, l'application résiduelle de $\Omega^{2}\moduli[1](s+1,s-1;\rec[-2][s])$ ne contient pas $\CC^{\ast}\cdot(1,\dots,1)$.
\end{prop}

\begin{proof}
Nous commençons par remarquer que l'image de l'application résiduelle de la strate $\Omega^{2}\moduli[1](1,3;-2,-2)$ ne contient pas $(1,1)$. Sinon on pourrait former une différentielle entrelacée lissable en recollant les deux pôles doubles. La différentielle quadratique obtenue en la lissant serait dans $\Omega^{2}\moduli[2](1,3)$, qui est vide en vertu du théorème~2 de \cite{masm}.

Nous traitons le cas $s\geq4$ par l'absurde. Nous supposons qu'il existe une différentielle quadratique dans la strate $\Omega^{2}(s-1,s+1;\rec[-2][s])$ avec pour résidus quadratiques $(1,\dots,1)$. Nous notons $a_{1}$ le zéro d'ordre inférieur et $a_{2}$ celui d'ordre supérieur. Dans un premier temps nous simplifions la différentielle quadratique puis nous montrons que les différentielles quadratiques simplifiées n'existent pas.

Nous supposerons que le cœur de la différentielle est dégénéré et qu'il y a donc $s+2$ liens-selles horizontaux (voir section~\ref{sec:coeur}). En coupant le long des liens-selles la surface nous obtenons~$s$ parties polaires quadratiques bordées par au moins~$2$ segments. En effet, si un domaine polaire est formé par un unique segment, celui-ci est de longueur $1$ et constitue le seul segment de bord de l'autre domaine polaire incident (également un cylindre de circonférence $1$). Une telle surface est simplement un cylindre infini. Les domaines polaires dont le bord est composé par strictement plus de deux segments sont appelés {\em spéciaux}. Les valences des sommets correspondants dans le graphe d'incidence peuvent être de la forme
 \begin{equation}\label{eq:spe}
  (6),\, (5,3),\, (4,4),\, (4,3,3) \text{ ou } (3,3,3,3)\,.
 \end{equation}

Nous considérerons tout d'abord le graphe d'incidence simplifié défini dans la section~\ref{sec:coeur}. On utilisera le vocabulaire des graphes et des surfaces de manière interchangeable. Par exemple, la {\em valence} d'un pôle est la valence du sommet correspondant dans le graphe d'incidence.

Supposons que  les deux liens-selles au bord d'un domaine polaire de valence $2$ connectent le zéro~$a_{1}$ à $a_{2}$ ou le zéro $a_{2}$ à lui même (rappelons que $a_{2}\geq a_{1}$). Nous pouvons alors réaliser la chirurgie suivante: retirer ce domaine polaire et fusionner les deux domaines polaires incidents comme sur la figure~\ref{fig:fusion}.
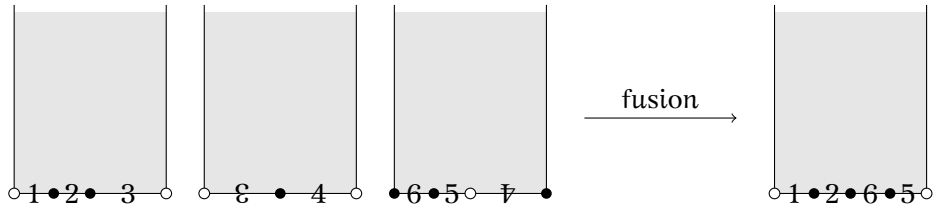
\begin{figure}[htb]
 \centering
\begin{tikzpicture}[scale=1]
\begin{scope}
\fill[fill=black!10] (0,0) coordinate (a1) -- (2,0) coordinate (a2) -- ++(0,2.4) -- ++(-2,0) -- cycle;

      \draw (a1) -- (a2) coordinate[pos=.26] (a3) coordinate[pos=.5] (a4) coordinate[pos=.13] (x1)coordinate[pos=.38] (x2) coordinate[pos=.75] (x3);

  \draw (a1)-- ++(0,2.5)coordinate[pos=.6](b);
    \draw (a2)-- ++(0,2.5)coordinate[pos=.6](c);
\node at (x1) {$1$};\node at (x2) {$2$};\node at (x3) {$3$};
\filldraw[fill=white](a1) circle (2pt);\filldraw[fill=white](a2) circle (2pt);\fill(a3) circle (2pt);\fill(a4) circle (2pt);
\end{scope}

\begin{scope}[xshift=2.5cm]
\fill[fill=black!10] (0,0) coordinate (a1) -- (2,0) coordinate (a2) -- ++(0,2.4) -- ++(-2,0) -- cycle;

      \draw (a1) -- (a2) coordinate[pos=.5] (a3)coordinate[pos=.25] (x1) coordinate[pos=.75] (x2);

  \draw (a1)-- ++(0,2.5)coordinate[pos=.6](b);
    \draw (a2)-- ++(0,2.5)coordinate[pos=.6](c);
\node[rotate=180] at (x1) {$3$};\node at (x2) {$4$};
\filldraw[fill=white] circle (2pt);\filldraw[fill=white] (a2) circle (2pt);\fill(a3) circle (2pt);
\end{scope}

\begin{scope}[xshift=5cm]
\fill[fill=black!10] (0,0) coordinate (a1) -- (2,0) coordinate (a2) -- ++(0,2.4) -- ++(-2,0) -- cycle;

      \draw (a1) -- (a2) coordinate[pos=.26] (a3) coordinate[pos=.5] (a4) coordinate[pos=.13] (x1)coordinate[pos=.38] (x2) coordinate[pos=.75] (x3);

  \draw (a1)-- ++(0,2.5)coordinate[pos=.6](b);
    \draw (a2)-- ++(0,2.5)coordinate[pos=.6](c);
\node at (x1) {$6$};\node at (x2) {$5$};\node[rotate=180] at (x3) {$4$};
\fill(a1) circle (2pt);\fill(a2) circle (2pt);\fill(a3) circle (2pt);\filldraw[fill=white](a4) circle (2pt);

\draw[->] (2.5,1) --node[above]{fusion} (4.5,1);
\end{scope}

\begin{scope}[xshift=10cm]
\fill[fill=black!10] (0,0) coordinate (a1) -- (2,0) coordinate (a2) -- ++(0,2.4) -- ++(-2,0) -- cycle;

      \draw (a1) -- (a2) coordinate[pos=.26] (a3) coordinate[pos=.5] (a4) coordinate[pos=.76] (a5)   coordinate[pos=.13] (x1)coordinate[pos=.38] (x2) coordinate[pos=.63] (x3)coordinate[pos=.88] (x4);

  \draw (a1)-- ++(0,2.5)coordinate[pos=.6](b);
    \draw (a2)-- ++(0,2.5)coordinate[pos=.6](c);
\node at (x1) {$1$};\node at (x2) {$2$};\node at (x3) {$6$};\node at (x4) {$5$};
\filldraw[fill=white](a1) circle (2pt);\filldraw[fill=white](a2) circle (2pt);\fill(a3) circle (2pt);\fill(a4) circle (2pt);\fill(a5) circle (2pt);
\end{scope}
\end{tikzpicture}
\caption{La fusion de deux pôles incidents à un pôle bordé par deux liens-selles connectant les deux zéros.}
\label{fig:fusion}
\end{figure}
La surface obtenue est toujours primitive. La chirurgie enlève deux pôles doubles et diminue l'ordre des deux zéros de $2$ ou $a_{2}$ de~$4$. La différentielle obtenue est  dans une strate du type $\Omega^{2}\moduli[1](s+1,s-1;\rec[-2][s])$. Nous supposerons désormais que tous les sommets de valence $2$ correspondent à des domaines polaires dont le bord ne contient que le zéro d'ordre inférieur~$a_{1}$.
\smallskip
\par
Si $s \geq 6$, comme il y a au plus quatre domaines polaires spéciaux (voir l'équation~\eqref{eq:spe}), il y a au moins deux sommets de valence $2$. Par le paragraphe antérieur, les liens-selles bordant ces domaines polaires connectent le zéro d'ordre inférieur $a_{1}$ à lui-même. Chacun de ces domaines polaires contribue donc à $a_{1}$ d'un angle $2\pi$. La contribution angulaire des domaines spéciaux aux deux zéros est au plus $12\pi$. Comme l'angle de la singularité associée à~$a_{2}$ est égal à celui associé à $a_{1}$ plus $2\pi$, il s'ensuit que la différence entre la contribution des domaines spéciaux à $a_{2}$ et $a_{1}$ est au moins $6\pi$. On en déduit que la contribution des domaines spéciaux au zéro d'ordre $a_{1}$ est au plus $3\pi$. Remarquons que les chaînes de domaines de valence~$2$ se rattachent aux domaines spéciaux. Donc, soit la contribution des domaines spéciaux au zéro d'ordre~$a_{1}$ est au moins $4\pi$, soit il y a une unique chaîne dont les deux extrémités sont incidentes au même sommet spécial $A$ via des liens-selles adjacents. Dans ce cas, la contribution de ce domaine polaire à $a_{1}$ est~$3\pi$. Ceci n'est possible que si le graphe d'incidence a exactement quatre sommets spéciaux de valence~$3$. Le sommet $A$ est alors relié à un autre sommet spécial par un lien-selle reliant $a_{1}$ à lui-même. Par conséquent, la contribution des domaines spéciaux à ce zéro est strictement plus grande que $3\pi$, ce qui est absurde.
\smallskip
\par
Il reste à traiter le cas de la strate $\Omega^{2}\moduli[1](3,5;\rec[-2][4])$.  S'il y a strictement moins de $4$ domaines polaires spéciaux, le même argument que précédemment permet de conclure. Nous supposerons donc que les quatre domaines polaires sont de valence $3$. Comme la différentielle est primitive certains cycles dans le graphe d'incidence ont un nombre impair de sommets. Un tel cycle n'est pas une boucle à un sommet. En effet, cela impliquerait que dans le domaine polaire correspondant, deux segments de bords sont identifiés et donc que nous avons un zéro d'ordre $-1$. Ces conditions impliquent que le seul graphe d'incidence possible est le graphe complet $K_{4}$. Dans chaque domaine polaire, deux angles $\pi$ contribuent à un zéro tandis que le troisième angle $\pi$ contribue à l'autre (sinon nous aurions un zéro d'ordre au moins $7$). Ainsi, au moins deux angles d'un domaine polaire contribuent au zéro d'ordre~$3$. Il s'ensuit qu'il existe un lien-selle reliant ce zéro à lui-même. Ainsi les deux domaines polaires adjacents contribuent deux fois à ce zéro, ce qui est absurde.
\end{proof}

Partant de la proposition~\ref{prop:excep1}, toute une série d'obstructions inédites se déduisent les unes des autres.

\begin{cor}\label{cor:except1}
Dans les strates de différentielles quadratiques listées ci-dessous, les configurations de résidus quadratiques mentionnées ne sont pas réalisables:
\begin{enumerate}
    \item pour $s \geq 2$ pair, les configurations de $\CC^{\ast}\cdot(1,\dots,1)$ dans $\Omega^{2}\moduli[1](2s;\rec[-2][s])$;
    \item pour $p \geq 1$, la configuration $(0,\dots,0)$ dans $\Omega^{2}\moduli[1](2p+1,2p-1;\rec[-4][p])$;
    \item pour $p \geq 1$, la configuration $(0,\dots,0)$ dans $\Omega^{2}\moduli[1](4p;\rec[-4][p])$;
    \item pour $s \geq 4$ pair, les configurations de la forme $(R,\dots,R,R',R')$ avec $R,R' \in \mathbb{C}^{\ast}$ dans $\Omega^{2}\moduli[0](s-1,s-3;\rec[-2][s])$;
    \item pour $s \geq 3$ impair, les configurations de la forme $(R_{1},R_{2},R_{3},\dots,R_{3})$ où les $R_{i}$ sont triangulaires (voir définition~\ref{def:triangulaire}) dans $\Omega^{2}\moduli[0](s-2,s-2;\rec[-2][s])$;
    \item pour $s \geq 2$ pair, les configurations de $\CC^{\ast}\cdot(0;1,\dots,1)$ dans $\Omega^{2}\moduli[0](s+1,s-1;-4;\rec[-2][s])$;
    \item pour $s \geq 1$ impair, les configurations de $\CC^{\ast}\cdot(1;1,\dots,1)$ dans $\Omega^{2}\moduli[0](s,s;-4;\rec[-2][s])$.
\end{enumerate}
\end{cor}

\begin{proof}
Pour établir l'obstruction (1), nous supposons qu'il existe une différentielle dans $\Omega^{2}\moduli[1](2s;\rec[-2][s])$ dont les résidus sont égaux à $(1,\dots,1)$. La proposition~\ref{prop:eclatZero} permet l'éclatement de
son zéro en deux zéros d'ordres~$s+1$ et $s-1$ sans changer les résidus. Ceci contredit la proposition~\ref{prop:excep1}.

Pour l'obstruction (2), nous supposons qu'il existe $\xi$ dans $\Omega^{2}\moduli[1](2p+1,2p-1;\rec[-4][p])$ dont les résidus sont nuls. Grâce à la proposition~\ref{prop:coeurdege}, nous pouvons supposer que tous les liens-selles de $\xi$ sont horizontaux. Comme les résidus sont nuls, chaque domaine polaire est un plan infini coupé d'une cicatrice horizontale de longueur $\ell$ finie. Nous faisons maintenant l'opération suivante illustrée par la figure~\ref{fig:metamorphose}. Pour chaque domaine polaire, nous considérons un rectangle de hauteur $1$ et de longueur $\ell$. Les bords horizontaux du rectangle correspondent aux segments de la cicatrice. Les deux bords verticaux sont collés à des cylindres de hauteur $1$.
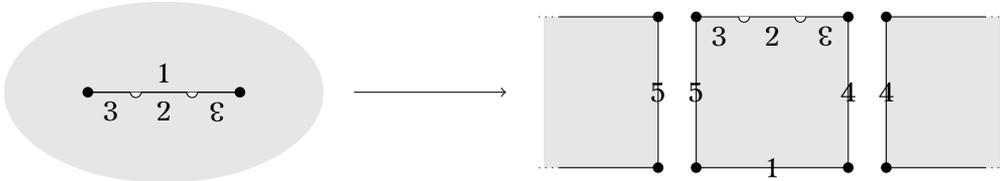
\begin{figure}[htb]
 \center
\begin{tikzpicture}[scale=1]
    \begin{scope}[xshift=-7cm,yshift=1cm]
     \fill[fill=black!10] (0,0) ellipse (2.1cm and 1.2cm);

\draw[] (-1,0)coordinate (Q) -- (1,0) coordinate (P) coordinate[pos=.5](c) coordinate[pos=.15](d) coordinate[pos=.85](e) coordinate[pos=.35](R) coordinate[pos=.65](S);
\draw[] (Q) -- (P);

\filldraw[fill=white](R)  arc (0:-180:2pt);
\filldraw[fill=white]  (S)  arc (-180:0:2pt);

\fill (P) circle (2pt);
\fill(Q) circle (2pt);

\node[above] at (c) {$1$};
\node[below] at (c) {$2$};
\node[below] at (d) {$3$};
\node[above, rotate=180] at (e) {$3$};

\draw[->] (2.5,0) --(4.5,0);
\end{scope}

\begin{scope}[]
\coordinate (a) at (0,0);
\coordinate (b) at (2,0);
\coordinate (c) at (2,2);
\coordinate (d) at (0,2);
\filldraw[fill=black!10] (a) -- (b)coordinate[pos=.5](h) -- (c) coordinate[pos=.5](e) -- (d) coordinate[pos=.5](f) coordinate[pos=.15](D) coordinate[pos=.85](E) coordinate[pos=.35](R) coordinate[pos=.65](S) -- (a) coordinate[pos=.5](g);
\fill (a) circle (2pt);\fill (b) circle (2pt);\fill (c) circle (2pt);\fill (d) circle (2pt);
\node[] at (h) {$1$};
\node[below] at (f) {$2$};
\node[below] at (E) {$3$};
\node[above, rotate=180] at (D) {$3$};
\node[] at (e) {$4$};
\node[] at (g) {$5$};

\filldraw[fill=white](S)  arc (0:-180:2pt);
\filldraw[fill=white] (R)  arc (-180:0:2pt);

\end{scope}

\begin{scope}[xshift=2.5cm]
\coordinate (a) at (0,0);
\coordinate (b) at (0,2);
    \fill[fill=black!10] (a) --  (b)coordinate[pos=.5](g) -- ++(1.5,0) --++(0,-2) -- cycle;
    \fill (a)  circle (2pt);
\fill[] (b) circle (2pt);
 \draw  (a)  -- (b);
 \draw (a) -- ++(1.3,0) coordinate (d)coordinate[pos=.5](h);
 \draw (b) -- ++(1.3,0) coordinate (e)coordinate[pos=.5](i);
 \draw[dotted] (d) -- ++(.3,0);
 \draw[dotted] (e) -- ++(.3,0);
\node[] at (g) {$4$};
    \end{scope}

    \begin{scope}[xshift=-.5cm]
\coordinate (a) at (0,0);
\coordinate (b) at (0,2);
    \fill[fill=black!10] (a) --  (b)coordinate[pos=.5](g) -- ++(-1.5,0) --++(0,-2) -- cycle;
    \fill (a)  circle (2pt);
\fill[] (b) circle (2pt);
 \draw  (a)  -- (b);
 \draw (a) -- ++(-1.3,0) coordinate (d)coordinate[pos=.5](h);
 \draw (b) -- ++(-1.3,0) coordinate (e)coordinate[pos=.5](i);
 \draw[dotted] (d) -- ++(-.3,0);
 \draw[dotted] (e) -- ++(-.3,0);
\node[] at (g) {$5$};
    \end{scope}

\end{tikzpicture}
\caption{Une chirurgie remplacant un pôle d'ordre $-4$ de résidu $0$ en deux pôles doubles de résidus égaux sans changer les ordres des zéros.} \label{fig:metamorphose}
\end{figure}
Cette chirurgie produit une différentielle de $\Omega^{2}\moduli[1](2p+1,2p-1;\rec[-2][2p])$ réalisant la configuration de résidus quadratiques $(-1,\dots,-1)$, contredisant la proposition~\ref{prop:excep1}.

L'obstruction (3) se déduit de l'obstruction (2) par éclatement du zéro.

Pour l'obstruction (4), s'il existe une différentielle de $\Omega^{2}\moduli[0](s-1,s-3;\rec[-2][s])$ de résidus de la forme $(R,\dots,R,R',R')$, alors on obtient une différentielle de $\Omega^{2}\moduli[1](s-1,s-3;\rec[-2][s-2])$ dont la configuration de résidus est $(R,\dots,R)$ en collant les pôles de résidus $R'$ ensemble. Cela contredit la proposition~\ref{prop:excep1}.

Pour l'obstruction (5), nous considérons une différentielle de $\Omega^{2}\moduli[0](s-2,s-2;\rec[-2][s])$ réalisant une configuration $(R_{1},R_{2},R_{3},\dots,R_{3})$ de résidus triangulaires. Sans perte de généralité (voir la section~\ref{sec:arrhyp}), nous pouvons supposer que les trois résidus $R_{1},R_{2},R_{3}$ n'appartiennent pas aux mêmes rayons issus de l'origine. 
Nous remplaçons alors le pôle de résidu $R_{1}$ par un triangle dont les trois bords appartiennent aux directions des racines des $R_{1},R_{2},R_{3}$. Puis nous collons sur les deux bords restants du triangle deux cylindres correspondant à des pôles doubles dont les résidus sont $R_{2}$ et $R_{3}$. L'ordre de l'un des zéros augmente de~$2$ et nous obtenons une différentielle de $\Omega^{2}\moduli[0](s,s-2;\rec[-2][s+1])$ réalisant la configuration $(R_{2},R_{2},R_{3},\dots,R_{3})$, ce qui contredit l'obstruction (4).

Pour l'obstruction (6), considérons une différentielle de $\Omega^{2}\moduli[0](s+1,s-1;-4,\rec[-2][s])$ dont le résidu au pôle quadruple est nul tandis que les résidus aux pôles doubles sont égaux entre eux. Nous obtenons une  différentielle quadratique entrelacée en attachant au pôle quadruple le carré d'une différentielle abélienne sur un tore. Le lemme~\ref{lem:lissdeuxcomp} permet de lisser cette
différentielle entrelacée et d'obtenir une différentielle de 
$\Omega^{2}\moduli[1](s+1,s-1;\rec[-2][s])$ dont les résidus aux pôles sont tous égaux, contredisant la proposition~\ref{prop:excep1}.

Enfin, pour l'obstruction (7), nous partons d'une différentielle de $\Omega^{2}\moduli[0](s,s;-4;\rec[-2][s])$ dont tous les résidus sont~$1$. On coupe alors une cicatrice dans le domaine du pôle quadruple pour insérer un cylindre de circonférence $1$. Une telle chirurgie va ajouter un pôle double de résidu $1$, l'ordre d'un zéro va augmenter de $2$ tandis que le résidu du pôle quadruple va devenir $0$. La surface obtenue contredit alors l'obstruction (6).
\end{proof}

Nous établissons maintenant une seconde famille d'obstructions exceptionnelles. 
\begin{prop}\label{prop:excep2}
Pour $a \geq 0$ et $b \geq 2$ pair, les configurations de résidus $\CC^{\ast}\cdot(1,1;0,\dots,0)$ ne sont pas réalisables dans la strate $\Omega^{2}\moduli[0](2a+b-1,2a+b-3;-b,-b,\rec[-4][a])$.
\end{prop}

\begin{proof}
Supposons par l'absurde qu'il existe une telle différentielle $\xi$. Tous les liens-selles de la surface correspondante sont horizontaux et de longueur entière (voir section~\ref{sec:arrhyp}). Plus précisément, les liens-selles reliant les deux zéros sont de longueur $1$ tandis que les liens-selles fermés sont de longueur $1$ ou $2$ selon le nombre de pôles à résidus non nuls qu'ils entourent.
Nous allons établir la non existence de $\xi$ en analysant les formes possibles du graphe d'incidence (voir section~\ref{sec:coeur}). Celui-ci admet $a+2$ sommets et arêtes. Seuls les domaines des deux pôles ayant un résidu non nul peuvent être de valence $1$. De plus, le graphe comporte un unique cycle.

Une chirurgie permet de simplifier un éventuel contre-exemple. Supposons que les liens-selles au bord d'un sommet de valence~$2$ correspondant à un pôle sans résidu relient soit les deux zéros, soit le zéro d'ordre le plus grand à lui-même. Nous retirons alors ce domaine polaire et recollons les deux bords (forcément de même longueur) l'un avec l'autre. Dans tous les cas, nous obtenons une différentielle d'une strate comme dans la proposition. 

Les liens-selles du cycle relient les deux zéros (car dans ce graphe plongé, les deux faces correspondent aux deux singularités coniques) et sont donc de longueur $1$. Il s'ensuit que les domaines de pôles ayant un résidu non nul ne peuvent être des sommets de valence $2$ dans le cycle. En particulier, le graphe n'est pas réduit au cycle et compte au moins un sommet de valence $1$.
\smallskip
\par
Supposons que le graphe a exactement deux sommets de valence $1$. Ce sont les domaines des pôles d'ordre $-b$. Si leurs bords contiennent des zéros différents, alors nous pouvons directement remplacer ces domaines polaires par des cylindres de circonférence $1$, obtenant des différentielles pour lesquelles $b=2$ tandis que les deux zéros seront d'ordres $2a+1$ et $2a-1$. En collant ces deux cylindres l'un sur l'autre, nous obtenons une différentielle de la strate $\Omega^{2}\moduli[1](2a+1,2a-1;\rec[-4][a])$ dont les résidus sont nuls. Ceci est interdit par l'obstruction~(2) du corollaire~\ref{cor:except1}.

Si au contraire, les deux domaines polaires d'ordre $-b$ ont le même zéro $z_{1}$ dans leur bord, alors les seuls angles contribuant à l'autre zéro $z_{2}$ appartiennent aux domaines polaires de l'unique cycle du graphe d'incidence. En effet, on peut voir le graphe d'incidence comme un graphe plongé découpant la sphère en deux composantes connexes correspondant aux deux zéros. Si le cycle compte un sommet de valence $4$, alors la chirurgie simplificatrice permet de supposer que c'est le seul sommet du cycle. Au bord de ce sommet, deux liens-selles adjacents sont identifiés. Il s'ensuit que $z_{2}$ est un zéro d'ordre $-1$, ce qui est impossible car $a \geq 1$ puisque le sommet de valence $4$ est l'un des pôles quadruples. Si au contraire, le cycle compte deux sommets de valence $3$ (et aucun autre sommet du fait de la chirurgie simplificatrice), alors $z_{2}$ sera au mieux un zéro d'ordre $1$, ce qui est incompatible avec $a \geq 2$.
\smallskip
\par
Dans le dernier cas, le graphe a un unique sommet de valence $1$ et donc un unique sommet de valence $3$. Supposons d'abord que ce sommet de valence $3$ est l'un des domaines polaires à résidu non nul. Dans ce cas, la chirurgie simplificatrice permet de supposer que c'est le seul sommet du cycle. Deux segments adjacents du bord de ce domaine sont identifiés, définissant un zéro d'ordre impair au plus $b-5$. L'autre zéro est l'unique singularité au bord du domaine polaire correspondant au sommet de valence $1$. Ce dernier est donc d'ordre au moins $b-1$ (un angle $(b-1)\pi$ pour le sommet de valence $1$ et $4\pi$ pour le sommet de valence~$3$). Ceci est impossible.

Supposons enfin que le sommet de valence $3$ correspond à un pôle quadruple à résidu nul. Dans ce cas, la chirurgie simplificatrice permet de se ramener au cas dans lequel c'est le seul sommet du cycle (car, nous l'avons vu, les sommets du cycle de valence $2$ correspondent à des domaines polaires à résidus nuls). Deux segments adjacents du bord de ce domaine sont identifiés, définissant un zéro d'ordre $-1$, ce qui contredit $a \geq 1$. Ainsi, la configuration $(1,1;0,\dots,0)$ n'est jamais réalisable dans $\Omega^{2}\moduli[0](2a+b-1,2a+b-3;-b,-b,\rec[-4][a])$.
\end{proof}

Une nouvelle famille d'obstructions se déduit de la proposition précédente.

\begin{cor}\label{cor:excep2}
Pour $a \geq 0$ et $b \geq 2$ pair, les configurations de résidus $\CC^{\ast}\cdot(1,1;0,\dots,0)$ ne sont pas réalisables dans la strate $\Omega^{2}\moduli[0](2a+b-1,2a+b-1;-b,-b-2,\rec[-4][a])$.
\end{cor}

\begin{proof}
Partant d'une différentielle d'une telle strate réalisant une telle configuration de résidus quadratiques, il suffit de couper, dans le domaine du pôle d'ordre $-b$, une cicatrice le long d'une demi-droite partant de l'une des singularités coniques. Insérant un plan infini dans cette cicatrice, le domaine polaire devient d'ordre $-b-2$ tandis que l'ordre du zéro en question devient $2a+b+1$. Une telle surface contredit alors la proposition~\ref{prop:excep2}.
\end{proof}

\section{Différentielles avec un pôle impair}
\label{sec:avecnondiv}

Dans cette section nous considérons le cas des différentielles de genre zéro où la décomposition possède au moins un pôle d'ordre impair. Nous traitons dans la section~\ref{sec:poleimpreszero} le cas de l'origine et dans la section~\ref{sec:poleimpresnozero} son complémentaire.

\subsection{Le cas du résidu nul}
\label{sec:poleimpreszero}

Lorsque la décomposition compte un unique pôle d'ordre impair et un unique zéro d'ordre impair, il y a une condition nécessaire non triviale pour que l'origine soit dans l'image de l'application résiduelle de la strate.

\begin{lem}
Considérons la décomposition $\mu=(a_{1},2l_{2},\dots,2l_{n};-b_{1},\dots,-b_{p};-c)$. Si l'inégalité $\sum_{i=2}^{n} l_{i}<p$ est satisfaite, alors l'application résiduelle $\appresk[0][2](\mu)$ ne contient pas $\lbrace (0,\dots,0) \rbrace$.
\end{lem}

\begin{proof}
\'Etant donnée une différentielle de la strate $\quadomoduli[0](\mu)$ dont tous les résidus sont nuls, le revêtement canonique définit une différentielle abélienne méromorphe $(\widehat X, \widehat\omega)$. D'après les résultats de la section~2.1 de \cite{BCGGM3}, le revêtement est uniquement ramifié au zéro d'ordre $a_{1}$ et au pôle d'ordre~$c$. Donc la surface $\widehat X$ est une sphère de Riemann. De plus la $1$-forme $\widehat \omega$ possède deux zéros d'ordres $l_{i}$ pour chaque $i\geq 2$, un zéro d'ordre $a_{1}+1$, deux pôles d'ordres $-b_{j}/2$ pour chaque $j\geq1$ et un pôle d'ordre $-c+1$. Notons de plus que le résidu à chaque pôle est nul. D'après le point i) du théorème~1.2 de \cite{getaab}, cela est possible si et seulement si l'ordre de tous les zéros est inférieur ou égal à $$(c-1)+2\sum_{j=1}^{p} \frac{b_{j}}{2} - (2p+2)\,.$$ 
Cette condition implique que 
$$ a_{1} \leq c + \sum b_{i} -4 -2p \,.$$
Le fait que la différence entre la somme des ordres des zéros et la somme des ordres des pôles est égale à~$-4$ implique le résultat. 
\end{proof}

Nous montrons maintenant dans quels cas l'origine est dans l'image de l'application résiduelle. Nous traitons tout d'abord le cas des strates avec un unique zéro.
\begin{lem}\label{lem:g=0gen1nouvbis}
Soit $\quadomoduli[0](a;-b_{1},\dots,-b_{p};-c_{1},\dots,-c_{r})$ une strate de genre zéro avec $b_{i}$ pairs et $c_{j}$ impairs. Si
$r\geq2$, alors l'image de l'application résiduelle contient $(0,\dots,0)$.
\end{lem}

\begin{proof}
Soient deux pôles $P_{1}$ et $P_{2}$ d'ordres respectifs $-c_{1}$ et $-c_{2}$ non pairs. Pour chaque pôle d'ordre $-b_{i}$ pair nous prenons une partie polaire d'ordre $b_{i}$ associée à $(1;1)$. Pour les pôles d'ordres impairs distincts de $P_{1}$, nous prenons la partie polaire associée à $(1;\emptyset)$. Pour $P_{1}$ nous prenons la partie polaire d'ordre $c_{1}$ associée à $(\emptyset;\rec[1][r])$. Cela est illustré par la figure~\ref{ex:8,8,12}.
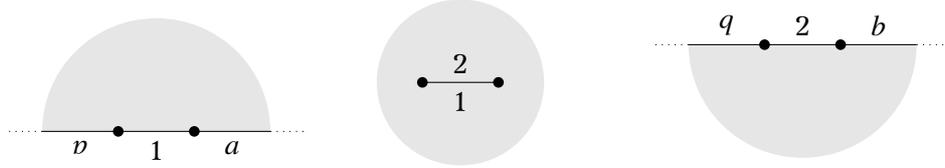
\begin{figure}[htb]
\begin{tikzpicture}

\begin{scope}[xshift=-1.5cm,yshift=.15cm]

 \fill[black!10] (-1,0)coordinate (a) -- (1.5,0)-- (a)+(2,0) arc (0:180:1.5)--(a)+(180:1.5) -- cycle;

   \draw (a)  -- node [below] {$1$} (0,0) coordinate (b);
 \draw (0,0) -- (1,0) coordinate[pos=.5] (c);
 \draw[dotted] (1,0) --coordinate (p1) (1.5,0);
 \fill (a)  circle (2pt);
\fill[] (b) circle (2pt);
\node[below] at (c) {$a$};

 \draw (a) -- node [above,rotate=180] {$a$} +(180:1) coordinate (d);
 \draw[dotted] (d) -- coordinate (p2) +(180:.5);

     \end{scope}

\begin{scope}[xshift=1.5cm,yshift=.8cm]
\fill[fill=black!10] (0.5,0)  circle (1.1cm);

 \draw (0,0) coordinate (a) -- coordinate (c) (1,0) coordinate (b);

 \fill (a)  circle (2pt);
\fill[] (b) circle (2pt);
\node[below] at (c) {$1$};
\node[above] at (c) {$2$};

    \end{scope}

\begin{scope}[xshift=6cm,rotate=180,yshift=-1.3cm]

 \fill[black!10] (-1.5,0)coordinate (a) -- (0,0)-- (0:1) arc (00:180:1.5) -- cycle;

   \draw (-1,0) coordinate (a) -- node [above] {$2$} (0,0) coordinate (b);
 \draw (a) -- +(-1,0) coordinate[pos=.5] (c) coordinate (e);
 \draw[dotted] (e) -- +(-.5,0);
 \fill (a)  circle (2pt);
\fill[] (b) circle (2pt);
\node[above] at (c) {$b$};

 \draw (b) -- node [below,rotate=-180] {$b$} +(0:1) coordinate (d);
 \draw[dotted] (d) -- +(0:.5);
\end{scope}

\end{tikzpicture}
\caption{Différentielle quadratique de $\Omega^{2}\moduli[0](6;-3,-3;-4)$ avec un résidu nul} \label{ex:8,8,12}
\end{figure}
\par
Les collages sont les suivants. Nous collons le bord inférieur de la partie polaire du $i$-ième pôle d'ordre pair au bord d'en haut du $(i+1)$-ième pôle d'ordre pair. Le bord inférieur de la partie polaire associée au pôle $P_{p}$ est collé au bord du segment du pôle~$P_{2}$. Enfin, tous les segments restants sont collés au bord de la partie polaire associée à~$P_{1}$.  Cette surface possède les propriétés souhaitées.
\end{proof}
 
Nous traitons maintenant le cas des strates ayant au moins deux zéros.
\begin{lem}\label{lem:g=0gen2}
Considérons la décomposition $\mu=(a_{1},\dots,a_{n};-b_{1},\dots,-b_{p};-c_{1},\dots,-c_{r})$ avec $n\geq2$. L'image de l'application résiduelle contient l'origine dans les trois cas suivants:
\begin{itemize}
\item[i)] $r \geq 2$;
\item[ii)] $\nim \geq 2$;
\item[iii)] $r=1$, $\nim = 1$ et la somme des $a_{i}$ pairs est supérieure ou égale à~$2p$.
\end{itemize} 
\end{lem}

\begin{proof}
La somme $r+\nim$ est toujours un nombre pair. Si $r\geq2$, on peut simplement éclater le zéro des différentielles données par le lemme~\ref{lem:g=0gen1nouvbis}.
\smallskip
\par
Nous traitons maintenant le cas (iii). L'éclatement de zéros permet de se ramener au cas $n=2$ avec un zéro d'ordre impair $a_{1}$ et un zéro d'ordre pair $a_{2} \geq 2p$. Nous écrivons $a_{1}=2l_{1}-1$ et $a_{2}=2l_{2}$ avec $l_{2} \geq p$. Nous considérons deux cas selon que $l_{1}\geq p$ ou $l_{1}<p$. Ces cas sont illustrés par la figure~\ref{fig:rgeq1ngeq2}: à gauche dans le premier et à droite dans le second.
\par
Dans le cas où $l_{1}\geq p$, on associe aux pôles d'ordres $-b_{i}$ les parties polaires d'ordres~$b_{i}$ et de types~$\tau_{i}$ associées à $(1;1)$. De plus, on choisit les types $\tau_{i}$ de telle sorte que la somme $\sum_{i}\tau_{i}$ soit inférieure ou égale à $ l_{1}$ et maximale pour cette propriété. On note $\bar{l_{1}}=l_{1}-\sum_{i}\tau_{i}$. Pour le pôle d'ordre $-c$ on procède de la façon suivante. On prend la partie polaire d'ordre~$c$ associée à $\left(\emptyset;v_{1},v_{2}\right)$ avec $v_{1}=v_{2}=1$ et de type $\bar{l_{1}}+1$.
On obtient la différentielle souhaitée en identifiant le bord inférieur de la partie polaire associée à $P_{i}$ au bord supérieur de celle de $P_{i+1}$. Le bord supérieur de $P_{1}$ est identifié au segment $v_{1}$ par rotation et le bord inférieur de la partie polaire d'ordre $b_{p}$ est identifié au segment $v_{2}$ par translation.
\par
Si $l_{1}<p$, alors on associe au pôle d'ordre $-c$ la partie polaire d'ordre $c$ associée à $(\emptyset;1)$. Pour l'un des pôles d'ordre $-b_{i}$, disons $-b_{1}$, on associe la partie polaire d'ordre $b_{1}$ associée à $(1;v_{1},v_{2})$, avec $v_{i}$ de même longueur, $v_{1}+v_{2}=1$ et l'angle (dans la partie polaire) entre ces deux étant $\pi$. Comme $a_{i}$ est impair, cette partie polaire est non dégénérée.
Pour~$l_{1}$ pôles d'ordre~$-b_{i}$, on associe des parties polaires triviales associées à $(v_{1};v_{1})$ et de type $\tau_{i}=b_{i}-1$. On colle ces parties polaires de manière cyclique à $v_{1}$ et~$v_{2}$. Le point correspondant à l'intersection entre $v_{1}$ et $v_{2}$ est la singularité d'ordre~$a_{1}$. On associe aux autres pôles d'ordres~$-b_{i}$ la partie polaire triviale associée à $(1;1)$. On colle ces parties polaires de manière cyclique aux parties polaires d'ordres $b_{1}$ et $c$ pour obtenir la surface plate souhaitée.
\begin{figure}[htb]
\center
 \begin{tikzpicture}[scale=.9]

\begin{scope}[xshift=-6cm]
    \fill[fill=black!10] (0,0) ellipse (1.5cm and .7cm);
\draw[] (0,0) coordinate (Q) -- (-1.5,0) coordinate[pos=.5](a);

\node[above] at (a) {$1$};
\node[below] at (a) {$2$};

\draw[] (Q) -- (.7,0) coordinate (P) coordinate[pos=.5](c);
\draw[] (Q) -- (P);

\fill (Q)  circle (2pt);

\fill[color=white!50!] (P) circle (2pt);
\draw[] (P) circle (2pt);
\node[above] at (c) {$v_{1}$};
\node[below] at (c) {$v_{2}$};

    \fill[fill=black!10] (0,-2.5) ellipse (1.1cm and .5cm);
\draw[] (.5,-2.5) coordinate (Q) --++ (-1.6,0) coordinate[pos=.5](b);

\fill (Q)  circle (2pt);

\node[above] at (b) {$2$};
\node[below] at (b) {$1$};

\end{scope}

\begin{scope}[xshift=-3cm]

 \fill[black!10] (.7,-.9)coordinate (a) --  ++ (90:1) coordinate (b)-- ++(-1.6,0) coordinate (c) coordinate[pos=.25] (e) coordinate[pos=.5] (f) coordinate[pos=.75] (g) -- ++ (90:1) coordinate (d) .. controls ++(10:3.5) and ++(-10:1) .. (a);
\draw[] (a) -- node [left] {$5$}  (b);
\draw[] (b) -- (c) -- node [right, rotate=180] {$5$}  (d);
\draw (b) --node [above] {$3$}node [below] {$4$}  ++(.85,0);
\node[above] at (e) {$v_{2}$};
\node[below,rotate=180] at (g) {$v_{1}$};
 \fill (f)  circle (2pt);
 \filldraw[fill=white] (b)  circle (2pt);
 \filldraw[fill=white] (c)  circle (2pt);

\fill[fill=black!10] (0,-2.5) ellipse (1.1cm and .5cm);
\draw[] (-.5,-2.5) coordinate (Q) --++ (1.5,0) coordinate[pos=.5](b);
\filldraw[fill=white] (Q)  circle (2pt);

\node[above] at (b) {$4$};
\node[below] at (b) {$3$};
\end{scope}

\begin{scope}[xshift=2cm]
\begin{scope}[yshift=0cm,xshift=-.5cm]
    \fill[fill=black!10] (.5,0) ellipse (1.1cm and .9cm);
   
 \coordinate (A) at (0,0);
  \coordinate (B) at (1,0);
    \coordinate (C)  at (.5,0);
 
 \filldraw[fill=white] (A)  circle (2pt);
\filldraw[fill=white] (B) circle (2pt);
 \fill (.57,0)  arc (0:-180:2pt); 

 \draw[]     (A) -- (B) coordinate[pos=.5](a);
 \draw (B) -- (C) coordinate[pos=.3](b);
 \draw (C) -- (A)  coordinate[pos=.6](c);

\node[above] at (a) {$1$};
\node[below] at (c) {$v_{1}$};
\node[below] at (b) {$v_{2}$};
\end{scope}

    \fill[fill=black!10] (0,-2.5) ellipse (1.1cm and .5cm);
\draw[] (-.5,-2.5) coordinate (Q) -- (.5,-2.5) coordinate[pos=.5](b) coordinate (P);

\filldraw[fill=white] (Q) circle (2pt);
\filldraw[fill=white] (P) circle (2pt);

\node[above] at (b) {$2$};
\node[below] at (b) {$1$};

\end{scope}
\begin{scope}[xshift=5.5cm]
\begin{scope}[rotate=180,yshift=-.5cm]
      
 \fill[black!10] (-1.5,0)coordinate (a) -- (0,0)-- (0:1) arc (0:180:1.5) -- cycle;

   \draw (-1,0) coordinate (a) -- node [above] {$2$} (0,0) coordinate (b);
 \draw (a) -- +(-1,0) coordinate[pos=.5] (c) coordinate (e);
 \draw[dotted] (e) -- +(-.5,0);
  \draw (b) -- node [below,rotate=180] {$b$} +(0:1) coordinate (d);
 \draw[dotted] (d) -- +(0:.5);
\filldraw[fill=white] (a) circle (2pt);
\filldraw[fill=white] (b) circle (2pt);
\node[above] at (c) {$b$};
\end{scope}

\begin{scope}[xshift=.5,yshift=-2.5cm]
\fill[fill=black!10] (0,0) coordinate (O) ellipse (1.1cm and .6cm);
\coordinate (P) at (0:.3);
\coordinate (Q) at (180:.3);
\draw[] (Q) -- (P)coordinate[pos=.5](b);
\filldraw[fill=white] (Q) circle (2pt);
\fill (P)  circle (2pt);

\node[below] at (b) {$v_{2}$};
\node[above] at (b) {$v_{1}$};
\end{scope}
\end{scope}
\end{tikzpicture}
\caption{Différentielle dont les résidus sont nuls dans la strate $\Omega^{2}\moduli[0](4,3;-6;-5)$ à gauche et dans $\Omega^{2}\moduli[0](1,10;\rec[-4][3];-3)$ à droite} \label{fig:rgeq1ngeq2}
\end{figure}
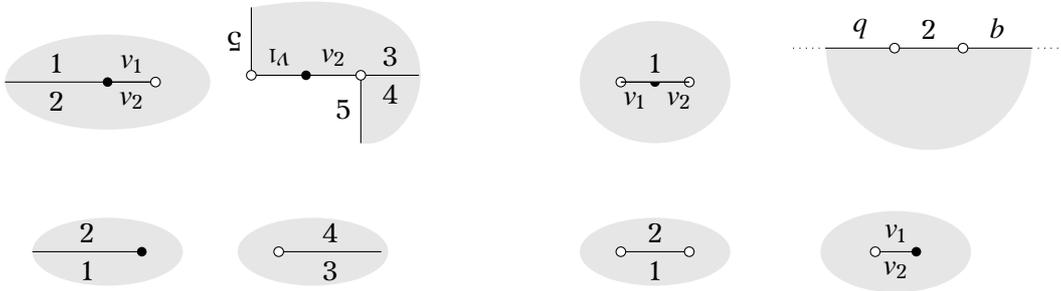
\par
Il reste à traiter le cas (ii) avec $r=1$. Comme le nombre de singularités d'ordre impair est pair, nous avons $\nim \geq 3$. L'éclatement de zéros permet de se ramener au cas $n=\nim=3$. Dans une strate $\Omega^{2}\moduli[0](a_{1},a_{2},a_{3};-b_{1},\dots,-b_{p};-c)$, l'ordre total des pôles est au moins $4p+3$, ce qui implique que l'ordre total des zéros vaut au moins $4p-1$. Par conséquent, la somme des ordres des deux zéros d'ordre le plus grand vaut toujours au moins $2p$ et donc ce cas s'obtient par éclatement de zéro à partir du cas (iii) traité plus haut.
\end{proof}

\subsection{Le complémentaire de l'origine}
\label{sec:poleimpresnozero}

Pour terminer la preuve du théorème~\ref{thm:g=0gen1bis} il suffit de montrer que le complémentaire de l'origine est dans l'image de l'application résiduelle des strates avec au moins un pôle impair. Nous traitons d'abord les strates avec un seul zéro.

\begin{lem}\label{lem:g=0gen1nouv}
Soit $\quadomoduli[0](a;-b_{1},\dots,-b_{p};-c_{1},\dots,-c_{r};\rec[-2][s])$ une strate de genre zéro avec $b_{i}$ pairs et $c_{j}$ impairs telle que $r\neq0$. L'image de l'application résiduelle contient le complémentaire de $(0,\dots,0)$ dans l'espace résiduel.
\end{lem}

\begin{proof}
Nous commençons par la construction d'une différentielle dans $\quadomoduli[0](\mu)$  de résidus $(R_{1},\dots,R_{p+s})\in \espresk[0][2](\mu)\setminus\left\{(0,\dots,0)\right\}$ avec $\mu=(a;-b_{1},\dots,-b_{p};-c_{1},\dots,-c_{r};\rec[-2][s])$. Si $R_{i}\neq0$, on note $r_{i}$ sa racine d'argument dans $\left] -\pi,0 \right]$. Cette construction est illustrée par la figure~\ref{ex:8,8,12res}.
\par 
Pour les pôles $P_{p+i}$ d'ordre $-2$, nous prenons une partie polaire d'ordre~$2$ associée à  $r_{p+i}$. Pour chaque pôle $P_{i}$ d'ordre $-b_{i}=-2\ell_{i}$ tel que $R_{i}\neq0$, nous prenons une partie polaire non triviale d'ordre $b_{i}$ associée à $(r_{i};\emptyset)$. Pour les pôles d'ordre $-b_{i}$ tels que $R_{i}=0$ nous prenons la partie polaire d'ordre $b_{i}$ associée à $(r_{j_{i}};r_{j_{i}})$ où $r_{j_{i}}$ est l'une des racines choisies précédemment.
\par
Maintenant, pour tous les pôles d'ordre $-c_{i}$ sauf un, disons $P_{1}$ d'ordre $-c_{1}$, nous prenons une partie polaire de type $c_{i}$ associée à $(1;\emptyset)$. Pour le dernier pôle $P_{1}$ d'ordre $-c_{1}$, nous prenons la partie polaire de type $c_{1}$, sans points d'intersection (voir le lemme~\ref{lem:noninter}), associée à $(\emptyset;\rec[1][r-1],r_{1},\dots,r_{t})$ où $t$ est le nombre de résidus non nuls.
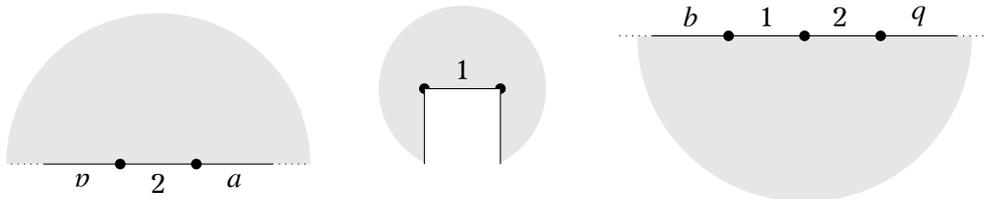
\begin{figure}[htb]
\begin{tikzpicture}


\begin{scope}[xshift=-1.5cm]
 \fill[black!10] (-1,0)coordinate (a) -- (1.5,0)-- (a)+(2.5,0) arc (0:180:2)--(a)+(180:1.5) -- cycle;

   \draw (a)  -- node [below] {$2$} (0,0) coordinate (b);
 \draw (0,0) -- (1,0) coordinate[pos=.5] (c);
 \draw[dotted] (1,0) --coordinate (p1) (1.5,0);
 \fill (a)  circle (2pt);
\fill[] (b) circle (2pt);
\node[below] at (c) {$a$};

 \draw (a) -- node [above,rotate=180] {$a$} +(180:1) coordinate (d);
 \draw[dotted] (d) -- coordinate (p2) +(180:.5);

     \end{scope}

\begin{scope}[xshift=1.5cm,yshift=1cm]
\fill[fill=black!10] (0.5,0)coordinate (Q)  circle (1.1cm);
    \coordinate (a) at (0,0);
    \coordinate (b) at (1,0);

     \fill (a)  circle (2pt);
\fill[] (b) circle (2pt);
    \fill[white] (a) -- (b) -- ++(0,-1.1) --++(-1,0) -- cycle;
 \draw  (a) -- (b);
 \draw (a) -- ++(0,-1);
 \draw (b) -- ++(0,-1);

\node[above] at (Q) {$1$};
    \end{scope}

\begin{scope}[xshift=6.5cm,yshift=1.7cm,rotate=180]

 \fill[black!10] (-1,0)coordinate (a) -- (1,0)-- (2.2,0) arc (0:180:2.2) -- cycle;
\draw (-1,0) coordinate (a) -- node [above,xshift=-1] {$2$} (0,0) coordinate (b);
\draw (b) -- node [above] {$1$} +(1,0) coordinate (c);
\draw (c) -- node [above] {$b$} +(1,0) coordinate (d);
\draw[dotted] (d) -- +(.5,0);
\fill (a)  circle (2pt);
\fill[] (b) circle (2pt);
\fill[] (c) circle (2pt);

\draw (a) -- node [below,rotate=180] {$b$} +(180:1) coordinate (e);
\draw[dotted] (e) -- +(180:.5);

\end{scope}

\end{tikzpicture}
\caption{Différentielle quadratique de $\Omega^{2}\moduli[0](6;-3,-3;-4)$ avec un résidu non nul au pôle d'ordre $-4$} \label{ex:8,8,12res}
\end{figure}
\par
La surface est obtenue par les recollements suivants. Nous collons le segment inférieur~$r_{j_{i}}$ de chaque partie polaire triviale au bord de la partie non triviale du pôle $P_{j_{i}}$. Nous faisons les collages similaires, pour chacun des pôles d'ordres pairs dont le résidu est nul. Ensuite nous collons par translation les bords des parties polaires différentes de $P_{1}$ aux segments correspondants du bord de la partie polaire de $P_{1}$.
La différentielle $\xi$ associée à cette surface plate possède les ordres de pôles et les résidus quadratiques souhaités.

Il reste donc à montrer que le genre de la surface est zéro et que la différentielle $\xi$ possède un unique zéro.
Pour cela, il suffit de vérifier que si l'on coupe la surface le long d'un lien-selle, alors on sépare cette surface en deux parties. C'est une conséquence du fait que les liens-selles correspondent aux bords des domaines polaires. 
\end{proof}

Le cas des strates avec au moins deux zéros s'obtient immédiatement par éclatement de singularité. Le théorème~\ref{thm:g=0gen1bis} est donc démontré, ainsi que le cas $r \geq 1$ du théorème~\ref{thm:surjimp4}.

\section{Différentielles à pôles pairs non tous doubles}
\label{sec:NTD}

Cette section est consacrée aux strates  $\quadomoduli[0](a_{1},\dots,a_{n};-b_{1},\dots,-b_{p},\rec[-2][s])$ avec $p\geq1$. Tout au long de cette section, nous notons $b_{i}:=2\ell_{i}$ et $a_{i}:=2l_{i}+\bar{a_{i}}$ avec $\bar{a_{i}}\in\{-1,0\}$.
À l'exception de la section~\ref{sub:pair4impair}, nous supposerons toujours que $a_{1},a_{2}$ sont impairs tandis que les $a_{i}$ sont pairs pour tout $i\geq3$.

\subsection{Le cas du résidu nul}\label{sec:reszeropolesdoubles}
Dans cette section nous montrons que les strates de la forme $\quadomoduli[0](a_{1},\dots,a_{n};-b_{1},\dots,-b_{p})$ qui possèdent une différentielle dont les résidus sont nuls si et seulement si la somme des zéros pairs est $\geq 2p $.

Nous commençons par montrer que cette condition est nécessaire pour que l'application résiduelle contienne l'origine.
\begin{lem}\label{lem:condnecngeq3}
Soit $\mu=(a_{1},\dots,a_{n};-b_{1},\dots,-b_{p})$ une décomposition telle que $a_{1}$ et $a_{2}$ sont impairs et~$a_{i}$ est pair pour $i \geq 3$. Si l'application résiduelle $\appresquad[0](\mu)$ contient $(0,\dots,0)$, alors
\begin{equation}\label{eq:inezero}
 \sum_{i=3}^{n} a_{i} \geq 2p \,.
\end{equation}
\end{lem}

\begin{proof}
Soit $\xi$ une différentielle dans la strate $\quadomoduli[0](\mu)$ dont les résidus sont nuls. Son revêtement canonique $\pi$ n'est ramifié qu'en deux singularités d'ordre impair et est donc de genre zéro. Les singularités de la racine $\whomega$ de $\pi^{\ast}\xi$ sont deux zéros $\whz_{1},\whz_{2}$ d'ordres $a_{1}+1$ et $a_{2}+1$ respectivement, $2n-4$ zéros d'ordres $\frac{a_{3}}{2},\frac{a_{3}}{2},\dots,\frac{a_{n}}{2},\frac{a_{n}}{2}$ et $2p$ pôles d'ordres $-\frac{b_{1}}{2},-\frac{b_{1}}{2},\dots,-\frac{b_{p}}{2},-\frac{b_{p}}{2}$ dont le résidu est nul. Il s'ensuit que $\whomega$ est  exacte.
\par
Pour toute fonction méromorphe $f\colon \PP^{1}  \to \CC$ telle que $\whomega=df$, la symétrie du revêtement canonique implique que $f(\whz_{1})=f(\whz_{2})$.  On supposera donc sans perte de généralité que $f(\whz_{1})=f(\whz_{2})=0$. En tant que primitive de $\whomega$, la fonction $f$ a $2p$ pôles d'ordres $1-\frac{b_{j}}{2}$. La somme des ordres des pôles est donc $2p-\sum_{j=1}^{p} b_{j}$.
\par
Aux points $\whz_{1}$ et $\whz_{2}$, la fonction $f$ possède des zéros d'ordres $2+a_{1}$ et $2+a_{2}$. Cela implique que $4+a_{1}+a_{2} \leq \sum_{j=1}^{p} b_{j} -2p$. De façon équivalente, $a_{1}+a_{2} \leq \sum_{i=1}^{n} a_{i} -2p$ et donc $\sum_{i=3}^{n} a_{i} \geq 2p$.
\end{proof}

Nous montrons maintenant que l'équation~\eqref{eq:inezero} est suffisante pour que l'origine soit dans l'image de l'application résiduelle.
Pour que cette équation soit satisfaite, il faut que la strate contienne au moins un zéro d'ordre pair.

Nous traitons d'abord le cas avec un unique zéro pair et $p\geq 1$ pôles.
\begin{lem}\label{lem:condsufngeq3}
Soit $\mu=(a_{1},a_{2},a_{3};-b_{1},\dots,-b_{p})$ une décomposition telle que $a_{1}$ et $a_{2}$ sont impairs et~$a_{3}$ est pair. Si $a_{3} \geq 2p$, alors l'application résiduelle $\appresquad[0](\mu)$ contient $(0,\dots,0)$.
\end{lem}

\begin{proof}
Nous supposons dans un premier temps que $a_{1}+a_{2} \geq 2p-4$. Rappelons que pour $i\in\{1,2\}$ on note $a_{i}=2l_{i}-1$ avec $l_{i}\geq0$. Nous supposerons que $a_{1}\leq a_{2}$ et  $b_{1}\geq b_{2}\geq \dots \geq b_{p}$.  

Commençons par le cas $p=1$.   Un plan infini coupé d'une cicatrice dont les deux côtés sont chacun divisés en deux segments égaux et identifiés par rotation donne une différentielle dans la strate $\quadomoduli[0](-1,-1,2;-4)$. Pour les strates $b_{1}>4$, on peut augmenter de $2$ l'ordre d'un zéro en coupant le plan par une demi-droite partant de celui-ci et en y collant un plan.  Nous obtenons de cette façon une différentielle à résidus nuls dans toutes les strates considérées.

Nous considérons maintenant le cas $p\geq2$ avec $a_{1}+a_{2} \geq 2p-4$. Soit $k'$ tel que 
$$ (b_{1}-4) + \sum_{i=2}^{k'-1}(b_{i}-2) < a_{2}  \leq (b_{1}-4) + \sum_{i=2}^{k'}(b_{i}-2) \,.$$
De plus, comme $a_{1}+a_{2} \geq 2p-4$ il existe  $k \geq k'$ tel que $p-k \leq l_{1}$.
On coupe la différentielle de $\quadomoduli[0](-1,-1,2;-4)$ le long d'un lien-selle et on colle  de manière cyclique les $k-1$ parties polaires associés à $(1;1)$ correspondant aux pôles $b_{i}$ pour $i\in\{2,\dots,k\}$. Les parties polaires associées aux autres pôles sont collées de manière cyclique aux segments obtenus en coupant l'autre lien-selle. On choisit alors les types des $k-1$ parties polaires de telle sorte que leurs contributions à l'angle de la singularité correspondant à $a_{2}$ soit inférieures ou égale à~$2l_{2}\pi$. Puis on coupe la partie polaire associée au pôle d'ordre $-b_{1}$ le long d'une demi-droite qui part de la singularité correspondant à $a_{2}$ et on ajoute de manière cyclique le nombre de plans afin d'obtenir l'angle $\pi(a_{2}+2)$. On fait la construction similaire pour $a_{1}$ et on obtient la différentielle avec les propriétés désirées.
\par
Nous supposons finalement que $a_{1}+a_{2} \leq 2p-6$ (la somme de ces ordres est toujours paire), i.e. que $l_{1}+l_{2} \leq p-2$. On considère deux parties polaires $P_{1}$ et $P_{2}$ respectivement associées   à $(1,1;2)$ et $(2;1,1)$. On colle $l_{i}$ parties polaires associées à $(1;1)$ de manière cyclique aux segments de longueur $1$ de $P_{i}$. Le type est choisi de sorte  que chaque partie polaire ne contribue que d'un angle $2\pi$ au singularité $a_{1}$ et $a_{2}$. Finalement, pour les autres pôles, on colle les parties polaires associée à $(2;2)$ entre les segments de longueur $2$ des $P_{i}$. 
\end{proof}

Pour obtenir l'origine dans une strate avec $\npa\geq2$ zéros pairs il suffit d'éclater  l'unique zéro pair d'une différentielle obtenue dans le lemme~\ref{lem:condsufngeq3}.

\subsection{Configurations non nulles pour $\mathfrak{i}=n=2$}

Dans cette section, nous considérons les strates avec deux zéros, d'ordres impairs.
Par la remarque~\ref{rem:redux}  pour établir qu'une configuration de résidus quadratiques appartenant à une strate de résonance est réalisable dans la strate de différentielle, il suffit de le prouver pour une configuration de résidus réels positifs de la même strate de résonance.
Dans cette section, nous notons par  $R_{1},\dots,R_{s} \in \mathbb{R}_{+}$ les résidus et $r_{1},\dots,r_{s}$ leurs racines positives respectives. Les surfaces plates correspondant à de telles différentielles sont formées de domaines polaires reliés par exactement~$s$ liens-selles horizontaux définissant un graphe d'incidence ayant un unique cycle (voir section~\ref{sec:coeur}).

\subsubsection{Le cas $a_{2}=-1$}

Nous donnons dans ce cas une construction systématique.

\begin{prop}\label{prop:5Negatif}
L'image de l'application résiduelle de $\quadomoduli[0](a,-1;-b_{1},\dots,-b_{p};\rec[-2][s])$ avec $p\neq 0$
 contient les résidus $R_{i}\in\RR_{+}$ non tous nuls.
\end{prop}

\begin{proof}
Nous supposerons que les pôles dont le résidu quadratique est nul sont $P_{1},\dots,P_{t}$ et que les résidus $R_{t+1},\dots,R_{p+s}$  sont non nuls. 
On associe au pôle $P_{1}$ la partie polaire triviale d'ordre $2\ell_{1}$ associée à $(v_{1},v_{2};r_{t+1},\dots,r_{p+s})$ si $R_{1}=0$ et la partie polaire non triviale associée à $(v_{1},v_{2};r_{2},\dots,r_{p+s})$ si $R_{1}\neq 0$ avec $v_{1}=v_{2}=\frac{1}{2} \sum_{i\geq1} r_{i}$.
Pour les pôles $P_{i}$ avec $i\geq2$, on prend une partie polaire associée à $(r_{i};\emptyset)$ si $R_{i} \neq 0$ et $(r_{p+s};r_{p+s})$  si $R_{i}=0$. Il reste à identifier tous les segments par translation à exception des $v_{i}$ que nous identifions par rotation d'angle $\pi$. Cette surface plate possède les invariants locaux souhaités.
\end{proof}

\subsubsection{Résidus tous non nuls (sauf peut-être l'un d'entre eux)}

Nous traitons d'abord du cas $p=1$ dans lequel un seul pôle n'est pas un pôle double.
\begin{prop}\label{prop:5p=1}
Toute configuration $(R_{1},\dots,R_{s+1})$ de résidus quadratiques tels que $R_{1} \in \mathbb{R}_{+}$ et $R_{2},\dots,R_{s+1} \in \mathbb{R}_{+}^{\ast}$ est réalisable dans la strate $\quadomoduli[0](a_{1},a_{2};-b;\rec[-2][s])$ avec $b \geq 4$ pair et $s \geq 1$ sauf dans les cas suivants :
\begin{enumerate}
    \item les configurations $\mathbb{C}^{\ast}\cdot (1,\dots,1)$ dans $\quadomoduli[0](2l+1,2l+1;-4;\rec[-2][2l+1])$ avec $l \geq 0$;
    \item les configurations $\mathbb{C}^{\ast}\cdot (0;1,\dots,1)$ dans $\quadomoduli[0](2l-1,2l+1;-4;\rec[-2][2l])$ avec $l \geq 1$.
\end{enumerate}
\end{prop}

\begin{proof}
Posons $a_{1} \geq a_{2}$ et $t=\frac{1}{2}(a_{1}-1)$. Nous supposerons d'abord que $b=4$.
Sans perte de généralité, nous supposerons que $R_{2} \geq \dots \geq R_{s+1}$. Posons $S= \sum\limits_{i=1}^{t+1} r_{i} - \sum\limits_{j=t+2}^{s+1} r_{j}$. Nous avons $S>0$ sauf éventuellement dans  les cas exceptionnels suivants:
\begin{enumerate}[(i)]
    \item $s$ impair et $a_{1}=a_{2}=s$;
    \item $r_{1}=0$, $r_{2}=\dots=r_{s+1}$ avec $s$ pair, $a_{1}=s+1$ et $a_{2}=s-1$.
\end{enumerate}
En dehors de ces deux cas, nous construisons la partie polaire d'ordre $4$ associée aux vecteurs $(r_{2},\dots,r_{t+1};v_{1},r_{t+2},\dots,r_{s+1},v_{2})$ avec $v_{1}=v_{2}=\frac{S}{2}$. En collant des cylindres de circonférences $r_{2},\dots,r_{s+1}$ sur les segments correspondants, puis en identifiant les deux segments $v_{1}$ et $v_{2}$ par rotation, nous obtenons une surface ayant les invariants adéquats.
\par
Le cas (i) correspond à la famille  d'obstructions (1) et le cas (ii) avec $r_{1}=\dots=r_{s+1}$ correspond à la famille~(2).  Si $S>0$, on fait la construction précédente.

Supposons que les $r_{1},\dots,r_{s+1}$ sont non tous égaux tels que $S \leq 0$. Ces conditions impliquent en particulier que $r_{1}<r_{2}$.
Comme  $s$ est impair, il existe une famille $\epsilon_{3},\dots,\epsilon_{s+1} \in \lbrace{ \pm 1 \rbrace}$ de signes telle que la somme $S'=\sum\limits_{i=3}^{s+1}\epsilon_{i}r_{i}$ satisfait $0 \leq S'< r_{2}-r_{1}$. La construction est la suivante. Le cylindre de circonférence $r_{2}$ est bordé par deux segments (reliant les deux zéros) dont les longueurs respectives sont $\frac{1}{2}(r_{2}-r_{1}-S')>0$ et $\frac{1}{2}(r_{2}+r_{1}+S')>0$. Tous les autres cylindres ont un unique segment de bord. La partie polaire d'ordre $4$ est  associée à $\left(\frac{1}{2}(r_{2}+r_{1}+S'), r_{i_{1}},\dots,r_{i_{j}}; r_{i_{j+1}},\dots,r_{i_{s-3}},\frac{1}{2}(r_{2}-r_{1}-S')\right)$ où $\epsilon_{i_{1}} = \dots = \epsilon
_{i_{j}} = 1$ et  $\epsilon_{i_{j+1}} = \dots = \epsilon
_{i_{s-3}} = -1$. Les longueurs des différents segments montrent que le résidu du pôle quadruple est bien $R_{1}$.
\par
Les constructions se généralisent pour tout $b\geq6$ en collant des plans de façon cyclique sur la partie polaire d'ordre $4$. Puisque nous pouvons choisir sur quelle singularité conique portera l'accroissement de l'angle, nous pouvons partir d'une strate dans laquelle il n'y a pas d'obstruction. Ainsi, il n'y a pas d'obstructions pour $b \geq 6$.
\end{proof}

Nous donnons maintenant la construction pour le cas dans lequel il n'y a que deux pôles.
\begin{prop}\label{prop:5polesp2}
Toute configuration $(R_{1},R_{2})\in(\RR_{+}^{\ast})^{2}$ de résidus est réalisable dans la strate $\quadomoduli[0](a_{1},a_{2};-b_{1},-b_{2})$ avec $b_{1},b_{2} \geq 4$ pairs sauf dans le cas des configurations de $\mathbb{C}^{\ast}\cdot (1,1)$ dans les strates suivantes:
\begin{enumerate}
    \item $\quadomoduli[0](b-1,b-1;-b,-b-2)$ avec $b \geq 2$ pair;
    \item $\quadomoduli[0](b-1,b-3;-b,-b)$ avec $b \geq 4$ pair.
\end{enumerate}
\end{prop}

\begin{proof}
Posons $a_{1} \leq a_{2}$ et $b_{1} \leq b_{2}$. Si $a_{2} \geq b_{1}$, alors nous pouvons donner une première construction. Le domaine du pôle d'ordre $-b_{1}$ est bordé  par un unique lien-selle de longueur $r_{1}$ reliant le zéro d'ordre $a_{2}$ avec lui-même pour un angle intérieur de $(b_{1}-1)\pi$. On associe au pôle d'ordre $-b_{2}$ la partie polaire $\left(\frac{r_{1}+r_{2}}{2},\frac{r_{1}+r_{2}}{2};r_{1}\right)$. Les deux premiers segments sont identifiés par une rotation et on ajoute $a_{1}+1$ demi-plans à leur intersection pour que l'angle entre ces deux segments vaille $(a_{1}+2)\pi$.
Nous obtenons bien une surface avec les invariants locaux voulus.
\par
Si $a_{2}<b_{1}$, alors la strate est soit de la forme $\quadomoduli[0](b-1,b-1;-b,-b-2)$, soit de la forme $\quadomoduli[0](b-1,b-3;-b,-b)$. En supposant que $R_{1} < R_{2}$, nous pouvons donner la construction suivante: un domaine polaire est associé à $\left(\frac{r_{2}-r_{1}}{2},\frac{r_{2}+r_{1}}{2};\emptyset\right)$ et l'autre à $\left(\frac{r_{2}-r_{1}}{2};\frac{r_{2}+r_{1}}{2}\right)$.  En choisissant de manière adéquate le type de ces parties polaires on obtient des différentielles dans les strates souhaitées.
\end{proof}

Le cas dans lequel l'un des zéros est d'ordre $1$ et qu'il y a au moins un pôle double nécessite une construction spécifique.

\begin{prop}\label{prop:5NONNUL+1}
Toute configuration $(R_{1},\dots,R_{p+s})$ de résidus quadratiques réels positifs non nuls est réalisable dans la strate $\quadomoduli[0](1,a;-b_{1},\dots,-b_{p},\rec[-2][s])$ avec $a \geq 1$, $p \geq 2$ et $s \geq 1$.
\end{prop}

\begin{proof}
Nous divisons la preuve en deux cas.
\par
\paragraph{\bf Il existe un pôle double tel que $r_{p+s}< \sum\limits_{i=1}^{p+s-1} r_{i}$.}
On généralise la construction de la proposition~\ref{prop:5Negatif} de la façon suivante. On choisit un pôle d'ordre $-b_{1}$ et on lui associe la partie polaire de la forme $(v_{1},r_{p+s},v_{2};r_{2},\dots,r_{p+s-1})$. Puis on colle les autres pôles aux segments $r_{i}$ et les $v_{i}$ ensemble pour obtenir la différentielle désirée.
\smallskip
\par
\paragraph{\bf Il n'existe pas de pôle double tel que $r_{p+s}< \sum\limits_{i=1}^{p+s-1} r_{i}$.}
Dans ce cas, on a $s=1$ et $r_{p+1} \geq \sum\limits_{i=1}^{p} r_{i}$. Le graphe d'incidence (voir la section~\ref{sec:coeur}) de la différentielle que nous construisons possède un cycle de longueur $2$, formé des domaines des pôles d'ordres $-b_{1}$ et $-b_{2}$. On colle sur ce dernier sommet une chaîne formée dans l'ordre des domaines des pôles d'ordres $-b_{3}$ jusqu'à $-b_{p}$, pour finir avec le domaine de l'unique pôle double comme sommet de valence $1$. Pour $3 \leq k \leq p$, le domaine du pôle d'ordre $-b_{k}$ a pour bord deux liens-selles de longueurs respectives $\sum\limits_{i=k}^{p+1} r_{i}$ et $\sum\limits_{i=k+1}^{p+1} r_{i}$, séparés par des angles qui sont des multiples entiers de $2\pi$. En particulier, le lien-selle reliant la chaîne au cycle est de longueur $\sum\limits_{i=3}^{p+1} r_{i}$. Les deux autres liens-selles de bord du domaine du pôle d'ordre $-b_{2}$ (qui correspondent aux arêtes du cycle) sont de longueurs $\frac{1}{2}\sum\limits_{i=1}^{p+1} r_{i}$ et $-r_{1}+\frac{1}{2}\sum\limits_{i=1}^{p+1} r_{i}$ (nécessairement positif d'après l'hypothèse sur $r_{p+1}$), tandis que les trois angles sont $\pi$, $2\pi$ et $(b_{2}-2)\pi$. Enfin, le domaine du pôle d'ordre $-b_{1}$ a deux angles de magnitudes $2\pi$ et $(b_{1}-2)\pi$, obtenant un résidu quadratique $r_{1}^{2}$. Nous avons bien construit une différentielle avec une singularité conique d'angle $3\pi$.
\end{proof}

Nous donnons maintenant une construction générale quand au moins un pôle est un pôle double.

\begin{prop}\label{prop:MAIN5s1}
Toute configuration $(R_{1},\dots,R_{p+s})$ de résidus quadratiques réels positifs non nuls est réalisable dans la strate $\quadomoduli[0](a_{1},a_{2};-b_{1},\dots,-b_{p};\rec[-2][s])$ avec $3\leq a_{1}\leq a_{2}$, $p \geq 2$ et $s \geq 1$.
\end{prop}

\begin{proof} 
Dans toute la preuve on pose $b_{1}\geq b_{2}\geq \dots \geq b_{p}>2=b_{p+1}=\dots=b_{p+s}$.
Nous la scindons  en deux cas. 
\par
\paragraph{\bf Il existe un pôle double tel que $r_{p+s} \geq \sum\limits_{i=1}^{p+s-1} r_{i}$.}
Considérons la construction du second cas de la preuve de la proposition~\ref{prop:5NONNUL+1}. En changeant le type des parties polaires d'ordres~$b_{1}$ et $b_{2}$ correspondant aux pôles de la boucle, on peut obtenir tous les ordres satisfaisant $1\leq a_{1}\leq 1 + (b_{1}-4)+(b_{2}-4)$. On retire ensuite un pôle d'ordre $-b_{3}$ de la chaîne du graphe d'incidence pour le mettre au bord de la singularité $a_{1}$ dans la partie polaire $b_{1}$. Comme précédemment, le fait que $r_{p+s} \geq \sum\limits_{i=1}^{p+s-1} r_{i}$ implique que cette opération est possible. On obtient alors tous les ordres $1 + b_{3}\leq a_{1}\leq 1 + b_{3} + (b_{1}-4)+(b_{2}-4)$ en variant les types des parties polaires d'ordres $b_{1}$ et $b_{2}$. En retirant d'autres pôles de la chaîne pour les mettre dans la boucle on pourra obtenir toutes les valeurs de $a_{1}$ à condition que $b_{1}>4$. Si $b_{1} = \dots = b_{p} = 4$, alors on obtient les zéros de la forme $1+4k$ pour $k\geq 0$.
\par
Les zéros de la forme $-1 +4k$ sont obtenus à partir de la construction pour un zéro d'ordre $-1$ dans la preuve de la proposition~\ref{prop:5Negatif}. On choisit $k=k_{1}+2k_{2}$ avec $k_{1}\leq p$ et $2k_{2}<s$. On retire $k_{1}$ pôles d'ordre $-4$ et $k_{1}$ pôles d'ordre $-2$ du bord du zéro $a_{2}$ et on les déplace au bord du zéro d'ordre $a_{1}$.
\smallskip
\par
\paragraph{\bf Tous les pôles doubles satisfont $r_{p+j}< \sum\limits_{i\neq p+j} r_{i}$.}
Partons de la construction du premier cas de la preuve de la proposition~\ref{prop:5NONNUL+1}. En répartissant les demi-plans de la partie polaire du pôle spécial d'ordre $-b_{1}$ le long de demi-droites partant de $a_{1}$, on obtient tous les ordres $1\leq a_{1}\leq 1 + (b_{1}-4)$. Comme $p\geq2$, soit $r_{1}$ soit $r_{2}$ est inférieur à la somme des autres résidus. En faisant jouer à ce pôle le rôle du pôle double, on obtient tous les ordres satisfaisant $b_{2}-1\leq a_{1}\leq b_{1}+b_{2}-5$.

Dans les autres cas, on va décrire un ensemble $A_{1}$ constitué des pôles qui vont contribuer à $a_{1}$. Notons $-b_{3,4}$ le pôle de $\lbrace -b_{3},-b_{4} \rbrace$ de résidu minimal. On le met alors dans $A_{1}$. Comme précédemment, la somme des résidus des pôles de $A_{1}$ est inférieure à la somme des autres pôles et on peut faire la construction du premier paragraphe. On obtient tous les ordres $b_{1}+b_{3,4}-1\leq a_{1}\leq b_{3,4}+ (b_{2}-1) + (b_{1}-4)$. On a $(b_{2}-1) + (b_{1}-4) \geq b_{1}+b_{3,4}-1$ sauf si $b_{3,4} = b_{2}-2$. Dans ce cas pour obtenir l'ordre manquant, on ajoute au cas précédent un pôle double de résidu minimal. On continue alors cette procédure en mettant le pôle $-b_{5,6}$ de résidu minimal de  $\lbrace -b_{5},-b_{6} \rbrace$ dans $A_{1}$ et ainsi de suite.

Il reste à montrer que l'on peut obtenir tous les ordres de cette façon. Comme $a_{1}+a_{2} +4 = \sum_{i=1}^{p+s} b_{i} $, on a $a_{1}+2 \leq \frac{1}{2}\sum_{i=1}^{p+s} b_{i} $.  La procédure que l'on vient de décrire permet d'obtenir tous les ordres $\leq\sum_{i=3}^{p+s-1} b_{i,i+1}+ (b_{2}-1) + (b_{1}-4)$ avec $b_{i,i+1} =\min \lbrace b_{i},b_{i+1}\rbrace$. Cette somme est supérieure ou égale à $b_{1}-3 + \sum \max \lbrace b_{i},b_{i+1}\rbrace$ où les $i$ sont impairs supérieurs ou égaux à~$3$. En effet, la somme des différences $\max \lbrace b_{i},b_{i+1}\rbrace-\min \lbrace b_{i},b_{i+1}\rbrace$ est inférieure ou égale à $b_{2}-2 \geq2$.  Cette somme est donc supérieure ou égale à $a_{1}+2$ comme souhaité.
\end{proof}

Enfin, nous donnons les constructions pour le cas dans lequel il n'y a aucun pôle double, obtenant le résultat général.

\begin{prop}\label{prop:MAIN5}
Toute configuration $(R_{1},\dots,R_{p})$ de résidus quadratiques réels positifs et non nuls est réalisable dans la strate
$\quadomoduli[0](a_{1},a_{2};-b_{1},\dots,-b_{p}))$ sauf dans les cas suivants :
\begin{enumerate}
    \item les configurations  $\mathbb{C}^{\ast}\cdot (1,1)$ dans $\quadomoduli[0](b-1,b-1;-b,-b-2)$ avec $b \geq 2$ pair;
    \item les configurations  $\mathbb{C}^{\ast}\cdot (1,1)$ dans les strates $\quadomoduli[0](b-1,b-3;-b,-b)$ avec $b \geq 4$ pair.
\end{enumerate}
\end{prop}

La proposition~\ref{prop:5Negatif} montre le cas $a_{1}=-1$ et la  proposition~\ref{prop:5polesp2} le cas $p=2$. Les deux lemmes qui viennent considèrent donc les cas $1\leq a_{1}\leq a_{2}$ et $p \geq 3$. Le premier se restreint au cas où tous les résidus sont égaux et le second considère le cas général où au moins deux résidus sont distincts.

\begin{lem}\label{lm:realquadcolib}
L'image de l'application résiduelle de $\Omega^{2}\moduli[0](a_{1},a_{2};-b_{1},\ldots,-b_{p})$ avec $a_{1},a_{2}$ impairs et $p\geq 3$ contient $(1,\dots,1)$.
\end{lem}

\begin{proof}
Nous scindons la preuve selon si $p$ est pair ou impair.
\smallskip
\par
\paragraph{\bf  $p$ est impair $\geq3$.}
Nous collons des domaines polaires standards (voir la section~\ref{sec:briques}) le long de liens-selles selon un graphe d'incidence. Ce graphe est composé d'une boucle formée par un nombre $s$  de sommets de valence $2$. On colle à l'un de ses sommets une chaîne composée de $p-s$ sommets. Le cas extrême $p=s$ correspond à un graphe cyclique.

Si $s\neq p$, le dernier sommet de la chaîne correspond à un domaine polaire associé aux vecteurs $(1;\emptyset)$. Les $p-s-2$ sommets intermédiaires de la chaîne correspondent à des domaines polaires associés aux vecteurs $(2;1)$ et $(1;2)$.
Le sommet de valence $3$ correspond aux vecteurs $(1,\frac{1}{2};\frac{1}{2})$ si la chaîne est de longueur paire et $(2;\frac{1}{2},\frac{1}{2})$ si la longueur est impaire. Dans tous les cas, les sommets de la boucle correspondent à des domaines polaires associés aux vecteurs $(\frac{1}{2},\frac{1}{2};\emptyset)$.
 La différentielle quadratique obtenue en recollant les parties polaires selon le graphe d'incidence  possède tous ses résidus égaux à~$1$. Cette construction est illustrée  dans le cas $p=3$ dans les  figures~\ref{fig:residu17} et~\ref{fig:residu35}.
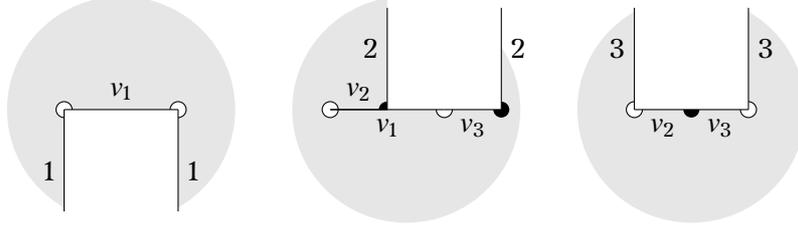
\begin{figure}[hbt]
\center
\begin{tikzpicture}[scale=1.5]
\begin{scope}[xshift=-2.5cm]
\fill[fill=black!10] (0,0)coordinate (Q)  circle (1cm);
    \coordinate (a) at (-1/2,0);
    \coordinate (b) at (1/2,0);

     \filldraw[fill=white] (a)  circle (2pt);
\filldraw[fill=white](b) circle (2pt);
    \fill[white] (a) -- (b) -- ++(0,-1.1) --++(-1,0) -- cycle;
 \draw  (a) -- (b) coordinate[pos=.5](c) ;
 \draw (a) -- ++(0,-.9) coordinate[pos=.6](e);
 \draw (b) -- ++(0,-.9)  coordinate[pos=.6](f);

\node[above] at (c) {$v_{1}$};
\node[left] at (e) {$1$};
\node[right] at (f) {$1$};
    \end{scope}
\begin{scope}[xshift=0cm]
\fill[fill=black!10] (0,0)coordinate (Q)  circle (1cm);
    \coordinate (a) at (-2/3,0);
    \coordinate (b) at (1/3,0);
    \coordinate (c) at (5/6,0);
      \coordinate (d) at (-1/6,0);

  \filldraw[fill=white] (a)  circle (2pt);
  \filldraw[fill=white](b) circle (2pt);
    \fill (c)  circle (2pt);
\fill (-.1,0)  arc (0:180:2pt);
    \fill[white] (d) -- (c) -- ++(0,1.1) --++(-1,0) -- cycle;
 \draw  (a) -- (b) coordinate[pos=.5](e);
 \draw  (a) -- (d) coordinate[pos=.5](f);
 \draw  (b) -- (c) coordinate[pos=.5](g);
 \draw (d) -- ++(0,.9) coordinate[pos=.6](h);
 \draw (c) -- ++(0,.9)  coordinate[pos=.6](i);

\node[below] at (e) {$v_{1}$};
\node[above] at (f) {$v_{2}$};
\node[below] at (g) {$v_{3}$};
\node[left] at (h) {$2$};
\node[right] at (i) {$2$};
    \end{scope}

\begin{scope}[xshift=2.5cm]
\fill[fill=black!10] (0,0)coordinate (Q)  circle (1cm);
    \coordinate (a) at (-1/2,0);
    \coordinate (b) at (1/2,0);

     \filldraw[fill=white] (a)  circle (2pt);
\filldraw[fill=white]  (b) circle (2pt);
\fill (Q) circle (2pt);
    \fill[white] (a) -- (b) -- ++(0,1.1) --++(-1,0) -- cycle;
 \draw  (a) -- (b) coordinate[pos=.25](c) coordinate[pos=.75](d);
 \draw (a) -- ++(0,.9) coordinate[pos=.6](e);
 \draw (b) -- ++(0,.9)  coordinate[pos=.6](f);

\node[below] at (c) {$v_{2}$};
\node[below] at (d) {$v_{3}$};
\node[left] at (e) {$3$};
\node[right] at (f) {$3$};
    \end{scope}
    \end{tikzpicture}
\caption{Différentielle de $\Omega^{2}\moduli[0](1,7;-4,-4,-4)$ dont les résidus sont $(1,1,1)$} \label{fig:residu17}
\end{figure}
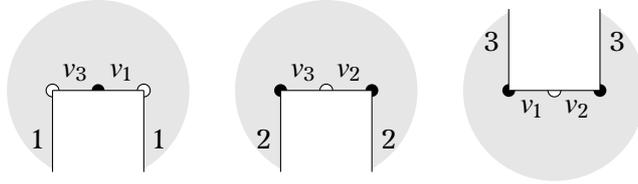
\begin{figure}[hbt]
\center
\begin{tikzpicture}[scale=1.2]
\begin{scope}[xshift=-2.5cm]
\fill[fill=black!10] (0,0)coordinate (Q)  circle (1cm);
    \coordinate (a) at (-1/2,0);
    \coordinate (b) at (1/2,0);

     \filldraw[fill=white] (a)  circle (2pt);
\filldraw[fill=white](b) circle (2pt);
\fill (Q) circle (2pt);
    \fill[white] (a) -- (b) -- ++(0,-1.1) --++(-1,0) -- cycle;
 \draw  (a) -- (b) coordinate[pos=.25](c) coordinate[pos=.75](d);
 \draw (a) -- ++(0,-.9) coordinate[pos=.6](e);
 \draw (b) -- ++(0,-.9)  coordinate[pos=.6](f);

\node[above] at (c) {$v_{3}$};
\node[above] at (d) {$v_{1}$};
\node[left] at (e) {$1$};
\node[right] at (f) {$1$};
    \end{scope}
\begin{scope}[xshift=0cm]
\fill[fill=black!10] (0,0)coordinate (Q)  circle (1cm);
    \coordinate (a) at (-1/2,0);
    \coordinate (b) at (1/2,0);

     \fill (a)  circle (2pt);
\fill (b) circle (2pt);
  \filldraw[fill=white](Q) circle (2pt);
    \fill[white] (a) -- (b) -- ++(0,-1.1) --++(-1,0) -- cycle;
 \draw  (a) -- (b) coordinate[pos=.25](c) coordinate[pos=.75](d);
 \draw (a) -- ++(0,-.9) coordinate[pos=.6](e);
 \draw (b) -- ++(0,-.9)  coordinate[pos=.6](f);

\node[above] at (c) {$v_{3}$};
\node[above] at (d) {$v_{2}$};
\node[left] at (e) {$2$};
\node[right] at (f) {$2$};
    \end{scope}

\begin{scope}[xshift=2.5cm]
\fill[fill=black!10] (0,0)coordinate (Q)  circle (1cm);
    \coordinate (a) at (-1/2,0);
    \coordinate (b) at (1/2,0);

     \fill (a)  circle (2pt);
\fill(b) circle (2pt);
\filldraw[fill=white] (Q) circle (2pt);
    \fill[white] (a) -- (b) -- ++(0,1.1) --++(-1,0) -- cycle;
 \draw  (a) -- (b) coordinate[pos=.25](c) coordinate[pos=.75](d);
 \draw (a) -- ++(0,.9) coordinate[pos=.6](e);
 \draw (b) -- ++(0,.9)  coordinate[pos=.6](f);

\node[below] at (c) {$v_{1}$};
\node[below] at (d) {$v_{2}$};
\node[left] at (e) {$3$};
\node[right] at (f) {$3$};
    \end{scope}
    \end{tikzpicture}
\caption{Différentielle de $\Omega^{2}\moduli[0](3,5;-4,-4,-4)$ dont les résidus sont $(1,1,1)$} \label{fig:residu35}
\end{figure}
\par
Nous montrons maintenant que la différentielle ainsi formée peut avoir les zéros d'ordres souhaités.
Cela se fait grâce au choix du type des domaines polaires (voir la section~\ref{sec:briques}). On peut remarquer que l'une des singularités coniques n'est adjacent qu'aux domaines polaires de la boucle tandis que l'autre est adjacent à tous les domaines polaires. On note $a_{1}$ l'ordre de la première singularité.
L'angle de cette singularité est $\pi+2s\pi$ auquel on ajoute des multiples de $2\pi$ correspondant aux types des domaines polaires qui contribuent à cette singularité. L'angle total de cette singularité est donc compris entre $\pi$ et $-\pi+\pi\sum_{j=1}^{p}(b_{j} -2)$. L'ordre de $a_{1}$ peut donc aller de $-1$ jusqu'à $-3+\sum_{j=1}^{p}(b_{j} -2)$. Ceci couvre tous les ordres possibles et donc toutes les strates considérées.
\smallskip
\par
\paragraph{\bf  $p$ est pair $\geq4$.}
Nous commençons par montrer que $(1,1,1,1)$ est dans l'image de l'application résiduelle pour les strates $\Omega^{2}\moduli[0](7,5;(-4^{4}))$, $\Omega^{2}\moduli[0](9,3;(-4^{4}))$, $\Omega^{2}\moduli[0](11,1;(-4^{4}))$ et $\Omega^{2}\moduli[0](13,-1;(-4^{4}))$.
La figure~\ref{fig:exeptpasexept} montre une différentielle dans $\Omega^{2}\moduli[0](7,5;(-4^{4}))$ qui possède ces résidus.
\begin{figure}[htb]
\begin{tikzpicture}

\begin{scope}[xshift=-6cm]
\fill[fill=black!10] (0,0)coordinate (Q)  circle (1.1cm);
    \coordinate (a) at (-.25,0);
    \coordinate (b) at (.25,0);
    \coordinate (c) at (-.75,0);
     \fill (a)  circle (2pt);
\fill[] (b) circle (2pt);
  \fill[] (c) circle (2pt);

    \filldraw[fill=white] (-.19,0)  arc (0:180:2pt);

    \fill[white] (a) -- (b) -- ++(0,-1.2) --++(-.5,0) -- cycle;
 \draw  (b)-- (a);
 \draw (a)-- (c) coordinate[pos=.5](d);
 \draw (a) -- ++(0,-1.05);
 \draw (b) -- ++(0,-1.05);

\node[above] at (Q) {$v_{2}$};
\node[above] at (d) {$v_{1}$};
\node[below] at (d) {$v_{3}$};
    \end{scope}

\begin{scope}[xshift=-3cm]
\fill[fill=black!10] (0,0)coordinate (Q)  circle (1.1cm);
    \coordinate (a) at (-.25,0);
    \coordinate (b) at (.25,0);
    \coordinate (c) at (-.75,0);

     \fill (a)  circle (2pt);
\fill[] (b) circle (2pt);
    \filldraw[fill=white] (-.19,0)  arc (0:-180:2pt);

    \fill[white] (a) -- (b) -- ++(0,1.2) --++(-.5,0) -- cycle;
 \draw  (a) -- (b);
 \draw (a) -- ++(0,1.05);
 \draw (b) -- ++(0,1.05);
 \draw (a)-- (c) coordinate[pos=.5](d);
  \filldraw[fill=white] (c) circle (2pt);

\node[below] at (Q) {$v_{2}$};
\node[below,rotate=180] at (d) {$v_{1}$};
\node[below] at (d) {$v_{4}$};
    \end{scope}

\begin{scope}[xshift=0cm]
\fill[fill=black!10] (0,0)coordinate (Q)  circle (1.1cm);
    \coordinate (a) at (-.25,0);
    \coordinate (b) at (.25,0);

     \fill (a)  circle (2pt);
\fill[] (b) circle (2pt);
    \fill[white] (a) -- (b) -- ++(0,-1.2) --++(-.5,0) -- cycle;
 \draw  (a) -- (b);
 \draw (a) -- ++(0,-1.05);
 \draw (b) -- ++(0,-1.05);

\node[above] at (Q) {$v_{3}$};
    \end{scope}

\begin{scope}[xshift=3cm]
\fill[fill=black!10] (0,0)coordinate (Q)  circle (1.1cm);
    \coordinate (a) at (-.25,0);
    \coordinate (b) at (.25,0);
  \filldraw[fill=white] (a) circle (2pt);
  \filldraw[fill=white] (b) circle (2pt);

    \fill[white] (a) -- (b) -- ++(0,-1.2) --++(-.5,0) -- cycle;
 \draw  (a) -- (b);
 \draw (a) -- ++(0,-1.05);
 \draw (b) -- ++(0,-1.05);

\node[above] at (Q) {$v_{4}$};
    \end{scope}

\end{tikzpicture}
\caption{Différentielle de $\Omega^{2}\moduli[0](7,5;(-4^{4}))$ dont les résidus sont $(1,1,1,1)$} \label{fig:exeptpasexept}
\end{figure}
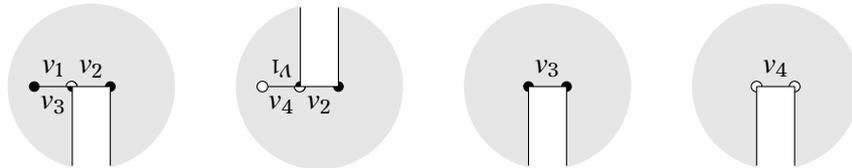
Pour les strates de la forme $\Omega^{2}\moduli[0](13-2k,-1+2k;(-4^{4}))$ avec $0 \leq k \leq 2$, nous collons des parties polaires d'ordre~$4$ le long de liens-selles selon les graphes d'incidence suivants. Ces graphes contiennent une boucle formée par $k$ sommets de valence $2$ qui est collée à un sommet de valence $3$ auquel est attaché une chaîne composée de $3-k$ sommets. Ces graphes sont représentés dans la figure~\ref{fig:graphesincidence}.
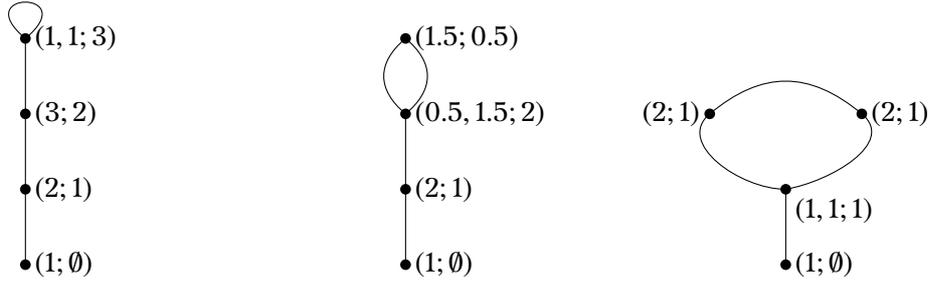
\begin{figure}[htb]
\begin{tikzpicture}
\begin{scope}[xshift=-5cm]
\coordinate (z1) at (0,0);\node[right] at (z1) {$(1,1;3)$};
\coordinate (z2) at (0,-1);\node[right] at (z2) {$(3;2)$};
\coordinate (z3) at (0,-2);\node[right] at (z3) {$(2;1)$};
\coordinate (z4) at (0,-3);\node[right] at (z4) {$(1;\emptyset)$};

\draw (z1) -- (z2) -- (z3) -- (z4);
\foreach \i in {1,2,3,4}
\fill (z\i) circle (2pt);
\draw (z1) .. controls ++(40:1) and ++(140:1) .. (z1);
\end{scope}

\begin{scope}[xshift=0cm]
\coordinate (z1) at (0,0);\node[right] at (z1) {$(1.5;0.5)$};
\coordinate (z2) at (0,-1);\node[right] at (z2) {$(0.5,1.5;2)$};
\coordinate (z3) at (0,-2);\node[right] at (z3) {$(2;1)$};
\coordinate (z4) at (0,-3);\node[right] at (z4) {$(1;\emptyset)$};

\draw (z2) -- (z3) -- (z4);
\foreach \i in {1,2,3,4}
\fill (z\i) circle (2pt);
\draw (z1) .. controls ++(-40:.5) and ++(40:.5) .. (z2);
\draw (z1) .. controls ++(-140:.5) and ++(140:.5) .. (z2);
\end{scope}

\begin{scope}[xshift=5cm]
\coordinate (z1) at (1,-1);\node[right] at (z1) {$(2;1)$};
\coordinate (z2) at (-1,-1);\node[left] at (z2) {$(2;1)$};
\coordinate (z3) at (0,-2);\node[below right] at (z3) {$(1,1;1)$};
\coordinate (z4) at (0,-3);\node[right] at (z4) {$(1;\emptyset)$};

\draw (z3) -- (z4);
\foreach \i in {1,2,3,4}
\fill (z\i) circle (2pt);
\draw (z1) .. controls ++(140:.9) and ++(40:.9) .. (z2);
\draw (z2) .. controls ++(-140:.6) and ++(180:.6) .. (z3);
\draw (z3) .. controls ++(10:.6) and ++(-40:.6) .. (z1);
\end{scope}
\end{tikzpicture}
\caption{Les graphes d'incidence pour les exemples dans les strates $\Omega^{2}\moduli[0](13-2k,-1+2k;(-4^{4}))$ avec $0 \leq k \leq 2$ de gauche à droite} \label{fig:graphesincidence}
\end{figure}
Les parties polaires correspondant aux sommets sont associés aux vecteurs donnés dans la  figure~\ref{fig:graphesincidence}. On vérifie sans difficulté que le recollement de ces parties polaires donne des différentielles ayant les singularités voulues.
\par
Dans le cas où tous les pôles sont d'ordres $-4$, on ajoute aux différentielles construites aux paragraphes précédents des parties polaires associées à $(v;v+1)$ soit à la boucle, soit à la chaîne. La première opération ajoute $2$ à l’ordre de chaque zéro et la seconde $4$ à l'ordre du zéro le plus grand. Cela permet d'obtenir toutes les strates considérées.
\smallskip
\par
Dans le cas où les pôles sont d'ordres $-b_{i}$, il suffit de coller des plans de manière cyclique à une demi-fente infinie partant de l'un des deux zéros. Un calcul similaire au cas où $p$ est impair montre que l'on peut obtenir tous les ordres des deux zéros.
\end{proof}

\begin{lem}\label{lem:polesdivksanszero}
L'image de $\appresk[0][2](a_{1},a_{2};-2\ell_{1},\ldots,-2\ell_{p})$ avec $p\geq3$ et $a_{1},a_{2} \geq 1$ contient tous les résidus~$(r_{1}^{2},\dots,r_{p}^{2})$ avec $r_{i}\in \RR_{>0}$.
\end{lem}

\begin{proof}
Le cas où les $r_{i}=1$ a été considéré dans le lemme précédent. Nous supposerons que les $r_{i}$ ne sont pas tous égaux entre eux.
\smallskip
\par

Commençons par le cas des strates $\quadomoduli[0](5,3;-4,-4,-4)$ et $\quadomoduli[0](7,1;-4,-4,-4)$ avec trois pôles d'ordre $-4$.

Pour la strate  $\quadomoduli[0](5,3;-4,-4,-4)$, on choisit $r_{3}>r_{2}-r_{1}\geq0$ et on considère les parties polaires $(r_{1};\emptyset)$, $(r_{2};\emptyset)$ et  $(r_{1};v,r_{2},v)$ avec $2v = r_{3}+r_{1}-r_{2}$. Les collages des $v$ et des $r_{i}$ entre eux donnent la différentielle avec les invariants souhaités.
\smallskip
\par
Dans le cas de la strate $\quadomoduli[0](7,1;-4,-4,-4)$, si $r_{3}>r_{1}+r_{2}$, on prend la partie polaire associée à $(v,r_{1}+r_{2},v;\emptyset)$, avec $2v=r_{3}-r_{1}-r_{2}$. Puis on colle les $v$ ensemble et les $r_{i}$ aux parties polaires $(r_{1};\emptyset)$ et $(r_{1};r_{1}+r_{2})$ pour obtenir la différentielle souhaitée. Si $r_{3}\leq r_{1}+r_{2}$, on peut supposer que les $r_{i}$ ne sont pas tous égaux entre eux car ce cas a été traité dans le lemme~\ref{lm:realquadcolib}. On peut donc supposer que $0<r_{2}-r_{1}<r_{3}$. On utilise alors les parties polaires $(r_{1},\emptyset)$, $(r_{1},r_{2}-r_{1};\emptyset)$ et $(v,r_{2}-r_{1},v;\emptyset)$ avec $2v+r_{2}-r_{1}=r_{3}$  pour obtenir la différentielle souhaitée.
\smallskip
\par
Considérons maintenant une strate $\quadomoduli[0](a_{1},a_{2};\rec[-4][p])$ avec $p\geq 4$.
Si $a_{1} \equiv 3 \mod 4$, alors on partitionne les résidus $r_{1},\dots,r_{p-1}$ en deux ensembles $A_{1}$ et $A_{2}$ tels que $A_{1}$ est de cardinal $m=\frac{a_{1}-1}{2}$ et  $r_{p}>\sum_{i\in A_{2}} r_{i}-\sum_{i\in A_{1}} r_{i}\geq0$. Cette partition existe car tous les $r_{i}$ ne sont pas égaux entre eux et comme $a_{1} \equiv 3 \mod 4$ on a $m\leq p-m-1$.

Quitte à changer la numérotation, on peut supposer que les résidus de $A_{1}$ sont les $m$ premiers résidus. On considère alors la partie polaire  $(r_{m+1},\dots,r_{p-1};v,r_{1},\dots,r_{m},v)$. La différentielle avec les invariants souhaités est obtenue en collant les parties polaires $(r_{i};\emptyset)$.

Si $a_{1} \equiv 1 \mod 4$, alors pour tout $r_{1},\dots,r_{p}>0$ on peut trouver une partition $A_{1}$ et $A_{2}$ des résidus tels que $A_{1}$ est de cardinal $m=\frac{a_{1}-1}{2}$ satisfaisant la condition suivante. Il existe un choix de signe tel que
$$S_{j}=\sum_{i\in A_{j}} \pm r_{i}>0 \text{ et } r_{p}>\sum_{i\in A_{2}} \pm r_{i}-\sum_{i\in A_{1}}\pm r_{i}\geq0\,.$$  On obtient une différentielle avec les invariants souhaités en généralisant la construction de la strate $\quadomoduli[0](7,1;-4,-4,-4)$. Plus précisément, on considère les résidus de $A_{1}$ de signe positif (supposons que ce sont les $k$ premiers). On note $S_{1}^{+}=\sum_{i\leq k} r_{i}$ la somme de ces résidus. La première partie polaire est $(r_{1};\emptyset)$, puis $(r_{2}+r_{1};r_{1})$, jusqu'à obtenir $(S_{1}^{+};\sum_{i< k} r_{i})$. Puis on considère les parties polaires $(S_{1}^{+}-r_{k+1};S_{1}^{+})$ jusqu'à obtenir  $(S_{1};S_{1}-r_{m})$. On fait de même pour les résidus de $A_{2}$ et on colle les liens-selles restants à la partie polaire $(v,S_{2},v;S_{1})$ pour obtenir la différentielle souhaitée.
\smallskip
\par
Les strates de la forme $\quadomoduli[0](a_{1},a_{2};-2\ell_{1},-2\ell_{2},\dots,-2\ell_{p})$ avec au moins un $\ell_{i}\geq3$ s'obtiennent via une variante des constructions précédentes. On suppose que $\ell_{1}\leq \ell_{2}\leq \cdots\leq \ell_{p}$. On commence par faire la construction du lemme~\ref{prop:5Negatif} où l'un des zéros est d'ordre $-1$ avec comme pôle spécial le pôle d'ordre $-2\ell_{p}$. On peut alors en changeant le type de la partie polaire obtenir tous les ordres entre $-1$ et $2\ell_{p}-5$. On peut alors faire la construction similaire à celle de la strate $\quadomoduli[0](7,1;-4,-4,-4)$ pour obtenir les pôles d'ordre $2\ell_{p}-5$. On continue alors à partir de la construction de la strate $\quadomoduli[0](5,3;-4,-4,-4)$. On obtient ainsi tous les ordres souhaités.
\end{proof}

\subsubsection{Cas général pour $\mathfrak{i}=n=2$}

La construction des différentielles à résidus quadratiques prescrits se fait désormais par récurrence. Nous prouvons que les seules obstructions portant sur les configurations non uniformément nulles sont les obstructions décrites dans les points~(6) et~(7) du corollaire~\ref{cor:except1}, la proposition~\ref{prop:excep2} et le corollaire~\ref{cor:excep2}.

\begin{prop}\label{prop:MAIN5+n2}
Toute configuration $(R_{1},\dots,R_{p+s})$ de résidus quadratiques réels positifs non tous nuls (et non nuls pour $i\geq p+1$) est réalisable dans $\quadomoduli[0](a_{1},a_{2};-b_{1},\dots,-b_{p},\rec[-2][s])$ sauf dans les cas suivants:
\begin{enumerate}
    \item les configurations $\mathbb{C}^{\ast}\cdot (0;1,\dots,1)$ dans $\quadomoduli[0](2l-1,2l+1;-4,\rec[-2][2l])$ avec $l \geq 1$;
    \item les configurations  $\mathbb{C}^{\ast}\cdot (1,\dots,1)$ dans $\quadomoduli[0](2l+1,2l+1;-4,\rec[-2][2l+1])$ avec $l \geq 0$;
    \item les configurations  $\mathbb{C}^{\ast}\cdot (\rec[0][a],1,1)$ dans $\quadomoduli[0](2a+b-1,2a+b-1;\rec[-4][a],-b,-b-2)$ avec $b \geq 2$ pair et $a \geq 0$;
    \item les configurations  $\mathbb{C}^{\ast}\cdot (\rec[0][a],1,1)$ dans $\quadomoduli[0](2a+b-1,2a+b-3;\rec[-4][a],-b,-b)$ avec $b \geq 2$ pair, $a \geq 0$ et $2a+b \geq 4$.
\end{enumerate}
\end{prop}

\begin{proof}
Nous procédons par récurrence sur le nombre $t$ de pôles non doubles pour lesquels le résidu est nul. La proposition~\ref{prop:MAIN5} établit le cas $t=0$. En supposant la propriété valide jusqu'au rang $t$, nous allons la démontrer pour le cas $t+1$.
\par
Soit une configuration formée de $t+1$ résidus nuls et de $p+s-(t+1)$ résidus non nuls $R_{t+2},\dots,R_{p+s}$. Il s'agit de construire une différentielle ayant ces résidus dans une certaine strate $\quadomoduli[0](a_{1},a_{2};-b_{1},\dots,-b_{p};\rec[-2][s])$. Nous supposerons naturellement que la configuration n'est pas interdite dans la strate concernée.
\par
La proposition~\ref{prop:5Negatif} permet de se restreindre aux strates pour lesquelles $a_{1},a_{2}>0$. Ainsi, il existe des entiers $m_{1},m_{2}>0$ tels que $2m_{1}+2m_{2}=b_{t+1}$ et $a_{i}-2m_{i}>-2$ pour $i=1,2$.
\par
Le cas $p=t=1$ a été réglé par la proposition~\ref{prop:5p=1} (dans laquelle $R_{1}$ peut prendre une valeur nulle). Dans les autres cas, l'hypothèse de récurrence implique que la configuration avec un résidu nul de moins est réalisable dans la strate dont les zéros sont d'ordres $a_{1}-2m_{1},a_{2}-2m_{2}$ et le pôle d'ordre $b_{t+1}$ a été retiré. Dans la surface plate définie par la différentielle, il existe un lien-selle reliant les deux zéros. Nous découpons le long de ce lien-selle pour y coller les deux bords d'une cicatrice découpée dans un domaine polaire d'ordre~$b_{t+1}$ et pour lesquelles les angles aux extrémités de la cicatrice sont $2m_{1}\pi$ et $2m_{2}\pi$. Cette chirurgie produit une différentielle réalisant une configuration ayant un résidu nul de plus dans la strate adéquate.
\par
Le cas des configurations soumises à des obstructions ne pose aucun problème car il s'agit toujours d'ajouter ou retirer des pôles quadruples à résidu nul. On a donc $m_{1}=m_{2}=1$. La différence entre les ordres des deux zéros demeure identique tout au long de la récurrence.
\end{proof}

\subsection{Strates avec $n \geq 3$ zéros dont $\mathfrak{i}=2$ impairs}

Dans pratiquement tous les cas, la construction est un simple éclatement de zéros. Certaines familles de strates nécessitent néanmoins une construction spécifique.

\begin{prop}\label{prop:5Except1}
L'application résiduelle de $\quadomoduli[0](b+2a-3,b+2a-3,2;\rec[-4][a],-b,-b)$ avec $b \geq 2$ pair, $a \geq 0$ et $b+2a \geq 4$ contient $(\rec[0][a],1,1)$.
\end{prop}

\begin{proof}
On considère deux cas selon la parité de~$a$.
Si $a$ est impair, la construction est représentée à gauche de la figure~\ref{fig:extroiszero} dans le cas où $a=1$ et $b=2$. On associe à un pôle de résidu nul la partie polaire $(v_{3},v_{1},v_{3};v_{4},v_{2},v_{4})$ avec $v_{i}=1$. La moitié des autres pôles à résidu nul sont associés à $(v_{1};v_{1})$ et l'autre moitié à $(v_{2};v_{2})$. Enfin, aux pôles d'ordre $-b$ sont associées les parties polaires $(v_{1};\emptyset)$ et $(v_{2};\emptyset)$.  Puis on colle cycliquement par translation les parties polaires associées à  $(v_{i};v_{i})$ à la partie polaire spéciale et enfin on colle les parties polaires $(v_{i};\emptyset)$.  Si $a$ est pair et non nul, on procède de manière similaire en partant de la configuration à droite de la figure~\ref{fig:extroiszero}.
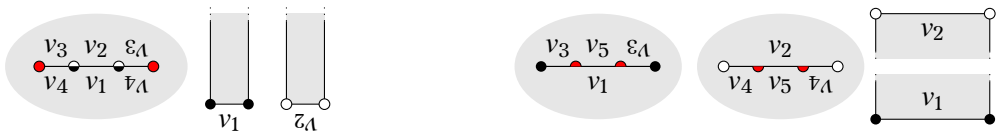
\begin{figure}[hbt]
\center
\begin{tikzpicture}

\begin{scope}[xshift=-4cm]
    \fill[fill=black!10] (0,0) ellipse (1.2cm and .7cm);

\draw[] (-.75,0)coordinate (Q) -- (.75,0) coordinate (P) coordinate[pos=.5](c) coordinate[pos=.15](d) coordinate[pos=.85](e) coordinate[pos=.35](R) coordinate[pos=.65](S);
\draw[] (Q) -- (P);

\filldraw[fill=white] (R)  arc (0:180:2pt); 
\filldraw[fill=white] (S)  arc (180:0:2pt); 
\fill (R)  arc (0:-180:2pt); 
\fill (S)  arc (-180:0:2pt);

\filldraw[fill=red] (P) circle (2pt);
\filldraw[fill=red] (Q) circle (2pt);

\node[above] at (c) {$v_{2}$};
\node[below] at (c) {$v_{1}$};
\node[above] at (d) {$v_{3}$};
\node[below] at (d) {$v_{4}$};
\node[below, rotate=180] at (e) {$v_{3}$};
\node[above, rotate=180] at (e) {$v_{4}$};

\begin{scope}[xshift=2cm]
\coordinate (a) at (-.5,-.5);
\coordinate (b) at (0,-.5);

    \fill[fill=black!10] (a)  -- (b)coordinate[pos=.5](f) -- ++(0,1.2) --++(-.5,0) -- cycle;
    \fill (a)  circle (2pt);
\fill[] (b) circle (2pt);
 \draw  (a) -- (b);
 \draw (a) -- ++(0,1.1) coordinate (d)coordinate[pos=.5](h);
 \draw (b) -- ++(0,1.1) coordinate (e)coordinate[pos=.5](i);
 \draw[dotted] (d) -- ++(0,.2);
 \draw[dotted] (e) -- ++(0,.2);
\node[below] at (f) {$v_{1}$};
\end{scope}

\begin{scope}[xshift=3cm]
\coordinate (a) at (-.5,-.5);
\coordinate (b) at (0,-.5);

    \fill[fill=black!10] (a)  -- (b)coordinate[pos=.5](f) -- ++(0,1.2) --++(-.5,0) -- cycle;
 \draw  (a) -- (b);
 \draw (a) -- ++(0,1.1) coordinate (d)coordinate[pos=.5](h);
 \draw (b) -- ++(0,1.1) coordinate (e)coordinate[pos=.5](i);
 \draw[dotted] (d) -- ++(0,.2);
 \draw[dotted] (e) -- ++(0,.2);
 \filldraw[fill=white] (a)  circle (2pt); 
\filldraw[fill=white] (b)  circle (2pt); 
\node[above, rotate=180] at (f) {$v_{2}$};
\end{scope}
\end{scope}

\begin{scope}[xshift=5cm]
     \fill[fill=black!10] (0,0) ellipse (1.1cm and .7cm);

\draw[] (-.75,0)coordinate (Q) -- (.75,0) coordinate (P) coordinate[pos=.5](c) coordinate[pos=.15](d) coordinate[pos=.85](e) coordinate[pos=.35](R) coordinate[pos=.65](S);
\draw[] (Q) -- (P);

\filldraw[fill=red] (R)  arc (0:-180:2pt); 
\filldraw[fill=red] (S)  arc (-180:0:2pt);

\filldraw[fill=white] (P) circle (2pt);
\filldraw[fill=white] (Q) circle (2pt);

\node[above] at (c) {$v_{2}$};
\node[below] at (c) {$v_{5}$};
\node[below] at (d) {$v_{4}$};
\node[above, rotate=180] at (e) {$v_{4}$};

\begin{scope}[xshift=-2.4cm]
    \fill[fill=black!10] (0,0) ellipse (1.1cm and .7cm);

\draw[] (-.75,0)coordinate (Q) -- (.75,0) coordinate (P) coordinate[pos=.5](c) coordinate[pos=.15](d) coordinate[pos=.85](e) coordinate[pos=.35](R) coordinate[pos=.65](S);
\draw[] (Q) -- (P);

\filldraw[fill=red] (R)  arc (0:180:2pt); 
\filldraw[fill=red] (S)  arc (180:0:2pt);

\fill (P) circle (2pt);
\fill (Q) circle (2pt);

\node[above] at (c) {$v_{5}$};
\node[below] at (c) {$v_{1}$};
\node[above] at (d) {$v_{3}$};
\node[below, rotate=180] at (e) {$v_{3}$};
\end{scope}

\begin{scope}[xshift=2cm,yshift=-.7cm]
\coordinate (a) at (-.75,0);
\coordinate (b) at (.75,0);

    \fill[fill=black!10] (a)  -- (b)coordinate[pos=.5](f) -- ++(0,.6) --++(-1.5,0) -- cycle;
    \fill (a)  circle (2pt);
\fill[] (b) circle (2pt);
 \draw  (a) -- (b);
 \draw (a) -- ++(0,.5) coordinate (d)coordinate[pos=.5](h);
 \draw (b) -- ++(0,.5) coordinate (e)coordinate[pos=.5](i);
 \draw[dotted] (d) -- ++(0,.15);
 \draw[dotted] (e) -- ++(0,.15);
\node[above] at (f) {$v_{1}$};
\end{scope}

\begin{scope}[xshift=2cm,yshift=.7cm]
\coordinate (a) at (-.75,0);
\coordinate (b) at (.75,0);

    \fill[fill=black!10] (a)  -- (b)coordinate[pos=.5](f) -- ++(0,-.6) --++(-1.5,0) -- cycle;

 \draw  (a) -- (b);
 \draw (a) -- ++(0,-.5) coordinate (d)coordinate[pos=.5](h);
 \draw (b) -- ++(0,-.5) coordinate (e)coordinate[pos=.5](i);
 \draw[dotted] (d) -- ++(0,-.15);
 \draw[dotted] (e) -- ++(0,-.15);
     \filldraw[fill=white](a)  circle (2pt);
   \filldraw[fill=white](b) circle (2pt);
\node[below] at (f) {$v_{2}$};
\end{scope}
\end{scope}
\end{tikzpicture}
\caption{Différentielles dans les strates $\Omega^{2}\moduli[0](1,1,2;-4;-2,-2)$ (à gauche) et $\Omega^{2}\moduli[0](3,3,2;-4,-4;-2,-2)$ (à droite) dont les résidus sont respectivement $(0;1,1)$ et $(0,0;1,1)$.} \label{fig:extroiszero}
\end{figure}
\par
Enfin, si $a=0$, on a $b \geq 4$. La figure~\ref{fig:deuxdiffresunun} montre une différentielle quadratique de $\Omega^{2}\moduli[0](1,1,2;-4,-4)$ dont les résidus quadratiques sont $(1,1)$.
\begin{figure}[hbt]
\center
\begin{tikzpicture}
\begin{scope}[xshift=-2cm]
\fill[fill=black!10] (1,0)coordinate (Q)  circle (1.2cm);
    \coordinate (a) at (0,0);
    \coordinate (b) at (2,0);

     \filldraw[fill=white] (a)  circle (2pt);
\filldraw[fill=white](Q) circle (2pt);
\filldraw[fill=red] (1.07,0)  arc (0:180:2pt); 
\fill (b) circle (2pt);
    \fill[white] (a) -- (Q) -- ++(0,-1.2) --++(-1,0) -- cycle;
 \draw  (a) -- (b) coordinate[pos=.25](c) coordinate[pos=.75](d);
 \draw (a) -- ++(0,-.5);
 \draw (Q) -- ++(0,-1.1);

\node[above] at (c) {$1$};
\node[above] at (d) {$2$};
\node[below] at (d) {$3$};

    \end{scope}
\begin{scope}[xshift=2cm]
\fill[fill=black!10] (1,0)coordinate (Q)  circle (1.2cm);
    \coordinate (a) at (0,0);
    \coordinate (b) at (2,0);

     \filldraw[fill=red] (a)  circle (2pt);
\filldraw[fill=red](Q) circle (2pt);
\filldraw[fill=white] (1.07,0)  arc (0:180:2pt); 
\fill (b) circle (2pt);
    \fill[white] (a) -- (Q) -- ++(0,-1.2) --++(-1,0) -- cycle;
 \draw  (a) -- (b) coordinate[pos=.25](c) coordinate[pos=.75](d);
 \draw (a) -- ++(0,-.5);
 \draw (Q) -- ++(0,-1.1);

\node[below,rotate=180] at (c) {$1$};
\node[above] at (d) {$3$};
\node[below] at (d) {$2$};

    \end{scope}
    \end{tikzpicture}
\caption{Différentielle quadratique de $\Omega^{2}\moduli[0](1,1,2;-4,-4)$ dont les résidus sont $(1,1)$} \label{fig:deuxdiffresunun}
\end{figure}
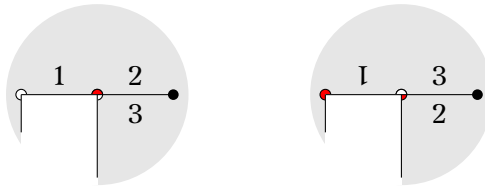
Le cas des strates de la forme $\Omega^{2}\moduli[0](b-3,b-3,2;-b,-b)$ s'obtient de manière similaire en utilisant des parties polaires d'ordre $b$ et de type $b-1$ associés aux mêmes vecteurs que dans la figure~\ref{fig:deuxdiffresunun}.
\end{proof}

\begin{prop}\label{prop:5Except2}
L'image de l'application résiduelle de $\quadomoduli[0](2a-1,2a-1,2;-4;\rec[-2][2a])$ avec $a \geq 1$ contient $(0;1,\dots,1)$.
\end{prop}

\begin{proof}
On procède à la construction suivante: on associe aux pôles d'ordres $-2$ les parties polaires d'ordre~$2$ associées à~$(1;\emptyset)$. Pour le pôle d'ordre~$-4$, on associe la partie polaire d'ordre $4$ associée à $(\rec[1][a+2];\rec[v][a+2])$. La différentielle est obtenue en collant par rotation d'angle $\pi$ le premier vecteur~$v$, resp. $1$, de cette partie polaire au dernier vecteur $v$, resp. $1$, puis les autres vecteurs aux bords des parties polaires d'ordre~$2$. Cette construction est représentée à gauche de la figure~\ref{fig:extroiszero}.
\end{proof}

Nous donnons maintenant la preuve du cas général.

\begin{prop}\label{prop:5i2n3}
Toute configuration $(R_{1},\dots,R_{p+s})$ de résidus quadratiques non tous nuls (et non nuls pour $R_{p+1},\dots,R_{p+s}$) est réalisable dans la strate $\quadomoduli[0](a_{1},\dots,a_{n};-b_{1},\dots,-b_{p};\rec[-2][s])$ avec $n \geq 3$, $a_{1},a_{2}$ impairs et $a_{3},\dots,a_{n}$ pairs.
\end{prop}

\begin{proof}
Considérons tout d'abord les strates  $\quadomoduli[0](a_{1},a_{2},2l_{3};-b_{1},\dots,-b_{p},\rec[-2][s])$.
Si une configuration est réalisable dans $\quadomoduli[0](a_{1}+2l_{3},a_{2};\rec[-2][s])$ ou $\quadomoduli[0](a_{1},a_{2}+2l_{3};\rec[-2][s])$, alors l'éclatement de zéros permet de le réaliser dans $\quadomoduli[0](a_{1},a_{2},2l_{3};\rec[-2][s])$. Rappelons que pour ces strates vérifiant $n=\mathfrak{i}=2$, la remarque~\ref{rem:redux} permet de caractériser les obstructions portant sur chaque strate résiduelle (voir section~\ref{sec:arrhyp}) en n'ayant à étudier que les configurations dans lesquelles les résidus quadratiques sont des réels positifs. On déduit de la proposition~\ref{prop:MAIN5+n2} que les seuls cas dans lesquels nous obtenons une obstruction sont:
\begin{enumerate}
    \item les configurations $\mathbb{C}^{\ast}\cdot (\rec[0][a],1,1)$ dans $\quadomoduli[0](2a+b-1,2a+b-1,-2;\rec[-4][a],-b,-b)$ avec $b \geq 2$ pair, $a \geq 0$ et $2a+b \geq 4$;
    \item les configurations $\mathbb{C}^{\ast}\cdot (0;1,\dots,1)$ dans les strates $\quadomoduli[0](2a-1,2a-1,2;-4,\rec[-2][2a])$ avec $a \geq 1$.
\end{enumerate}
Des constructions ont été données pour ces cas dans les propositions~\ref{prop:5Except1} et ~\ref{prop:5Except2}.
\end{proof}

\subsection{Au moins quatre zéros d'ordre impair}\label{sub:pair4impair}

Toutes les configurations qui ne sont pas uniformément nulles peuvent s'obtenir dans les strates $\quadomoduli[0](a_{1},\dots,a_{n};-b_{1},\dots,-b_{p},\rec[-2][s])$ par éclatement de zéro avec $\mathfrak{i} \geq 4$ en partant d'une strate vérifiant $n=3$ et $\mathfrak{i}=2$. Le lemme suivant traite du cas des configurations uniformément nulles.

\begin{lem}
Si la décomposition $\mu=(a_{1},\dots,a_{n};-b_{1},\dots,-b_{p})$ contient au moins quatre zéros impairs, alors l'application résiduelle $\appresquad[0](\mu)$ contient $(0,\dots,0)$.
\end{lem}

\begin{proof}
Considérons le cas d'une strate avec quatre zéros d'ordre impair. Le cas avec davantage de zéros d'ordres impairs se déduit par éclatement. Le lemme~\ref{lem:condsufngeq3} montre que la configuration uniformément nulle est réalisable par une différentielle quadratique avec deux zéros d'ordres impairs $a_{1}$ et $a_{2}$ et un zéro d'ordre $a_{3}$ pair vérifiant $a_{3} \geq 2p$. L'éclatement de ce dernier zéro permet d'obtenir toutes les strates avec quatre zéros d'ordre impair, sauf si la somme des ordres de deux d'entre eux est  $<2p$. Les pôles étant d'ordres $b_{i}\geq 4$, les seules exceptions sont les strates $\Omega^{2}\moduli[0](\rec[p-1][4];\rec[-4][p])$ avec $p$ pair.
\par
Dans le cas $\Omega^{2}\moduli[0](1,1,1,1;-4,-4)$ on considère deux parties polaires d'ordre $4$ associées à $(v_{1},v_{2};v_{3},v_{4})$ et à $(v_{4},v_{2};v_{3},v_{1})$ respectivement, avec $v_{i}=1$. On colle les segments de même nom ensemble pour obtenir la différentielle quadratique souhaitée. Enfin pour obtenir les différentielles dans les strates $\Omega^{2}\moduli[0](\rec[p-1][4];\rec[-4][p])$ avec $p$ pair, on considère les parties polaires d'ordre $4$ associées à $(1;1)$. Puis on colle cycliquement la moitié de ces parties polaires aux segments $v_{1}$ et l'autre moitié aux segments $v_{3}$.
\end{proof}

\section{Différentielles dont les pôles sont doubles}
\label{sec:juste-k}

Comme pour les différentielles abéliennes (voir \cite[section 3.3]{getaab}), ce cas est subtil. Dans les strates ayant un unique zéro, toutes les singularités ont un ordre pair. Elles ne sont donc pas primitives (lemme~\ref{lem:puissk}). Nous considérerons les strates avec $n\geq 2$ zéros, dont au moins deux sont impairs. 
\smallskip
\par
\subsection{Obstructions}
\label{sec:doubleobstr}

Deux familles de configuration de résidus à deux paramètres présentent des obstructions exceptionnelles:
\begin{defn}\label{def:crosse}
Les $s$-uplets proportionnels à $(1,\dots,1,R,R)$ avec $R\in\CC^{\ast}$ et $s$ pair sont {\em en crosse} tandis que les $s$-uplets de la forme $(\lambda^{2},\mu^{2},\nu^{2},\dots,\nu^{2})$ avec $s$ impair et $\lambda\mu\nu \in \CC^{\ast}$ vérifiant $\lambda+\mu+\nu=0$ sont {\em triangulaires}.
\end{defn}

Les cas (3) et (4) du corollaire~\ref{cor:except1} prouvent que les configurations en crosse et triangulaires sont impossibles à réaliser dans les strates de la forme $\Omega^{2}\moduli[0](s-1,s-3;\rec[-2][s])$ et $\Omega^{2}\moduli[0](s-2,s-2;\rec[-2][s])$ respectivement.
Outre ces deux obstructions, les configurations de résidus doivent aussi satisfaire une remarquable obstruction arithmétique pour être réalisables dans une strate de différentielles quadratiques à pôles doubles.

\begin{defn}\label{def:arithmétique}
Un $s$-uplet $(R_{1},\dots,R_{s})\in \NN^{s}$ est {\em arithmétique} si $R_{i} = r_{i}^{2}$ avec $r_{i}\in \NN$ premiers entre eux.
\end{defn}

Il existe deux obstructions distinctes concernant les configurations arithmétiques selon si la somme des racines $r_{i}$ est paire ou impaire. Ces obstructions sont les inégalités~\eqref{eq:sumimpintro} et~\eqref{eq:sumpairintro}  du théorème~\ref{thm:geq0quad2}. La fin de cette section est dédiée à la preuve de la nécessité de ces conditions.

\begin{prop}\label{prop:ObstrArithm}
Soit $\xi$ une différentielle  de $\quadomoduli[0](a_{1},a_{2},2l_{3},\dots,2l_{n};\rec[-2][s])$ avec $a_{1},a_{2}$ impairs et $l_{3},\dots,l_{n} \geq 1$ dont les résidus $(R_{1},\dots,R_{s})$ sont arithmétiques.
 \begin{itemize}
  \item  Si la somme $\sum r_{i}$ est impaire, alors elle est supérieure ou égale à $a_{2}+2$.
  \item  Si la somme $\sum r_{i}$ est paire, alors elle est supérieure ou égale à $a_{1}+a_{2}+4$.
 \end{itemize}
\end{prop}

\begin{proof}
L'application développante $dev$ de la structure plate de $\xi$ est une fonction multivaluée sur la surface privée de ses singularités. Toutefois, l'holonomie de $\xi$ est engendrée par les fonctions $z \mapsto z+1$ et $z \mapsto -z$. Le quotient de $\CC$ par ces transformations est une surface de demi-translation : le demi-cylindre infini $\mathcal{X}$ avec deux singularités coniques $A,B$ d'angle $\pi$ situées en $0$ et $\frac{1}{2}$ et un pôle double $C$. Ainsi, l'application $dev$ définit un revêtement ramifié $\pi$ de $(\mathbb{CP}^{1},\xi)$ vers $\mathcal{X}$.

La fibre au-dessus du pôle double $C$ est constituée de tous les pôles doubles de $\xi$. Le degré du revêtement est donc la somme $\sum r_{i}$. De plus, les zéros d'ordres impairs $a_{1}$ et $a_{2}$ ont un angle conique qui est un multiple impair de $\pi$. Leur image par le revêtement $\pi$ est donc nécessairement l'une des deux singularités $A$ ou $B$.

Dans la fibre au-dessus de $A$ (ou $B$), les points d'ordre pair (incluant les points réguliers de~$\xi$) ont un ordre de ramification qui est pair également. Autrement dit, si $\sum r_{i}$ est impair, la fibre au-dessus de $A$ contient un nombre impair de points d'ordre impair. Dans le cas présent, elle en contient exactement un. Le degré du revêtement borne donc l'ordre de ramification en ce point. Nous avons donc $\sum r_{i} \geq \max(a_{1},a_{2})+2$.

Si au contraire $\sum r_{i}$ est pair, la fibre au-dessus de $A$ contient un nombre pair de points dont l'ordre de ramification est impair. Il s'ensuit que les singularités d'ordres $a_{1}$ et $a_{2}$ sont projetées sur le même point ($A$ ou $B$) par le revêtement $\pi$ et nous obtenons l'inégalité $\sum r_{i} \geq a_{1}+a_{2}+4$.
\end{proof}

\begin{rem}\label{rem:longueurLS}
La construction du revêtement ramifié prouve en particulier que la longueur des liens-selles reliant les deux zéros d'ordres impairs se trouve dans $\mathbb{N}^{\ast}$ si $\sum r_{i}$ est pair et dans $\frac{1}{2}+\mathbb{N}$ si $\sum r_{i}$ est impair.
\end{rem}

\subsection{Le cas de strates avec deux zéros}\label{sub:i2n2double}

Pour des strates telles que $n=\mathfrak{i}=2$ (autrement dit, pour lesquelles les seules singularités coniques sont deux zéros d'ordre impair), la remarque~\ref{rem:redux} permet de caractériser les obstructions portant sur chaque strate résiduelle (voir section~\ref{sec:arrhyp}) en n'ayant à étudier que les configurations dans lesquelles les résidus quadratiques sont des réels positifs.

Dans cette section, nous donnerons les résultats d'existence de différentielles à résidus prescrits pour des $s$-uplets $R_{1},\dots,R_{s} \in \mathbb{R}_{+}^{\ast}$ de résidus quadratiques. Nous noterons $r_{1},\dots,r_{s}$ leurs racines positives respectives. De plus, si tous ces nombres sont commensurables, nous les normaliserons de telle façon à ce que les $r_{1},\dots,r_{s}$ soient des entiers premiers entre eux (voir la définition~\ref{def:arithmétique} des configurations arithmétiques).

Les surfaces plates correspondant à de telles différentielles sont formées de cylindres horizontaux de circonférences $r_{1},\dots,r_{s}$ reliés par exactement $s$ liens-selles horizontaux définissant un graphe d'incidence ayant un unique cycle (voir section~\ref{sec:coeur}). Ces différentielles sont en particulier des différentielles de Strebel. Dans chaque cas, nous donnerons le graphe d'incidence et l'attribution de chaque circonférence de cylindre au sommet du graphe qui lui correspond. Le lemme suivant sera implicitement employé dans le reste de la section pour établir que les différentielles construites ont bien les ordres des zéros voulus.

\begin{lem}\label{lem:ordresgraphes}
Soit  $\xi$ une différentielle sur $\mathbb{CP}^{1}$ constituée de $s$ cylindres horizontaux reliés par exactement $s$ liens-selles. Son graphe d'incidence est un graphe plongé dans la sphère
 formé d'un unique cycle de longueur impaire $t_{0}$ sur lequel sont collés des arbres. Soit $t_{1}$ et $t_{2}$ le nombre des sommets des arbres appartenant à chacune des composantes connexes découpées par le cycle. La différentielle $\xi$ compte deux zéros d'ordres respectifs $t_{0}+2t_{1}-2$ et $t_{0}+2t_{2}-2$.
\end{lem}

\begin{proof}
La section~\ref{sec:coeur} et un calcul de caractéristique d'Euler montrent qu'une telle surface plate compte exactement deux singularités coniques. Puisque la surface est constituée de cylindres horizontaux, chaque angle entre deux liens-selles est égal à $\pi$. Nous déduisons du comptage des coins de chacune des deux faces dans le graphe d'incidence que les deux singularités coniques sont d'angles respectifs $t_{0}+2t_{1}$ et $t_{0}+2t_{2}$. Finalement, le nombre de cylindres du cycle est impair sans quoi l'holonomie linéaire de la surface serait triviale (et la différentielle non primitive).
\end{proof}

\begin{rem}\label{rem:LongueurLS}
Les longueurs des différents liens-selles sont entièrement déterminées par la position de l'arête correspondante dans le graphe d'incidence.
En effet, un lien-selle reliant un zéro à lui-même sépare le graphe en deux composante connexe. L'une est un arbre $\mathcal{A}$ (dont la racine est le sommet incident au lien-selle) tandis que l'autre contient un cycle. La longueur du lien-selle est donnée par la somme des circonférences des cylindres de $\mathcal{A}$, chacune étant assortie d'un signe positif si la distance à la racine est paire, négatif dans le cas contraire.
\end{rem}

\subsubsection{Si l'un des zéros est d'ordre négatif}

Nous donnons d'abord une construction géométrique du cas générique.

\begin{lem}\label{lem:polesimple}
Toute configuration de la forme $(r_{1}^{2},\dots,r_{s}^{2})$ avec $r_{1} < r_{2} \leq r_{3} \leq \dots \leq r_{s}$ est réalisable par une différentielle de la strate $\quadomoduli[0](-1,2s-3;\rec[-2][s])$.
\end{lem}

\begin{proof}
Nous donnons une construction géométrique directe: le graphe est une chaîne de $s$ sommets avec un sommet de valence $1$ à un bout et un sommet de valence $3$ avec une boucle à l'autre bout (et $s-2$ sommets de valence $2$ entre les deux). Les circonférences $r_{1},\dots,r_{s}$ sont ordonnées le long du graphe avec le plus petit cylindre sur le sommet de valence $1$ et le plus grand sur le sommet de valence $3$. Il suffit de vérifier que ces conditions permettent toujours d'obtenir des liens-selles de longueur strictement positive. En fait, la longueur du lien-selle entre les cylindres de circonférences $r_{i}$ et $r_{i+1}$ est la somme alternée $\sum\limits_{k=1}^{i} (-1)^{i-k} r_{i}$ (pour $1 \leq i \leq s-1$). Comme les $r_{i}$ sont croissants et que $r_{2}>r_{1}$, toutes ces sommes seront non nulles. Les deux segments identifiés par rotation sur le domaine polaire de valence $3$ (du fait de la boucle dans le graphe) ont pour longueur $\frac{1}{2}\sum\limits_{k=0}^{s-1}(-1)^{k}r_{s-k} > 0$.
\end{proof}

La caractérisation complète se démontre par récurrence sur le nombre de pôles doubles.

\begin{prop}\label{prop:poledouble-1}
La configuration $(r_{1}^{2},\dots,r_{s}^{2})$ est réalisable dans $\quadomoduli[0](-1,2s-3;\rec[-2][s])$ si et seulement s'il ne s'agit pas d'une configuration arithmétique de somme $T=\sum r_{i} \leq 2s-2$.
\end{prop}

\begin{proof}
Le lemme~\ref{lem:polesimple} donne les constructions pour $s=1$ et $s=2$ sauf pour les configurations proportionnelles à $(1,1)$. Celles-ci ne sont pas réalisables car nous obtiendrions une différentielle entrelacée lissable dans le bord de $\quadomoduli[1](-1,1)$ en collant les pôles doubles.
\par
Supposons la proposition valide jusqu'au rang $s\geq2$. Prenons un uplet $C$ de $(s+1)$ nombres réels $r_{i}$ classés par ordre croissant. Soit ils sont non commensurables, soit la proposition~\ref{prop:ObstrArithm} implique que ce sont des entiers premiers entre eux de somme $T \geq 2s + 1$. Comme les éléments de $C$ ne sont pas tous identiques, on a $r_{s+1}>r_{1}$. Soit $C'$ la configuration obtenue à partir de~$C$ en éliminant $r_{1}$ et en remplaçant $r_{s+1}$ par $r_{s+1}-r_{1}$.
\par
Supposons qu'il existe une différentielle de $\quadomoduli[0](-1,2s-3;\rec[-2][s])$ dont les racines des résidus sont les éléments de $C'$. On grossit le cylindre de taille $r_{s+1}-r_{1}$  à partir du zéro d'ordre $2s-3$ jusqu'à obtenir un cylindre de taille de~$r_{s+1}$ à bord de taille $r_{1}$. On colle alors un cylindre de circonférence $r_{1}$ à ce bord. La surface obtenue appartient à la strate voulue et ses résidus sont les carrés des éléments de~$C$.
\par
Supposons à présent que les carrés des éléments de $C'$ n’apparaissent pas comme résidus dans la strate $\quadomoduli[0](-1,2s-3;\rec[-2][s])$. La configuration $C'$ est donc proportionnelle à une configuration arithmétique de somme au plus $2s-2$. Soit $k$ le pgcd des éléments de $C'$, la somme des éléments de $C'$ est strictement inférieure à $k(2s-1)$ et $C'$ compte donc au moins deux éléments égaux à $k$. Nous avons donc $r_{2}=r_{3}=k$. Si $k>r_{1}$, alors le lemme~\ref{lem:polesimple} prouve que $C$ est réalisable. Si $k=r_{1}$, alors $k$ divise chaque élément de $C$ et donc $k=1$. Puisque la somme des éléments de $C'$ vaut au plus $2s-2$, la somme des éléments de $C$ vaut au plus $2s-1$ et nous avons une contradiction.
\end{proof}

\subsubsection{Constructions génériques}

Nous étendons la construction donnée dans le lemme~\ref{lem:polesimple} afin d'obtenir la plupart des configurations de résidus réels positifs dans les strates vérifiant $\mathfrak{i}=n=2$. Nous distinguons plusieurs cas selon si les ordres des zéros se trouvent dans $4\mathbb{Z}+1$ ou $4\mathbb{Z}+3$. Dans toute cette section, les résidus sont classés par ordre croissant.

\begin{lem}\label{lem:4z+3}
Toute configuration de la forme $(r_{1}^{2},\dots,r_{s}^{2})$ avec $r_{1} < r_{2}$ et $r_{3}<r_{4}$ est réalisable par une différentielle de la strate $\quadomoduli[0](a_{1},a_{2};\rec[-2][s])$ avec $a_{1},a_{2} \in 4\mathbb{N}+3$.
\end{lem}

\begin{proof}
Soit $t=\frac{a_{1}+1}{2}\in 2\NN_{>0}$. Notons $\sigma=3,4,\dots,t-2$ et $\tau=1,2,t-1,\dots,s-1$ les suites croissantes de $t$  et $s-t-1$ nombres entiers.
\par
Nous donnons une construction géométrique directe: le graphe est constitué d'un sommet, correspondant au pôle de résidu  $r_{s}$, avec une boucle à laquelle sont collées deux chaînes (faites de sommets de valence $2$ avec un sommet de valence $1$ au bout). De plus, chaque chaîne appartient à une composante connexe distincte de la sphère privée de la boucle. Ainsi, les angles au bord des cylindres pour leurs domaines polaires sont attribués à deux zéros distincts.
\par
Pour les deux chaînes, nous attribuons les circonférences des cylindres selon les deux suites $\sigma$ et $\tau$ de la façon suivante. Les sommets de valence $1$ correspondent aux cylindres de circonférences $r_{1}$ et $r_{3}$. Puis les cylindres suivants sont de circonférences $r_{2}-r_{1}$ et $r_{4}-r_{3}$. Ainsi de suite nous prenons des cylindres dont le bord est la somme alternée des résidus de $\sigma$ et $\tau$ respectivement. L'hypothèse $r_{1}<r_{2}$ et $r_{3}<{4}$ garantit que tous les liens-selles ont une longueur strictement positive. Les deux segments identifiés sur le bord du cylindre de circonférence $r_{s}$ ont pour longueur la moitié de la somme alternée
$\sum\limits_{k\text{ impair}}r_{k}- \sum\limits_{k\text{ pair}}r_{k}$ qui est aussi strictement positive. En effet, comme $a_{1},a_{2} \in 4\mathbb{N}+3$ on a $\sigma$ et $\tau$ qui sont de longueurs paires strictement positives (voir le lemme~\ref{lem:ordresgraphes}).
\end{proof}

Lorsque l'un des ordres des zéros est dans $4\mathbb{N}+1$ et l'autre dans $4\mathbb{N}+3$, une variante de la construction précédente permet d'obtenir la plupart des configurations de résidus.

\begin{lem}\label{lem:4z+1/3}
Toute configuration de la forme $(r_{1}^{2},\dots,r_{s}^{2})$ avec $r_{2} < r_{3}$ et $\sum\limits_{k\text{ impair}}r_{k} < \sum\limits_{k\text{ pair}}r_{k}$ est réalisable par une différentielle de la strate $\quadomoduli[0](a_{1},a_{2};\rec[-2][s])$ avec $a_{1} \in 4\mathbb{N}+3$ et $a_{2} \in 4\mathbb{N}+1$.
\end{lem}

\begin{proof}
La construction est pratiquement identique à celle du lemme~\ref{lem:4z+3}. Nous avons toujours $t=\frac{a_{1}+1}{2}$ qui est pair mais ici $s \geq 4$ est un nombre pair. Nous posons $\sigma=2,3,\dots,t+1$ et $\tau=1,t+2,\dots,s-1$ et construisons le graphe de façon analogue. Les inégalités strictes permettent de garantir que tous les liens-selles de la surfaces sont de longueur non nulle.
\end{proof}

Quand les ordres des zéros sont tous deux dans $4\mathbb{Z}+1$, les constructions précédentes ne fonctionnent pas. Nous procéderons par récurrence en partant de la strate $\quadomoduli[0](5,1;\rec[-2][5])$. La strate $\quadomoduli[0](1,1;\rec[-2][3])$ étant exceptionnelle, nous la traitons à part d'emblée.

\begin{lem}\label{lem:double11}
Dans la strate $\quadomoduli[0](1,1;\rec[-2][3])$, la configuration $(r_{1}^{2},r_{2}^{2},r_{3}^{2})$ est réalisable par une différentielle si et seulement si $r_{3} \neq r_{1} + r_{2}$.
\end{lem}

\begin{proof}
Si $r_{3}=r_{1}+r_{2}$, l'obstruction (5) du corollaire~\ref{cor:except1} prouve que la configuration n'est pas réalisable dans la strate. Nous donnons les constructions dans les cas restants.
\par
Si $r_{3}>r_{1}+r_{2}$, le graphe d'incidence a un cycle fait d'un seul sommet, correspondant au cylindre de circonférence $r_{3}$. Le bord de ce domaine polaire se divise en quatre segments de longueurs $r_{1},\frac{1}{2}(r_{3}-r_{2}-r_{1}), r_{2}, \frac{1}{2}(r_{3}-r_{2}-r_{1})$ dans cet ordre. On colle les cylindres de tailles $r_{1}$ et $r_{2}$ sur les segments correspondants et identifie les deux derniers segments par rotation.
\par
Si $r_{3}<r_{1}+r_{2}$, le graphe d'incidence est un cycle de trois sommets. Les trois liens-selles sont de longueurs $\frac{1}{2}(r_{3}-r_{2}+r_{1})$, $\frac{1}{2}(r_{3}-r_{1}+r_{2})$ et $\frac{1}{2}(r_{2}-r_{3}+r_{1})$.
\end{proof}

Le cas de la strate $\quadomoduli[0](5,1;\rec[-2][5])$ constitue l'étape initiale de notre construction par récurrence quand les ordres des zéros sont tous deux dans $4\mathbb{Z}+1$.

\begin{lem}\label{lem:double15}
Dans la strate $\quadomoduli[0](1,5;\rec[-2][5])$, la configuration $(r_{1}^{2},\dots,r_{5}^{2})$ est réalisable par une différentielle si $r_{1}<r_{2}$ et $r_{3}<r_{4}$. De plus, nous pouvons exiger que pour cette différentielle les deux zéros  appartiennent au bord du cylindre de circonférence $r_{5}$.
\end{lem}

\begin{proof}
Supposons d'abord que $r_{5}+r_{4}>r_{3}+r_{2}+r_{1}$. Nous donnons la construction du lemme~\ref{lem:4z+3} avec $\sigma=(3,4,2)$ et $\tau=(1)$. Le cylindre de circonférence $r_{5}$ est de valence $4$ et les deux segments identifiés sur son bord sont de longueur $\frac{1}{2}(r_{5}-r_{3}-r_{1}+r_{4}-r_{2})$.
\par
Dans les cas restants, nous avons $r_{5}+r_{4} \leq r_{3}+r_{2}+r_{1}$ et donc $r_{5}+r_{3}< r_{4}+r_{2}+r_{1}$. Nous prenons alors un graphe d'incidence avec un cycle de longueur $3$ composé des cylindres de circonférences $r_{1},r_{3},r_{4}$. Nous collons le bord du cylindre de taille $r_{2}$ sur celui de taille $r_{3}$ et l'autre bord du cylindre de taille $r_{3}$ sur celui de taille $r_{4}$. Les longueurs des trois liens-selles du cycle seront alors $\frac{1}{2}((r_{4}-r_{3}+r_{2})+r_{1}-r_{5})$, $\frac{1}{2}((r_{4}-r_{3}+r_{2})-r_{1}+r_{5})$ et $\frac{1}{2}(r_{5}+r_{1}-(r_{4}-r_{3}+r_{2}))$. Ces longueurs sont par hypothèse strictement positives.
\end{proof}

Nous pouvons enfin réaliser notre construction par récurrence. Celle-ci couvre en particulier le cas des configurations dont tous les résidus sont deux à deux distincts.

\begin{lem}\label{lem:4z+1}
Pour $s \geq 5$, toute configuration de la forme $(r_{1}^{2},\dots,r_{s}^{2})$ avec $r_{2k-1} < r_{2k}$ pour chaque $k$ vérifiant $1 \leq k \leq \frac{s-1}{2}$ est réalisable dans $\quadomoduli[0](a_{1},a_{2};\rec[-2][s])$ avec $a_{1},a_{2} \in 4\mathbb{N}+1$.
\end{lem}

\begin{proof}
Pour de telles strates, $s$ est un nombre impair. Nous montrons la proposition par récurrence sur $s$ en exigeant en outre que le
cylindre de circonférence $r_{s}$ (le plus grand) soit bordé par les deux zéros. La propriété est valide pour $s=5$ par le lemme~\ref{lem:double15} et nous la supposerons vraie jusqu'au rang $s$.
\par 
Considérons une configuration de $s+2$ racines de résidus $C=(r_{1},r_{2},\dots,r_{s+2})$ rangées par ordre croissant et vérifiant $r_{2k-1} < r_{2k}$ pour chaque $k$ entre $1$ et $\frac{s+1}{2}$.
\par
Nous définissons une nouvelle configuration $C'$ de $s$ éléments en remplaçant les trois derniers éléments $r_{s},r_{s+1},r_{s+2}$ de $C$ par $\tilde{r}=r_{s+2}-r_{s+1}+r_{s}$. Par hypothèse, nous avons $\tilde{r} \geq r_{s}$. La nouvelle configuration $C'$ satisfait donc les hypothèses de récurrence et peut donc être réalisée par une différentielle de n'importe quelle strate ayant $s$ pôles doubles et des ordres de zéros dans $4\mathbb{Z}+1$. Nous pouvons remplacer le cylindre de circonférence $\tilde{r}$ par un cylindre de circonférence $r_{s+2}$ sur lequel on colle par un lien-selle de longueur $r_{s+1}-r_{s}>0$ une chaîne formée du cylindre de circonférence $r_{s+1}$ et de celui de circonférence $r_{s}$. Comme le lien-selle formé peut être adjacent à n'importe quel zéro, nous obtenons n'importe quelle strate ayant $s+2$ pôles doubles. De plus, le cylindre de plus grande circonférence est encore bordé par les deux zéros.
\end{proof}

\subsubsection{La configuration est arithmétique de somme impaire}

Nous introduisons d'abord une chirurgie permettant d'ajouter des paires de résidus identiques.

\begin{lem}\label{lem:+2res}
Considérons une configuration de résidus réels positifs $(r_{1}^{2},\dots,r_{s}^{2})$ réalisable dans une strate $\quadomoduli[0](a_{1},a_{2};\rec[-2][s])$. Soit $x \in \mathbb{R}_{+}^{\ast}$ satisfaisant l'une des deux propositions suivantes:
\begin{itemize}
    \item au moins un  $r_{i}$ n'est pas un multiple entier de $x$;
    \item si chaque $r_{i}$ est un multiple entier de $x$, alors $\frac{1}{x}\sum\limits_{i=1}^{s}r_{i} \neq 2s$.
\end{itemize}
Alors $(r_{1}^{2},\dots,r_{s}^{2},x^{2},x^{2})$ est réalisable dans la strate $\quadomoduli[0](a_{1}+2,a_{2}+2;\rec[-2][s+2])$.
\end{lem}

\begin{proof}
Considérons  une différentielle  de la strate
$\quadomoduli[0](a_{1},a_{2};\rec[-2][s])$ dont les résidus sont~$(r_{1}^{2},\dots,r_{s}^{2})$. Rappelons que le graphe d'incidence de cette différentielle possède un cycle. Si l'un des segments du cycle correspond à un lien-selle de longueur strictement inférieure à $x$, alors on coupe la surface le long de ce lien-selle et on y colle successivement les deux cylindres de circonférence $x$.
\par
Si l'un des liens-selles du cycle est de longueur strictement supérieure à $x$, on fait la construction suivante. On découpe le lien-selle correspondant et on colle un cylindre sur chacun des segments ainsi créé l'un adjacent à l'une des singularités et l'autre cylindre adjacent à l'autre singularité. Les parties restantes des segments sont collées entre elles en respectant les singularités. On obtient dans les deux cas une différentielle dans la strate  $\quadomoduli[0](a_{1}+2,a_{2}+2;\rec[-2][s+2])$ dont les résidus sont $(r_{1}^{2},\dots,r_{s}^{2},x^{2},x^{2})$.
\par
Si tous les liens-selles reliant les deux zéros sont de longueur $x$ et que le graphe d'incidence est un cycle, alors pour $1 \leq i \leq s$, nous avons $r_{i}=2x$, ce qui est exclu. Par conséquent, le graphe d'incidence contient au moins un arbre collé sur le cycle.
\par
Si un lien-selle reliant un arbre au cycle est de longueur strictement supérieure à $x$, on fait la construction représentée sur la figure~\ref{fig:ajoutdeuxpoles}. On réduit la longueur de ce lien-selle de $x$, puis collons un cylindre de circonférence sur le bord de chacun des cylindres incidents à ce lien-selle.
\par
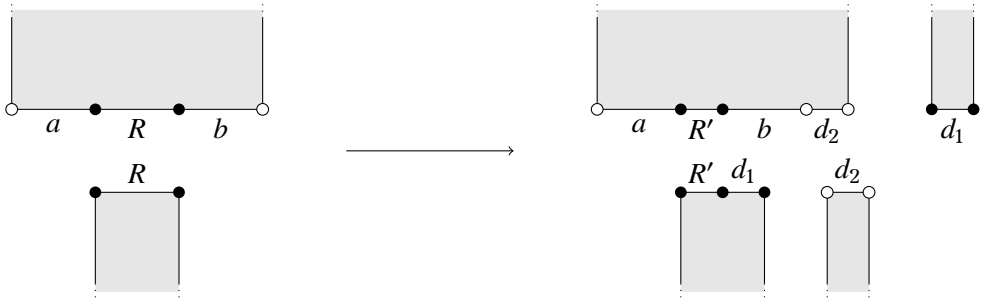
\begin{figure}[htb]
\begin{tikzpicture}[scale=1.1]
\begin{scope}[xshift=-6cm]
\coordinate (a) at (-1,1);
\coordinate (b) at (2,1);
\coordinate (c) at (0,1);
\coordinate (d) at (1,1);

    \fill[fill=black!10] (a)  -- (c)coordinate[pos=.5](f)-- (d)coordinate[pos=.5](g)-- (b)coordinate[pos=.5](j) -- ++(0,1.2) --++(-3,0) -- cycle;
    \fill (c)  circle (2pt);
\fill (d) circle (2pt);
 \draw  (a) -- (b);
 \draw (a) -- ++(0,1.1) coordinate (d)coordinate[pos=.5](h);
 \draw (b) -- ++(0,1.1) coordinate (e)coordinate[pos=.5](i);
 \draw[dotted] (d) -- ++(0,.2);
 \draw[dotted] (e) -- ++(0,.2);
\node[below] at (f) {$a$};
\node[below] at (g) {$R$};
\node[below] at (j) {$b$};
    \filldraw[fill=white] (a)  circle (2pt);
\filldraw[fill=white] (b) circle (2pt);
\end{scope}

\begin{scope}[xshift=-5cm,yshift=-1cm]
\coordinate (a) at (-1,1);
\coordinate (b) at (0,1);

\fill[fill=black!10] (a)  -- (b)coordinate[pos=.5](f) -- ++(0,-1.2) --++(-1,0) -- cycle;
    \fill (a)  circle (2pt);
\fill[] (b) circle (2pt);
 \draw  (a) -- (b);
 \draw (a) -- ++(0,-1.1) coordinate (d)coordinate[pos=.5](h);
 \draw (b) -- ++(0,-1.1) coordinate (e)coordinate[pos=.5](i);
 \draw[dotted] (d) -- ++(0,-.2);
 \draw[dotted] (e) -- ++(0,-.2);
\node[above] at (f) {$R$};

\draw[->](2,1.5) -- (4,1.5);
\end{scope}


\begin{scope}[xshift=7cm]
\begin{scope}[xshift=-6cm]
\coordinate (a) at (-1,1);
\coordinate (b) at (2,1);
\coordinate (c) at (0,1);
\coordinate (d) at (.5,1);
\coordinate (o) at (1.5,1);

    \fill[fill=black!10] (a)  -- (c)coordinate[pos=.5](f)-- (d)coordinate[pos=.5](g)-- (o)coordinate[pos=.5](k)-- (b)coordinate[pos=.5](j) -- ++(0,1.2) --++(-3,0) -- cycle;
    \fill (c)  circle (2pt);
\fill (d) circle (2pt);
 \draw  (a) -- (b);
 \draw (a) -- ++(0,1.1) coordinate (d)coordinate[pos=.5](h);
 \draw (b) -- ++(0,1.1) coordinate (e)coordinate[pos=.5](i);
 \draw[dotted] (d) -- ++(0,.2);
 \draw[dotted] (e) -- ++(0,.2);
\node[below] at (f) {$a$};
\node[below] at (j) {$d_{2}$};
\node[below] at (g) {$R'$};
\node[below] at (k) {$b$};
    \filldraw[fill=white] (a)  circle (2pt);
\filldraw[fill=white] (b) circle (2pt);
\filldraw[fill=white] (o) circle (2pt);

\end{scope}

\begin{scope}[xshift=-2.5cm]
\coordinate (a) at (-.5,1);
\coordinate (b) at (0,1);

    \fill[fill=black!10] (a)  -- (b)coordinate[pos=.5](f) -- ++(0,1.2) --++(-.5,0) -- cycle;
    \fill (a)  circle (2pt);
\fill[] (b) circle (2pt);
 \draw  (a) -- (b);
 \draw (a) -- ++(0,1.1) coordinate (d)coordinate[pos=.5](h);
 \draw (b) -- ++(0,1.1) coordinate (e)coordinate[pos=.5](i);
 \draw[dotted] (d) -- ++(0,.2);
 \draw[dotted] (e) -- ++(0,.2);
\node[below] at (f) {$d_{1}$};
\end{scope}

\begin{scope}[xshift=-5cm,yshift=-1cm]
\coordinate (a) at (-1,1);
\coordinate (b) at (0,1);

\fill[fill=black!10] (a)  -- (b)coordinate[pos=.25](f)coordinate[pos=.5](g)coordinate[pos=.75](j) -- ++(0,-1.2) --++(-1,0) -- cycle;
    \fill (a)  circle (2pt);
\fill[] (b) circle (2pt);
\fill[] (g) circle (2pt);
 \draw  (a) -- (b);
 \draw (a) -- ++(0,-1.1) coordinate (d)coordinate[pos=.5](h);
 \draw (b) -- ++(0,-1.1) coordinate (e)coordinate[pos=.5](i);
 \draw[dotted] (d) -- ++(0,-.2);
 \draw[dotted] (e) -- ++(0,-.2);
\node[above] at (f) {$R'$};
\node[above] at (j) {$d_{1}$};
\end{scope}

\begin{scope}[xshift=-4.25cm,yshift=-1cm]
\coordinate (a) at (0,1);
\coordinate (b) at (.5,1);

    \fill[fill=black!10] (a)  -- (b)coordinate[pos=.5](f) -- ++(0,-1.2) --++(-1/2,0) -- cycle;
    \fill (a)  circle (2pt);
 \draw  (a) -- (b);
 \draw (a) -- ++(0,-1.1) coordinate (d)coordinate[pos=.5](h);
 \draw (b) -- ++(0,-1.1) coordinate (e)coordinate[pos=.5](i);
 \draw[dotted] (d) -- ++(0,-.2);
 \draw[dotted] (e) -- ++(0,-.2);
 \filldraw[fill=white] (a) circle (2pt);
 \filldraw[fill=white]  (b)  circle (2pt);
\node[above] at (f) {$d_{2}$};
\end{scope}
\end{scope}

\end{tikzpicture}
\caption{Une opération qui ajoute deux pôles doubles de résidus égaux.} \label{fig:ajoutdeuxpoles}
\end{figure}

Si au contraire un tel lien-selle entre un cylindre $C_{1}$ du cycle et un cylindre $C_{2}$ de l'arbre est de longueur $L<x$, alors on découpe ce lien-selle, déconnectant l'arbre du reste du graphe, pour le recoller sur le bord d'un cylindre de circonférence $x$. Sur ce dernier, un segment de bord de longueur $x-L$ est alors recollé sur le bord $C_{1}$ tandis qu'on ajoute un nouveau cylindre de circonférence $x$ sur le bord de $C_{1}$. Les longueurs des liens-selles du cycle doivent alors changer, mais comme elles sont des sommes signées de racines de tous les résidus quadratiques (voir section~\ref{sec:arrhyp}), elles ne peuvent diminuer que de $L$ dans le pire des cas et vont donc rester strictement positives.
\par
Nous considérons finalement les cas où les segments du cycle sont égaux à $x$ et les segments entre tous les arbres et le cycle sont égaux à $x$. Considérons un arbre $\Gamma$ collé au cycle. Remontons $\Gamma$ en partant du cycle et considérons le premier segment de longueur strictement inférieure à $x$. Nous pouvons alors couper ce segment et coller la branche sur le cycle au sommet où est collé~$\Gamma$. Cette opération est représentée sur la figure~\ref{fig:operationgraphe}. 
Si le segment que l'on coupe est à distance impaire du cycle, on obtient un nouveau graphe pour lequel les longueurs des segments du cycle ne sont pas toutes égales entre elles. Si le segment est à distance paire l'arbre que nous avons rajouté sur le cycle est attaché par un segment de longueur strictement inférieure à $x$. Dans les deux cas, on est ramené à un cas précédent.
\begin{figure}[htb]
\begin{tikzpicture}

\begin{scope}[xshift=0cm]
\coordinate (z1) at (1,-1);
\coordinate (z2) at (-1,-1);
\coordinate (z3) at (0,0);
\coordinate (z4) at (1,-2);
\coordinate (z5) at (.2,-3);
\coordinate (z6) at (1,-3);
\coordinate (z7) at (-1,-1.7);
\coordinate (z8) at (-1,-2.3);
\coordinate (z9) at (-1,-3);

\draw (z1) --node[right] {$x$} (z4);
\draw (z4) --node[left] {$y$} (z5);
\draw (z4) --node[right] {$z$} (z6);
\draw (z2) --node[left] {$x$} (z7);
\draw (z7) --node[left] {$x$} (z8);
\draw (z8) --node[left] {$z$} (z9);
\foreach \i in {1,2,...,9}
\fill (z\i) circle (2pt);
\draw (z1) ..node[above] {$x$} controls ++(-140:.9) and  ++(-40:.9)   .. (z2);
\draw (z2) ..node[above] {$x$} controls ++(90:.6) and ++(180:.6) .. (z3);
\draw (z3) ..node[above] {$x$} controls ++(0:.6) and ++(90:.6) .. (z1);

\draw[->] (1.7,-3) -- (2.8,-4);
\draw[->] (-1.7,-3) -- (-2.8,-4);
\end{scope}

\begin{scope}[xshift=-4.5cm,yshift=-2cm]
\coordinate (z1) at (1,-1);
\coordinate (z2) at (-1,-1);
\coordinate (z3) at (0,0);
\coordinate (z4) at (1,-2);
\coordinate (z5) at (.2,-3);
\coordinate (z6) at (1,-3);
\coordinate (z7) at (-1,-1.7);
\coordinate (z8) at (-1,-2.3);
\coordinate (z9) at (-1.7,-1);

\draw (z1) --node[right] {$x$} (z4);
\draw (z4) --node[left] {$y$} (z5);
\draw (z4) --node[right] {$z$} (z6);
\draw (z2) --node[left] {$x+z$} (z7);
\draw (z7) --node[left] {$x-z$} (z8);
\draw (z2) --node[above] {$z$} (z9);
\foreach \i in {1,2,...,9}
\fill (z\i) circle (2pt);
\draw (z1) ..node[above] {$x$} controls ++(-140:.9) and  ++(-40:.9)   .. (z2);
\draw (z2) ..node[above] {$x$} controls ++(90:.6) and ++(180:.6) .. (z3);
\draw (z3) ..node[above] {$x$} controls ++(0:.6) and ++(90:.6) .. (z1);
\end{scope}

\begin{scope}[xshift=4.5cm,yshift=-2cm]
\coordinate (z1) at (1,-1);
\coordinate (z2) at (-1,-1);
\coordinate (z3) at (0,0);
\coordinate (z4) at (1,-2);
\coordinate (z5) at (.2,-3);
\coordinate (z6) at (1.7,-1);
\coordinate (z7) at (-1,-1.7);
\coordinate (z8) at (-1,-2.3);
\coordinate (z9) at (-1,-3);

\draw (z1) --node[right] {$x+z$} (z4);
\draw (z4) --node[left] {$y$} (z5);
\draw (z1) --node[above] {$z$} (z6);
\draw (z2) --node[left] {$x$} (z7);
\draw (z7) --node[left] {$x$} (z8);
\draw (z8) --node[left] {$z$} (z9);
\foreach \i in {1,2,...,9}
\fill (z\i) circle (2pt);
\draw (z1) ..node[above] {$x-z$} controls ++(-140:.9) and  ++(-40:.9)   .. (z2);
\draw (z2) ..node[above left] {$x+z$} controls ++(90:.6) and ++(180:.6) .. (z3);
\draw (z3) ..node[above right] {$x-z$} controls ++(0:.6) and ++(90:.6) .. (z1);
\end{scope}
\end{tikzpicture}
\caption{L'opération de couper une branche et la recoller avec $z<x$ si la branche est à distance paire (à gauche) ou impaire (à droite)} \label{fig:operationgraphe}
\end{figure}
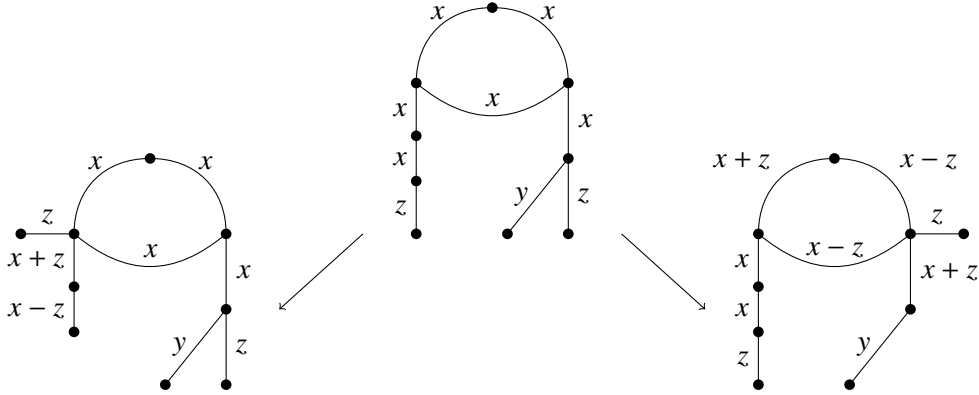
\par
Demeurent donc les cas dans lesquels tous les segments internes des arbres sont de
longueurs supérieures ou égales à $x$.
 Remontons un arbre en partant du cycle et considérons le
premier segment de longueur $L > x$, entre un sommet $A$ et un sommet $B$. Le sommet $A$ est attaché au reste de l’arbre par une arête de taille $x$.
Nous en retirons $A$ et sa branche pour la remplacer par une feuille correspondant à un cylindre de taille $x$. Puis nous coupons une arête du cycle et insérons les sommets $A$ et $B$ dans ce cycle en préservant la branche de $A$. De plus, ces sommets sont reliés au cycle par des segments de longueur $x$ et entre eux par un segment de longueur $L-x$. Nous collons enfin une nouvelle feuille correspondant à un cylindre de taille $x$ sur le sommet $A$. La différentielle associée à ce graphe possède les propriétés souhaitées.
\par
Finalement considérons le cas où tous les liens-selles de la surface sont de longueur $x$. Comme il y a $2s$ liens-selles, cela signifie que $\frac{1}{x}\sum\limits_{i=1}^{s}r_{i}=2s$, ce qui est exclu par hypothèse.
\end{proof}

En s'appuyant sur les constructions précédentes, nous pouvons traiter le cas des configurations arithmétiques de somme impaire.

\begin{prop}\label{prop:polesdoublesimpair}
Étant donnée configuration arithmétique $(r_{1}^{2},\dots,r_{s}^{2})$ de somme $T=\sum r_{i}$ impaire. Si  $T \geq a_{2}+2$ alors elle est réalisable dans la strate $\quadomoduli[0](a_{1},a_{2};\rec[-2][s])$ avec $a_{1} \leq a_{2}$.
\end{prop}

\begin{proof}
Nous procédons par récurrence sur le nombre $s$ de pôles doubles. Quand $s \leq 3$, les strates sont soit de la forme $\quadomoduli[0](-1,2s-3;\rec[-2][s])$, soit la strate  $\quadomoduli[0](1,1;\rec[-2][3])$. Les propositions~\ref{prop:poledouble-1} et~\ref{lem:double11} donnent le résultat. Nous supposerons l'existence de différentielles pour chaque configuration arithmétique de somme impaire vérifiant la borne dans chaque strate jusqu'au rang $s$. Pour une certaine strate $\quadomoduli[0](a_{1},a_{2};\rec[-2][s+1])$, choisissons maintenant une configuration arithmétique $(R_{1},\dots,R_{s+1})$ dont les racines positives forment la configuration $C=(r_{1},\dots,r_{s+1})$ de circonférences de cylindres. Si $T=\sum\limits_{i=1}^{s+1} r_{i}$ est impair et satisfait $T \geq a_{2}+2$, nous prouvons l'existence d'une différentielle réalisant cette configuration.
\par
Grâce à la proposition~\ref{prop:poledouble-1} nous pouvons nous restreindre au cas $a_{1},a_{2} \geq 1$. De plus, si les~$r_{i}$ sont deux à deux distincts, les lemmes~\ref{lem:4z+3},~\ref{lem:4z+1/3} et~\ref{lem:4z+1} permettent aussi de conclure. Nous supposerons donc que deux éléments de $C$ sont identiques. Notons $d$ la plus petite valeur apparaissant au moins deux fois dans la configuration $C$. Nous définissons $C'$ en retirant à~$C$ deux éléments égaux à $d$. On note $k$ le pgcd des éléments de $C'$ et $T'=T-2d$ leur somme.
\par
Comme  $T$  est impair, les nombres $k$ et $\frac{T'}{k}$ sont impairs. Si  $\frac{T'}{k} \geq a_{2}$, alors par récurrence les carrés des éléments de $C'$ sont réalisables dans $\quadomoduli[0](a_{1}-2,a_{2}-2;\rec[-2][s-1])$. Cela implique d'après le lemme~\ref{lem:+2res} que $(R_{1},\dots,R_{s+1})$ est réalisable dans $\quadomoduli[0](a_{1},a_{2};\rec[-2][s+1])$.
\par
Supposons à présent que $\frac{T'}{k} < a_{2}$. Comme $a_{2}\geq 2s-3$ (puisque $a_{1} \geq 1$) et que $\frac{T'}{k}$ est impair, cela signifie que $\frac{T'}{k}$ vaut au plus $2s-5$ (pour $s-1$ éléments) donc au moins trois éléments de~$C'$ valent $k$. Définissons $C''$ en retirant de $C$ deux éléments égaux à $k$. Puisque les éléments de $C$ et $C''$ prennent les mêmes valeurs, les éléments de $C''$ sont premiers entre eux. Si $C''$ est réalisable dans la strate $\quadomoduli[0](a_{1}-2,a_{2}-2;\rec[-2][s-1])$, alors on en déduit (comme plus haut)
que la configuration formée des carrés des éléments de $C$ est réalisable dans la strate originale. Si enfin $C''$ n'est pas réalisable, c'est que la somme de ses éléments $T''$ vaut au plus $2s-5$, qu'elle compte trois éléments égaux à $1$, ce qui implique que $d=k=1$. C'est absurde car alors $T''=T-2$ et donc $T'' \geq a_{2}$.
\end{proof}

\subsubsection{La configuration est arithmétique de somme paire}

Afin de contourner les obstructions portant sur les configurations en crosse ou triangulaires, nous donnons des constructions spécifiques pour certaines strates.

\begin{lem}\label{lem:double33}
Toute configuration de la forme $(x^{2},y^{2},z^{2},t^{2},t^{2})$ avec $x,y,z,t > 0$, $z=x+y$ et $t \notin \lbrace{ x,y,z \rbrace}$ est réalisable par une différentielle de la strate $\quadomoduli[0](3,3;\rec[-2][5])$.
\end{lem}

\begin{proof}
Selon les valeurs de $x,y,t$, nous choisissons différents graphes d'incidence pour construire la différentielle voulue.

Si $z>t>x,y$, alors on choisit un graphe cyclique avec des cylindres de circonférences $(x,y,t,z,t)$ dans cet ordre. Les longueurs des segments sont alors données par $(z-t,t-x,x,y,t-y)$.

Si $t>z$, alors on choisit un graphe avec un cycle de longueur $3$ fait de cylindres de tailles respectives $t$, $t$ et $z$. On peut ensuite coller un cylindre de taille $x$ sur le bord d'un cylindre de taille $t$  et le cylindre de taille $y$ sur l’autre cylindre de
taille $t$. Les longueurs des segments du cycle sont alors $(t-z,x,y)$.

Si $x<t<y<z$, alors le cycle est fait de trois cylindres de tailles $t$, $y$ et $z$. On colle ensuite un cylindre de taille $x$ sur le cylindre de taille $y$ et l’autre cylindre de taille $t$ sur celui de taille $z$. Les longueurs des segments du cycle sont alors $(x,y-t,t-x)$.

Si $t<x$, alors le cycle est simplement fait du cylindre de taille $z$. On colle sur ce sommet deux chaînes de deux sommets dont le sommet de valence $2$ est un cylindre de taille $x$ ou~$y$ tandis que le sommet de valence $1$ est un cylindre de taille $t$.
\end{proof}

\begin{lem}\label{lem:double53}
Toute configuration de la forme $(x^{2},x^{2},y^{2},y^{2},z^{2},z^{2})$ avec $0<x<y<z$ est réalisable par une différentielle de la strate $\quadomoduli[0](5,3;\rec[-2][6])$.
\end{lem}

\begin{proof}
On choisit un graphe avec un cycle de longueur $3$ fait de cylindres de tailles respectives $y$, $z$ et $z$. Sur l'un de ces cylindres de taille $z$, on colle le bord d'un cylindre de taille $x$. Sur l'autre on colle une chaîne de deux sommets dont le sommet de valence $2$ est un cylindre de taille $y$ tandis que le sommet de valence $1$ est un cylindre de taille $x$. Les longueurs des liens-selles du cycle sont $z-y$, $x$ et $y-x$.
\end{proof}

\begin{lem}\label{lem:double55}
Dans la strate $\quadomoduli[0](5,5;\rec[-2][7])$, les configurations suivantes sont réalisables par une différentielle quadratique:
\begin{enumerate}
    \item $(x^{2},y^{2},y^{2},y^{2},z^{2},z^{2},z^{2})$ avec $y>x$ et $z=x+y$;
    \item $(x^{2},x^{2},x^{2},y^{2},z^{2},z^{2},z^{2})$ avec $y>x$ et $z=x+y$;
    \item $(x^{2},x^{2},x^{2},y^{2},y^{2},y^{2},z^{2})$ avec $y>x$, $z=x+y$ et $y \neq 2x$.
\end{enumerate}
\end{lem}

\begin{proof}
Dans le cas (1), on choisit un graphe d'incidence avec un cycle de longueur~$3$ fait de cylindres de tailles $y,z,z$. Sur chacun des deux cylindres de taille $z$ dans le cycle, on colle une chaîne (de telle façon que chaque chaîne contribue à un zéro distinct) formée de deux sommets. Sur la première chaîne, le sommet de valence $2$ est un cylindre de taille $y$ tandis que le sommet de valence $1$ est un cylindre de taille $x$. Sur la seconde chaîne, ils sont respectivement de tailles $z$ et $y$. Les trois liens-selles du cycle sont de longueur
$y-x$, $x$ et $x$. Le calcul des ordres est une conséquence directe du lemme~\ref{lem:ordresgraphes}.

Dans le cas (2), on se donne un graphe d'incidence avec un cycle de longueur $3$ fait des trois cylindres de taille $z$. Sur l'un de ces cylindres, on colle une chaîne comprenant un cylindre de taille $x$ pour le sommet de valence $1$ et un cylindre de taille $y$ pour le sommet de valence $2$. Sur les deux autres cylindres du cycle, on colle un cylindre de taille $x$. Les trois liens-selles du cycle sont de longueurs $x,x,y-x$.

Dans le cas (3), nous donnons deux constructions selon le signe de $y-2x$. Si $y<2x$, la construction du lemme~\ref{lem:polesimple} fonctionne avec $\sigma=1,4,2$ et $\tau=3,5,6$. Si au contraire $y>2x$, alors nous donnons une construction avec également un cycle de longueur $1$ dont le sommet est toujours le cylindre de taille $z$. D'un côté du cycle, nous collons une chaîne comprenant un cylindre de taille $x$ pour le sommet de valence $1$ et deux cylindres de taille $y$ pour les sommets de valence $2$. De l'autre côté du cycle, nous collons un cylindre de taille $y$ sur lequel sont recollés deux cylindres de taille $x$.
\end{proof}

\begin{lem}\label{lem:double75}
Toute configuration de la forme $(x^{2},x^{2},x^{2},x^{2},y^{2},y^{2},y^{2},y^{2})$ avec $0<x<y$ et $y \neq 2x$ est réalisable par une différentielle de la strate $\quadomoduli[0](7,5;\rec[-2][8])$.
\end{lem}

\begin{proof}
Nous donnerons deux constructions selon le signe de $y-2x$.

Si $y>2x$, on choisit un cycle de longueur $3$ fait de trois cylindres de taille $y$. Sur l'un d'eux, on colle un cylindre de taille $y$ sur lequel sont recollés deux cylindres de taille $x$. Sur les deux autres cylindres du cycle, on colle un cylindre de taille $x$. Les longueurs des trois liens-selles du cycle sont $y-2x,x,x$.

Si $y<2x$, alors on choisit un cycle de longueur $3$ fait de deux cylindres de taille $y$ et un de taille $x$. Sur l'un des cylindres de taille $y$, on colle une chaîne faite d'un cylindre de taille~$y$ puis un cylindre de taille $x$. Sur l'autre cylindre de taille $y$ du cycle, on colle une chaîne faite de trois cylindres de tailles respectives $x,y,x$. Les longueurs des trois liens-selles du cycle sont $2x-y,y-x,y-x$.
\end{proof}

Les cas particuliers ayant été traités, nous pouvons donner pour chaque strate la caractérisation des configurations arithmétiques de somme paire qui peuvent y être réalisées. La démonstration couvre aussi le cas des configurations non arithmétiques.

\begin{prop}\label{prop:polesdoublespair}
Étant donnée une strate $\quadomoduli[0](a_{1},a_{2};\rec[-2][s])$ avec $a_{1} \leq a_{2}$, toute configuration $(r_{1}^{2},\dots,r_{s}^{2})$ non arithmétique ou arithmétique de somme $T=\sum r_{i}$ pair est réalisable par une différentielle de cette strate sauf précisément dans les cas suivants:
\begin{enumerate}
    \item $T<a_{1}+a_{2}+4$;
    \item $a_{1}=a_{2}$ et la configuration est triangulaire;
    \item $a_{2}=a_{1}+2$ et la configuration est en crosse.
\end{enumerate}
\end{prop}

\begin{proof}
Les obstructions proviennent de la proposition~\ref{prop:ObstrArithm} et des cas (4) et (5) du corollaire~\ref{cor:except1}. Pour établir la réciproque, nous procédons par récurrence sur le nombre~$s$ de pôles doubles. Quand $s \leq 3$, les strates sont soit de la forme $\quadomoduli[0](-1,2s-3;\rec[-2][s])$ ou bien $\quadomoduli[0](1,1;\rec[-2][3])$. Les propositions~\ref{prop:poledouble-1} et~\ref{lem:double11} donnent le résultat. Nous supposerons le résultat valide dans chaque strate jusqu'au rang $s \geq 3$. Pour une strate $\quadomoduli[0](a_{1},a_{2};\rec[-2][s+1])$, choisissons maintenant une configuration  $(R_{1},\dots,R_{s+1})$ dont les racines positives forment la configuration $C=(r_{1},\dots,r_{s+1})$ de circonférences de cylindres. En supposant qu'aucune obstruction ne s'y applique, nous prouvons l'existence d'une différentielle réalisant cette configuration. La démonstration est en partie analogue à celle de la proposition~\ref{prop:polesdoublesimpair}.
\par
La proposition~\ref{prop:poledouble-1} permet de se restreindre au cas $a_{1},a_{2} \geq 1$. Les lemmes~\ref{lem:4z+3},~\ref{lem:4z+1/3} et~\ref{lem:4z+1} permettent aussi de supposer qu'au moins deux éléments de $C$ sont égaux. Pour chaque valeur $d$ apparaissant au moins deux fois dans la configuration $C$, nous définissons $C_{d}$ en retirant à $C$ deux éléments égaux à $d$. Trois situations peuvent apparaître:
\begin{enumerate}
    \item $C_{d}$ n'est pas réalisable dans la strate $\quadomoduli[0](a_{1}-2,a_{2}-2;\rec[-2][s-1])$;
    \item $C_{d}$ est réalisable dans la strate $\quadomoduli[0](a_{1}-2,a_{2}-2;\rec[-2][s-1])$, $C$ est proportionnelle à une configuration arithmétique de somme $2(s+1)$ avec $d=1$;
    \item $C_{d}$ est réalisable dans la strate $\quadomoduli[0](a_{1}-2,a_{2}-2;\rec[-2][s-1])$ mais soit $C$ n'est pas proportionnelle à une configuration arithmétique de somme $2(s+1)$ soit c'est le cas mais au moins un élément de $C$ n'est pas un multiple entier de $d$.
\end{enumerate}

Dans le cas (3), le lemme~\ref{lem:+2res} montre que la configuration $C$ est réalisable dans la strate voulue. 

Dans le cas (2), normalisons en supposant que $C=(1,1,r_{3},\dots,r_{s+1})$ de somme $2s+2$. La configuration est donc $(1,1,2,\dots,2,3,3)$ ou $(1,1,2,\dots,2,4)$. Si la configuration compte au moins trois éléments valant $2$, la configuration $C_{2}$ est dans chaque cas réalisable. Si elle en compte deux, $C_{2}=(1,1,3,3)$ est en crosse mais $(1,1,2,2,3,3)$ est réalisable dans $\quadomoduli[0](5,3;\rec[-2][6])$ (lemme~\ref{lem:double53}). Si $C$ compte un seul élément valant~$2$, la configuration $(1,1,2,3,3)$ est réalisable dans $\quadomoduli[0](3,3;\rec[-2][5])$ (voir lemme~\ref{lem:double33}) et $\quadomoduli[0](5,1;\rec[-2][5])$ (prenons un graphe d'incidence avec un cycle de longueur $3$ et deux feuilles collées sur deux sommets du cycle, attribuons une longueur de $1$ pour chaque arête et nous obtenons la surface voulue). Ensuite, $(1,1,2,4)$ est réalisable dans $\quadomoduli[0](1,3;\rec[-2][4])$ (voir lemme~\ref{lem:4z+1/3}). Enfin, considérons les cas dans lesquels aucun élément de $C$ ne vaut $2$. La configuration $(1,1,4)$ est réalisable dans  $\quadomoduli[0](1,1;\rec[-2][3])$ (voir lemme~\ref{lem:double11}) tandis que $(1,1,3,3)$ est de toute façon interdit dans $\quadomoduli[0](1,3;\rec[-2][4])$.
\smallskip
\par
Par conséquent, il reste à étudier le cas (1), i.e. pour toute valeur de $d$, la configuration $C_{d}$ est interdite dans la strate $\quadomoduli[0](a_{1}-2,a_{2}-2;\rec[-2][s-1])$. Nous montrerons que cela implique que $C$ est en fait interdite dans $\quadomoduli[0](a_{1},a_{2};\rec[-2][s+1])$.
Une configuration $C_{d}$ peut ne pas être réalisable pour trois raisons: configuration  en crosse, triangulaire ou arithmétique de somme trop faible.
\par
Si  $C_{d}$ est en crosse et $a_{2}=a_{1}+2$, alors $C$ est de la forme $(x,x,y,y,z,\dots,z)$ avec un nombre pair d'éléments valant $z$. Si $x,y,z$ sont distincts, alors soit $s+1=6$ et le lemme~\ref{lem:double53} donne une construction directe, soit $s+1 \geq 8$ et nous pouvons retirer une paire d'éléments valant $z$ pour obtenir une configuration réalisable dans la strate $\quadomoduli[0](a_{1}-2,a_{2}-2;\rec[-2][s-1])$. Si $x,y,z$ ne sont pas distincts, comme $C$ n'est pas (par hypothèse) une configuration en crosse, alors $x=y$ et $s+1 \geq 8$. Comme $C$ n'est concerné par aucune obstruction, nous avons nécessairement $y \neq 2z$ et $x \neq 2z$ ($C$ serait dans ce cas proportionnel à une configuration arithmétique de somme strictement inférieure à $2s+2=a_{1}+a_{2}+4$). Dans le cas $s+1 \geq 10$, $C_{z}$ est une configuration réalisable. Quand $s=8$, c'est le lemme~\ref{lem:double75} qui donne la construction.
\par
Si $C_{d}$ est triangulaire et $a_{1}=a_{2}$, alors $C$ est de la forme $(d,d,x,y,z,\dots,z)$ avec un nombre impair d'éléments valant $z$. Si $s+1=5$ et $d \notin \lbrace{ x,y,z \rbrace}$, alors le lemme~\ref{lem:double33} donne la construction. Si $d \in \lbrace{ x,y,z \rbrace}$, alors $C$ est une configuration triangulaire. Si $s+1=7$,  alors soit $d \in \lbrace{ x,y,z \rbrace}$ soit $d \notin \lbrace{ x,y,z \rbrace}$. Dans le premier cas le lemme~\ref{lem:double55} donne la construction sauf si $C$ est proportionnel à $(1,1,1,2,2,2,3)$ qui est exclu par hypothèse. Dans le second cas nous retirons une paire d'éléments valant $z$ pour obtenir une configuration de la forme $(d,d,x,y,x+y)$ (quitte à échanger les rôles de $x,y,z$) réalisable dans $\quadomoduli[0](3,3;\rec[-2][5])$ (voir lemme~\ref{lem:double33}). Si $s+1 \geq 9$, alors $C_{z}$ est une configuration de la forme $(d,d,x,y,z,\dots,z)$ avec un nombre impair (au moins $3$ d'éléments valant $z$). Par hypothèse, $C_{z}$ n'est pas non plus réalisable dans $\quadomoduli[0](a_{1}-2,a_{2}-2;\rec[-2][s-1])$. Si c'est parce que $C_{z}$ est triangulaire, alors $d=z$ et $C$ est donc aussi une configuration triangulaire (ce qui est exclu par hypothèse). Si $C_{z}$ est proportionnelle à une configuration arithmétique, alors $C$ (qui compte précisément les mêmes valeurs) est aussi une configuration arithmétique. En fait, les deux sont de somme paire (en l'occurrence $T$ et $T-2z$). Puisque $T \geq 2s+2$ mais que $T-2z \leq 2s-4$ (puisque les deux sont de somme paire et que $C_{z}$ est interdit), alors $z \geq 3$. Or, $T=2d+(s-2)z$ donc $T-2z=2d+(s-4)z$. Ainsi, $T\leq 2s-4$ est impossible même si $d=1$.
\par
Si pour chaque valeur possible de $d$, la configuration~$C_{d}$ est arithmétique de somme totale insuffisante (les bornes diffèrent si la somme totale, après normalisation, est paire ou impaire), alors nous considérons deux cas.
\par
Supposons d'abord qu'il existe une valeur $d$ apparaissant au moins trois fois dans $C$. Dans ce cas, comme $C_{d}$ est proportionnelle à une configuration arithmétique et compte les mêmes valeurs que $C$, nous supposerons que $C$ est une configuration arithmétique normalisée de somme $T$. La configuration $C_{d}$ est aussi de somme paire $T-2d \leq 2s-4$ tandis que $T \geq 2s+2$. Il s'ensuit que $d \geq 3$ et donc au moins trois éléments de $C_{d}$ valent $1$. Ainsi, $C$ compte aussi trois éléments qui valent $1$. Sans perte de généralité, nous aurions donc pu supposer $d=1$ et obtenir une contradiction avec $d \geq 3$. 
\par
Désormais, nous supposerons que chaque valeur apparaît seulement une ou deux fois dans la configuration $C$. C'est donc aussi le cas pour chaque configuration $C_{d}$. Pour que leur somme normalisée soit strictement inférieure à $2(s-1)$, $C_{d}$ ne peut valoir (après normalisation) que $(1,1)$, $(2,1)$, $(2,1,1)$, $(3,1,1)$, $(2,2,1)$, $(2,2,1,1)$, $(3,2,1,1)$ ou $(3,2,2,1,1)$. Dans chaque cas, $d$ est par hypothèse distinct des valeurs restantes dans $C_{d}$. Dans chaque cas pour lequel~$2$ apparaît deux fois, nous aurions pu prendre $2$ au lieu de $d$ et obtenir une configuration réalisable. Nous nous restreignons donc aux cas dans lesquels $C_{d}$ est proportionnelle à $(1,1)$, $(2,1)$, $(2,1,1)$, $(3,1,1)$ ou $(3,2,1,1)$. Les configurations correspondantes pour~$C$ sont donc proportionnelles à:
\begin{enumerate}[(a)]
    \item $(1,1,d,d)$ avec $d \neq 1$;
    \item $(2,1,d,d)$ avec $d \notin \lbrace{ 1,2 \rbrace}$;
    \item $(2,1,1,d,d)$ avec $d \notin \lbrace{ 1,2 \rbrace}$;
    \item $(3,1,1,d,d)$ avec $d \notin \lbrace{ 1,3 \rbrace}$;
    \item $(3,2,1,1,d,d)$ avec $d \notin \lbrace{ 1,2,3 \rbrace}$.
\end{enumerate}
Le cas (a) est une configuration en crosse pour la strate $\quadomoduli[0](1,3;\rec[-2][4])$. Comme la configuration $(2,1)$ est réalisable dans la strate $\quadomoduli[0](1,-1;\rec[-2][2])$, le cas (b) n'est pas un problème. Le cas (c) est réglé par le lemme~\ref{lem:double33} pour la strate $\quadomoduli[0](3,3;\rec[-2][5])$. En revanche, pour la strate $\quadomoduli[0](1,5;\rec[-2][4])$, on s'appuie sur le fait que chaque configuration $(2,d,d)$ est réalisable dans $\quadomoduli[0](-1,3;\rec[-2][4])$. Le cas (d) est identique. Enfin, pour le cas (e), la configuration $(3,2,d,d)$ est toujours réalisable dans les strates voulues.
\end{proof}

\subsection{Les configurations avec $\mathfrak{i} = 2$ et $n \geq 3$}

Nous prouvons d'abord que les seules obstructions pour de telles strates sont arithmétiques.

\begin{lem}\label{lem:n3}
Toute configuration non arithmétique de résidus non nuls est réalisable dans la strate $\quadomoduli[0](a_{1},a_{2},2l_{3},\dots,2l_{n};\rec[-2][s])$ avec deux zéros impairs et un nombre quelconque de zéros pairs.
\end{lem}

\begin{proof}
Les différentielles de ces strates peuvent s'obtenir en éclatant un zéro de différentielles des strates $\quadomoduli[0](a_{1},a_{2}+2l_{3}+\dots+2l_{n};\rec[-2][s])$ et $\quadomoduli[0](a_{2},a_{1}+2l_{3}+\dots+2l_{n};\rec[-2][s])$. Pour ces deux strates, les résultats de la section~\ref{sub:i2n2double} nous disent que les seules obstructions non arithmétiques sont les configurations triangulaires ou en crosse (voir définition~\ref{def:crosse}). Elles concernent les deux strates à la fois uniquement dans le cas des strates $\quadomoduli[0](a,a,2;\rec[-2][s])$ avec $a$ impair et $s$ pair. Il s'agit donc de montrer que les configurations en crosse qui ne sont pas arithmétiques sont toujours réalisables dans de telles strates.

Pour $\lambda \in \mathbb{C} \setminus \mathbb{Q}$, la configuration $(1,\dots,1,4\lambda^{2})$ est réalisable dans  $\quadomoduli[0](a,a;\rec[-2][s-1])$. On colle sur le pôle double de résidu $4\lambda^{2}$ le carré d'une différentielle abélienne de la strate $\Omega\mathcal{M}_{0}(1,\rec[-1][3])$ de résidus abéliens $(2\lambda,-\lambda,-\lambda)$. En lissant, on obtient une différentielle de $\quadomoduli[0](a,a,2;\rec[-2][s])$ qui réalise la configuration $(1,\dots,1,\lambda^{2},\lambda^{2})$.
\end{proof}

Pour les configurations arithmétiques, nous distinguons selon la parité de la somme $T$ des circonférences des cylindres.

\begin{prop}\label{prop:F1}
Étant donnée une strate $\quadomoduli[0](a_{1},a_{2},2l_3,\dots,2l_n;\rec[-2][s])$ avec deux zéros impairs $a_{1}<a_{2}$ et un nombre quelconque de zéros pairs et une configuration arithmétique $(r_{1}^{2},\dots,r_{s}^{2})$, si la somme $\sum r_{i}$ est impaire et vérifie $\sum r_{i} \geq a_{2}+2$ alors il existe une différentielle de cette strate avec ces résidus.
\end{prop}

\begin{proof}
Le résultat se prouve par éclatement des zéros de la proposition~\ref{prop:eclatZero}. Le cas des strates $\quadomoduli[0](a_{1},a_{2};\rec[-2][s])$  est prouvé dans la proposition~\ref{prop:polesdoublesimpair}. En éclatant un zéro d'ordre $a_{1}+2l_{3}$ on obtient le résultat pour les strates $\quadomoduli[0](a_{1},a_{2},2l_{3};\rec[-2][s])$ lorsque la somme des racines des résidus $T$ est supérieure ou égale à $2s-1$. Pour ces strates, il reste donc à considérer les cas où  $a_{1} + 2 \leq T \leq 2s - 3$. Les autres strates s'obtiennent en éclatant le zéro d'ordre pair.
 
Soit $C$ un $s$-uplet tel que $a_{1} + 2 \leq T \leq 2s - 3$. Il contient au moins trois éléments égaux à $1$. On forme un ensemble $C_{1}$ de $l_{3}+1$ éléments de $C$ tels que si le complémentaire $C_{2}$ n'est pas vide, alors il contient au moins un~$1$. Par le théorème~1.2 de \cite{getaab} il existe une différentielle abélienne $\omega_{1}$ de $\omoduli[0](l_{3};\rec[-1][l_{3}+2])$ dont les résidus abéliens sont les éléments de~$C_{1}$ et l'opposé de leur somme. De plus, il existe une différentielle $\eta_{2}$ de $\quadomoduli[0](a_{1},a_{2};\rec[-2][s-l_{3}])$ dont les résidus sont les carrés des éléments de $C_{2}$ et le carré de la somme des éléments de $C_{1}$. La différentielle avec les invariants souhaités est obtenue en lissant la différentielle obtenue en collant $\eta_{2}$ à $\omega_{1}^{2}$ le long du pôle de résidu maximal.
\end{proof}

Nous traitons enfin le cas des configurations arithmétiques de somme paire.

\begin{prop}\label{prop:H1}
Étant donnée une strate $\quadomoduli[0](a_{1},a_{2},2l_3,\dots,2l_n;\rec[-2][s])$ avec deux zéros impairs $a_{1}\leq a_{2}$ et un nombre quelconque de zéros pairs et une configuration arithmétique $(r_{1}^{2},\dots,r_{s}^{2})$. Si la somme $\sum r_{i}$ est paire et vérifie $\sum r_{i} \geq a_{1}+a_{2}+4$, alors il existe une différentielle de cette strate avec ces résidus.
\end{prop}

\begin{proof}
Considérons tout d'abord le cas des strates  $\quadomoduli[0](a_{1},a_{2},2l_{3};\rec[-2][s])$.
Si un uplet est réalisable dans $\quadomoduli[0](a_{1}+2l_{3},a_{2};\rec[-2][s])$ ou $\quadomoduli[0](a_{1},a_{2}+2l_{3};\rec[-2][s])$, alors l'éclatement de zéros permet de le réaliser dans $\quadomoduli[0](a_{1},a_{2},2l_{3};\rec[-2][s])$. Nous
distinguons deux cas selon que $T \geq 2s$ ou $2s - 2l_{3} \leq T \leq 2s- 2$.

Si $T \geq 2s$, alors $r$ est réalisable dans $\quadomoduli[0](a_{1}+2l_{3},a_{2};\rec[-2][s])$ ou $\quadomoduli[0](a_{1},a_{2}+2l_{3};\rec[-2][s])$ sauf si elle est triangulaire ou en crosse (voir proposition~\ref{prop:polesdoublespair}). Les seuls cas problématiques sont les résidus en crosse (de la forme $(A,\dots,A,B,B)$) dans les strates $\quadomoduli[0](a_{1},a_{1},2;\rec[-2][s])$. Dans ce cas $s$ est pair supérieur ou égal à $4$. Pour réaliser ces résidus il suffit de lisser la différentielle stable formée en collant les pôles doubles de résidu $4A^{2}$ du carré d'une différentielle abélienne de $\omoduli[0](1;-1,-1,-1)$ dont les résidus sont $(A,A,2A)$ et d'une différentielle de $\quadomoduli[0](a_{1},a_{2};\rec[-2][s-1])$ dont les résidus sont les carrés de $(A,\dots,A,2A,B,B)$. Il suffit donc de montrer que ce dernier cas est possible.
Si $s\geq 6$, alors $A,2A$ et $B$ sont premiers entre eux et il existe donc une différentielle de $\quadomoduli[0](a_{1},a_{1},2;\rec[-2][s])$ dont les résidus sont les carrés de ces éléments. Si $s=4$ et $B$ est pair, alors $A+B$ est impair, car $A$ et $B$ sont premiers entre eux. Par la proposition~\ref{prop:polesdoublesimpair}, dans ce cas la condition est simplement que la somme des éléments est supérieure ou égale à $3$. Ce uplet est donc toujours réalisable.
\par
Nous traitons à présent le cas $2s- 2l_{3} \leq T \leq  2s - 2$. Cette condition implique qu’au moins deux éléments de $r$ sont égaux à $1$. Parmi les $s$ éléments, on en choisit $l_{3}+1$ (sachant que $2l_{3}\leq 2s-2$) tels que s'il reste des éléments, au moins l’un d’entre eux est égal à $1$. On considère une différentielle abélienne $\omega_{1}$ de $\omoduli[0](l_{3};\rec[-1][l_{3}+2])$ dont les résidus abéliens sont les éléments choisis et l'opposé $S$ de leur somme. Nous allons construire une différentielle~$\omega_{2}$ de $\quadomoduli[0](a_{1},a_{2};\rec[-2][s-l_{3}])$ telle que les résidus sont donnés par les carrés des éléments restants et de $S$. S’il ne restait aucun résidu, alors $a_{1}=a_{2}=-1$ et l'existence de $\omega_{2}$ vient simplement du fait que la strate $\quadomoduli[0](-1,-1;-2)$ est non vide. S’il restait des résidus, au moins l’un d’entre eux vaut $1$. Pour construire $\omega_{2}$, il faut simplement vérifier que le uplet $r'$ formé par ces éléments et $S$  est réalisable dans la strate. Notons que leur pgcd est $1$ car $r'$ contient $1$.
\par
La somme de  $r'=( A_{1} ,\dots, A_{q} , S)$ est supérieure ou égale à $2s-2l_{3}=a_{1}+a_{2}+4$. Par la proposition~\ref{prop:polesdoublespair} il suffit de vérifier qu’elle n'est ni triangulaire ni en crosse. Pour que la configuration ne soit pas en crosse, il suffit  d’avoir choisi les $l_{3}+1$ plus grands éléments de $r$ pour que $S$ soit strictement supérieur aux autres éléments de~$r'$. Si la configuration est triangulaire, alors la configuration est de la forme $(A,B,\dots,B, S)$ avec $S= A + B$. Cela contredit le fait que $S$ est la somme des plus grands éléments de $r$.
\par
Finalement considérons le cas des strates $\quadomoduli[0](a_{1},a_{2},2l_{3},\dots,2l_{n};\rec[-2][s])$ avec $n\geq4$. Les différentielles avec les résidus souhaités sont obtenues par éclatement du zéro pair de la strate $\quadomoduli[0](a_{1},a_{2},2l_{3}+\dots+2l_{n};\rec[-2][s])$.
\end{proof}

\subsection{Réalisation des configurations pour $\mathfrak{i} \geq 4$}

Il n'y a pas d'obstruction dans ce cas. Une construction s'appuyant sur le cas $n=2$ prouvé dans la section~\ref{sub:i2n2double} permet de réaliser chaque configuration de résidus quadratiques dans chacune de ces strates. Ceci démontre le théorème~\ref{thm:surjimp4} dans le cas $r=p=0$.

\begin{prop}\label{prop:quadsurjbcpimp}
L'application résiduelle de la strate $\Omega^{2}\moduli[0](a_{1},\dots,a_{n};\rec[-2][s])$ avec $\nim\geq4$ zéros impairs est surjective.
\end{prop}

\begin{proof}
On commence par supposer que $n=4$ et les $a_{i}$ sont impairs. Les cas $n>4$ s'obtiennent alors par éclatement de zéros. Posons $s_{1}=\frac{a_{1}+a_{2}}{2}$ et $s_{2}=\frac{a_{3}+a_{4}}{2}$. Pour chaque configuration de résidus $(R_{1},\dots,R_{s})$, il est facile de trouver un nombre complexe $R \in \mathbb{C}^{\ast}$ tel que les uplets $(R,R_{1},\dots,R_{s_{1}})$ et $(R,R_{s_{1}+1},\dots,R_{s})$ ne vérifient aucune condition de résonance (voir section~\ref{sec:arrhyp}).
Ainsi, le cas $n=2$ du théorème~\ref{thm:geq0quad2} prouve l'existence de différentielles quadratiques réalisant ces deux configurations de résidus dans les strates respectives $\Omega^{2}\moduli[0](a_{1},a_{2};\rec[-2][s_{1}+1])$ et $\Omega^{2}\moduli[0](a_{3},a_{4};\rec[-2][s_{2}+1])$. Dans chacune de ces surfaces de translation, nous découpons une géodésique fermée d'un cylindre correspondant au pôle de résidu $R$. En recollant les deux surfaces l'une sur l'autre le long de ce bord, nous obtenons la surface réalisant la configuration $(R_{1},\dots,R_{s})$ dans la strate $\Omega^{2}\moduli[0](a_{1},a_{2},a_{3},a_{4};\rec[-2][s])$.
\end{proof}

\section{Différentielles quadratiques en genre supérieur}
\label{sec:ggeq1}

Dans cette section, nous donnons l'image de l'application résiduelle pour chaque composante connexe de strate de genre $g \geq 1$, prouvant les théorèmes~\ref{thm:ggeq2} et \ref{thm:geq1}.
\par
Nous commençons par un rappel sur les composantes connexes des strates de différentielles quadratiques méromorphes dans la section~\ref{sec:rapquad}.
Nous considérons dans la section~\ref{sec:pluri1} les strates ayant au moins un pôle d'ordre strictement inférieur à~$-2$. Dans la section~\ref{sec:pluri2} nous traitons les cas où tous les pôles sont d'ordre~$-2$.

\subsection{Rappels sur les composantes connexes des strates}
\label{sec:rapquad}
%

En genre $1$, les composantes connexes des strates sont caractérisées par le nombre de rotation $\rot(S)$ de la surface plate~$S$ associée à la différentielle quadratique $\xi$ (voir la section~2.4 de \cite{chge}).  Pour une surface plate $S$ définie par une différentielle de $\quadomoduli[1](a_{1},\dots,a_{n};-b_{1},\dots,-b_{p})$ avec une base symplectique de lacets lisses de l'homologie $(\alpha,\beta)$ le {\em nombre de rotation} est $$\rot(S):=\pgcd(a_{1},\dots,a_{n};b_{1},\dots,b_{p},\ind(\alpha),\ind(\beta))\,,$$ où $\ind(\alpha)$ est l'indice de l’application de Gauss du lacet $\alpha$. Notons que l'indice est égal au nombre de demi-tours effectués par le vecteur tangent. On a alors le résultat suivant.
\begin{itemize}
 \item[i)] Si $n=p=1$, la strate est $\quadomoduli[1](a;-a)$ avec $a\geq2$ et chaque composante connexe correspond à un nombre de rotation qui est un diviseur strict de $a$.
 \item[ii)] Sinon, il existe une composante connexe correspondant à chaque nombre de rotation qui est un diviseur de $\pgcd(a_{1},\dots,a_{n};b_{1},\dots,b_{p})$.
\end{itemize}
La composante de $\quadomoduli[1](\mu)$ associée au nombre de rotation $\rho$ est notée $\quadomoduli[1]^{\rho}(\mu)$. La restriction de l'application résiduelle $\appresquad[1](\mu)$ à cette composante est notée $\appresquadrho[1](\mu)$.
\par
Le nombre de rotation peut être calculé de manière algébrique (voir la proposition~3.13 de \cite{chge}). De ce fait, on déduit directement le résultat suivant.
\begin{lem}\label{lem:rotprim}
 Une composante $\quadomoduli[1]^{\rho}(\mu)$ paramètre des différentielles quadratiques primitives si et seulement si $\rho$ est impair.
\end{lem}
En éclatant un zéro (voir la proposition~\ref{prop:eclatZero}) on obtient le résultat suivant.
\begin{lem}
 Le nombre d'une rotation d'une différentielle  obtenue en éclatant un zéro d'une différentielle de nombre de rotation $\rho$ est égal à $\rho$  modulo les ordres des nouveaux zéros.
\end{lem}
\par
Afin de simplifier les énoncés dans les deux prochaines sections, nous introduisons la notation suivante.
\begin{defn}
Une composante de genre $1$ qui possède au moins un pôle d'ordre strictement inférieur à~$-2$ est {\em exceptionnelle} si elle est l'une des suivantes:
\begin{enumerate}
  \item la composante primitive de $\Omega^{2}\moduli[1](4a;\rec[-4][a])$  ou $\Omega^{2}\moduli[1](2a-1,2a+1;\rec[-4][a])$ pour tout $a\geq1$;
  \item la composante primitive de $\Omega^{2}\moduli[1](2s;\rec[-2][s])$ ou $\Omega^{2}\moduli[1](s-1,s+1;\rec[-2][s])$ avec $s$ un entier pair non nul;
  \item les composantes $\Omega^{2}\moduli[1]^{1}(6;-6)$, $\Omega^{2}\moduli[1]^{1}(3,3;-6)$ et $\Omega^{2}\moduli[1]^{3}(12;-6,-6)$.
 \end{enumerate}
\end{defn}
\smallskip
\par
Le cas des strates de genre $g\geq 2$ est donné en détail dans le théorème~1.3 de \cite{chge}. Ces strates sont connexes sauf pour quatre types de décompositions $\mu$ pour lesquelles une composante hyperelliptique existe:
\begin{itemize}
 \item $\mu=(2m_{1},-2m_{2})$,
 \item $\mu=(m_{1},m_{1},-2m_{2})$,
 \item $\mu=(2m_{1},-m_{2},-m_{2})$,
 \item $\mu =(m_{1},m_{1},-m_{2},-m_{2})$,
\end{itemize}
où $m_{1},m_{2}>0$ ne sont pas tous les deux pairs et $2m_{1}-2m_{2}=4g-4$. Notons que ces deux conditions impliquent que $m_{1}$ et $m_{2}$ sont impairs. On a donc le résultat suivant.
\begin{lem}\label{lem:nocomphyp}
Les strates $\quadomoduli( \mu)$    telles que  $\mu=(2m_{1},-2m_{2})$ ou  $\mu=(m_{1},m_{1},-2m_{2})$ avec $m_{1},m_{2}$ impairs $\geq1$ sont les seules strates possédant au moins un pôle pair avec une  composante hyperelliptique primitive.
\end{lem}
De plus, les résultats de \cite{chge} impliquent qu'il existe une autre composante primitive, sauf dans les strates $\quadomoduli[1](2;-2)$ et $\quadomoduli[1](1,1;-2)$

Dans la preuve nous utiliserons principalement le fait que toutes les composantes connexes (même en genre un) peuvent être obtenues en combinant la couture d'anse à partir d'une différentielle de genre zéro (voir la proposition~6.3 de \cite{chge}) et l'éclatement de singularités (par la proposition~6.4 de \cite{chge}).
De plus, la composante hyperelliptique est obtenue en cousant une anse de nombre de rotation maximal. Notons que la différentielle de genre un que l'on colle pour effectuer cette opération n'est pas primitive en général. La composante non-hyperelliptique est obtenue en cousant une anse de n'importe quel nombre de rotation strictement inférieur et peut donc être choisie primitive. Notons que cela n'est pas explicitement écrit dans \cite{chge}, mais cela se déduit de sa section~7, en particulier la remarque~7.3.

\subsection{Différentielles avec un pôle d'ordre strictement inférieur à~$-2$}
\label{sec:pluri1}

Nous montrons tout d'abord que l'application résiduelle est surjective pour les composantes non exceptionnelles de genre~$1$ possédant un unique zéro. Dans le cas des composantes exceptionnelles, l'application résiduelle contient le complémentaire de l'origine (lemmes~\ref{lem:geq1} et \ref{lem:geq1bis}). Puis nous montrons la surjectivité des applications résiduelles de certaines strates (lemme~\ref{lem:geq1ter}). Nous en déduisons la surjectivité de l'application résiduelle dans le cas général par éclatement de zéro et couture d'anse (lemme~\ref{lem:geq1qua}). Comme nous avons montré que l'origine n'est pas  l'image de l'application résiduelle des composantes exceptionnelles dans le cas~(1), il restera à montrer dans le lemme~\ref{lem:compexept} qu'elle n'est pas dans celles du cas~(3).
\smallskip
\par
Commençons par le cas des strates de genre $1$ avec un unique zéro et un unique pôle.
\begin{lem}\label{lem:geq1}
La restriction de l'application résiduelle $\appresquad[1](m;-m)$ à chaque composante non exceptionnelle est surjective pour $m\geq3$. L'image de l'application résiduelle des composantes $\quadomoduli[1](4;-4)$ et $\quadomoduli[1]^{1}(6;-6)$ contient $\CC^{\ast}$.
\end{lem}

\begin{proof}
Si $m$ est impair, la surjectivité de $\appresquadrho[1](m;-m)$ est équivalente au fait que la composante $\quadomoduli[1]^{\rho}(m;-m)$ soit non vide. Ceci est une conséquence élémentaire du théorème d'Abel. \`A partir de maintenant, nous supposons que $m$ est pair. 

Une différentielle dans $\quadomoduli[1](2\ell,-2\ell)$ avec $\ell\geq2$ et au moins un résidu non nul est donnée par le recollement de la partie polaire non triviale d'ordre $2\ell$ et de type $(\rho-1)/2$ associée aux vecteurs $(v_{1},v_{2};-v_{2},-v_{1})$. Une telle différentielle est représentée en blanc sur la figure~\ref{fig:a-a}.
 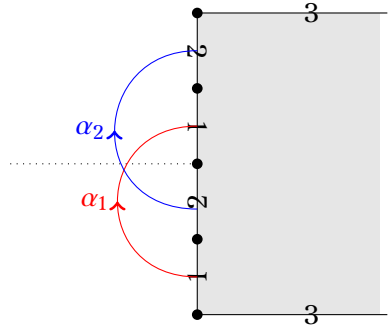
\begin{figure}[htb]
 \centering
\begin{tikzpicture}[scale=1,decoration={
    markings,
    mark=at position 0.5 with {\arrow[very thick]{>}}}]
\fill[fill=black!10] (0,0)  -- (0,4)  -- (2.4,4) -- (2.4,0) -- cycle;

      \draw (0,0) coordinate (a1) -- node [sloped] {$1$} (0,1) coordinate (a2) -- node [sloped] {$2$} (0,2) coordinate (a3) -- node [sloped,rotate=180] {$1$} (0,3) coordinate (a4) -- node [sloped,rotate=180] {$2$} (0,4) coordinate (a5);
  \foreach \i in {1,2,...,5}
  \fill (a\i) circle (2pt);

  \draw (a1)-- ++(2.5,0)coordinate[pos=.6](b);
    \draw (a5)-- ++(2.5,0)coordinate[pos=.6](c);
      \draw[dotted] (a3)-- ++(-2.5,0);
    \node[] at (b) {$3$};
     \node[] at (c) {$3$};

      \draw[postaction={decorate},red] (0,.5) .. controls ++(180:1.4)  and ++(180:1.4) ..node[left] {$\alpha_{1}$} (0,2.5);
  \draw[postaction={decorate},blue](0,1.4) .. controls ++(180:1.5)  and ++(180:1.4) ..node[left] {$\alpha_{2}$} (0,3.5);
\end{tikzpicture}
\caption{Différentielles de $\Omega^{2}\moduli[1](4;-4)$ (en blanc) et $\Omega^{2}\moduli[1](2;-2)$ (en gris)}
\label{fig:a-a}
\end{figure}
Rappelons que l'indice est le nombre de demi-tours que fait le vecteur tangent. L'indice des lacets~$\alpha_{i}$ qui relient les milieux des segments $v_{i}$ est donc $\rho$. On en déduit que le nombre de rotation de la différentielle quadratique ainsi obtenue est égal à $\rho$.

Nous traitons maintenant le cas des composantes non exceptionnelles $\Omega^{2}\moduli[1]^{\rho}(2\ell ;-2\ell)$ avec $\ell\geq3$ et pour des différentielles dont le résidu au pôle est nul. Une différentielle avec ces invariants locaux est obtenue à partir de deux parties polaires. La première est d'ordre $\rho$, de type $(\rho -1)/2$ et associée à $(v_{1};v_{2})$ avec $v_{1}=v_{2}$. La seconde est d'ordre $\ell-\rho$, de type $\ell-\rho$ et associée à $(v_{1};v_{2})$. Nous coupons ces deux domaines polaires le long des demi-droites horizontales commençant au point final des $v_{i}$. Nous collons alors les fentes par translations et les vecteurs $v_{i}$ entre eux par rotation. Cette construction est illustrée par la figure~\ref{fig:a-a,pasderes} dans le cas de la composante $\quadomoduli[1]^{3}(6;-6)$. L'indice des lacets $\alpha$ et $\beta$ entre les milieux des $v_{i}$, comme représentés sur la figure~\ref{fig:a-a,pasderes}, est égal à~$\pm\rho$.
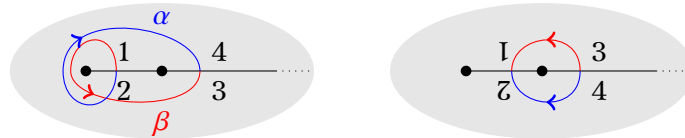
\begin{figure}[htb]
 \centering
\begin{tikzpicture}[scale=1,decoration={
    markings,
    mark=at position 0.5 with {\arrow[very thick]{>}}}]

\begin{scope}[xshift=-1.5cm]
      \fill[fill=black!10] (0,0) ellipse (2cm and .9cm);
   \draw (-1,0) coordinate (a) -- node [below] {$2$} node [above] {$1$}  (0,0) coordinate (b);
 \draw (0,0) -- (1.5,0) coordinate[pos=.5] (c);
  \draw[dotted] (1.5,0) -- (2,0);
 \fill (a)  circle (2pt);
\fill[] (b) circle (2pt);
\node[above] at (c) {$4$};
\node[below] at (c) {$3$};

 \draw[postaction={decorate},red] (-.6,0) .. controls ++(90:.6)  and ++(90:.5) .. (-1.2,0)  .. controls         ++(-90:.5) and ++(-90:.6) .. (.5,0);
  \draw[postaction={decorate},blue] (-.6,0) .. controls ++(-90:.6)  and ++(-90:.6) .. (-1.3,0)  .. controls         ++(90:.9) and ++(90:.6) .. (.5,0);
  
     \node[blue] at (0,.7) {$\alpha$};
  \node[red] at (0,-.7) {$\beta$};
    \end{scope}

\begin{scope}[xshift=3.5cm]
      \fill[fill=black!10] (0,0) ellipse (2cm and .9cm);
   \draw (-1,0) coordinate (a) -- node [above,rotate=180] {$2$} node [below,rotate=180] {$1$}  (0,0) coordinate (b);
 \draw (0,0) -- (1.5,0) coordinate[pos=.5] (c);
  \draw[dotted] (1.5,0) -- (2,0);
 \fill (a)  circle (2pt);
\fill[] (b) circle (2pt);
\node[above] at (c) {$3$};
\node[below] at (c) {$4$};

 \draw[postaction={decorate},red] (.5,0) .. controls ++(90:.6)  and ++(90:.5) .. (-.4,0);
  \draw[postaction={decorate},blue] (.5,0) .. controls ++(-90:.6)  and ++(-90:.5) .. (-.4,0);
    \end{scope}
\end{tikzpicture}
\caption{Différentielle de $\Omega^{2}\moduli[1]^{3}(6;-6)$ dont le résidu est nul}
\label{fig:a-a,pasderes}
\end{figure}

Cela donne une différentielle quadratique souhaitée, sauf dans le cas où le nombre de rotation est $\rho=1$. Dans ce cas, on choisit $1<\rho'<\ell -1$ premier avec $2\ell$. Cela est toujours possible lorsque $\ell \geq 4$. On fait la même construction que ci-dessus et on obtient une différentielle quadratique de nombre de rotation $\pgcd(\rho',\rho',2\ell)=1$.
\end{proof}

Nous déduisons du résultat précédent le cas des strates de genre $1$ avec un unique zéro mais un profil de pôles arbitraire.
\begin{lem}\label{lem:geq1bis}
Étant donnée $\mu=(a;-b_{1},\dots,-b_{p};-c_{1},\dots,-c_{r};\rec[-2][s])$ une décomposition de~$0$ telle que $p+r \geq 2$ et $\rho$ un nombre de rotation. Si $\mu\neq(4p;\rec[-4][p])$, alors l'application résiduelle~$\appresquadrho[1](\mu)$ est surjective. De plus, l'image de $\appresquad[1](4p,\rec[-4][p])$ contient $\CC^{p}\setminus\lbrace 0\rbrace$.
\end{lem}

\begin{proof}
Nous commençons la preuve par le cas des strates connexes puis adaptons les arguments au cas non connexe.
\smallskip
\par
\paragraph{\bf Les strates connexes.}
Nous commençons par le cas où $r\geq2$. On a montré dans les lemmes~\ref{lem:g=0gen1nouvbis} et~\ref{lem:g=0gen1nouv} que l'application résiduelle de $\quadomoduli[0](a-4;-b_{1},\dots,-b_{p};-c_{1},\dots,-c_{r};\rec[-2][s])$ est surjective. On obtient donc la surjectivité de  $\appresquadrho[1](\mu)$ par couture d'anse sauf dans le cas de la composante $\quadomoduli[1]^{1}(6;-3,-3)$. Dans ce dernier cas, la surjectivité est équivalente au fait que cette composante est non vide, ce qui est vrai par le théorème d'Abel.
\par
Supposons maintenant qu'il existe un unique pôle d'ordre impair. Le lemme~\ref{lem:g=0gen1nouv} dans le cas des strates $\omoduli[0](a-4;-b_{1},\dots,-b_{p};-c;\rec[-2][s])$ implique que $\CC^{p}\setminus\lbrace 0\rbrace$ est contenu dans l'image de $\appresquadrho[1](\mu)$ par couture d'anse. Il suffit donc de prouver que l'origine est dans l'image de l'application résiduelle. En particulier, ces strates ne possèdent pas de pôles d'ordre~$-2$.  Nous choisissons un pôle d'ordre $-2\ell_{1}$ et nous lui associons une partie polaire d'ordre~$2\ell_{1}$, de type $(\rho-1)/2$ associée aux vecteurs $(v_{1},v_{2},-2(v_{1}+v_{2});-v_{2},-v_{1})$ avec $v_{1}=v_{2}$. On associe aux autres pôles d'ordre pair les parties polaires d'ordres $2\ell_{i}$ associées aux vecteurs $(2(v_{1}+v_{2});2(v_{1}+v_{2}))$. Au pôle d'ordre $c$, on choisit une partie polaire d'ordre $c$ associée à $(2(v_{1}+v_{2}),\emptyset)$. On colle alors les pôles d'ordre paire de manière cyclique et au bout le pôle d'ordre $c$. Enfin on identifie les vecteurs $v_{i}$ ensemble. Les indices des lacets qui joignent les milieux des $v_{i}$ est~$\rho$. Comme~$\rho$ divise $2\ell_{1}$, on obtient toutes les composantes. 
\par
Il nous reste à traiter le cas où tous les pôles sont d'ordre pair. Nous commençons par construire une différentielle dans $\quadomoduli[1](a;-2\ell_{1},\dots,-2\ell_{p})$ dont tous les résidus sont nuls, sauf dans le cas où $\mu=(4p;\rec[-4][p])$ et $\mu=(12;-6,-6)$ avec $\rho=3$.  L'idée générale est la suivante. Supposons que $\ell_{1}\geq3$ et partons de la différentielle de $\quadomoduli[1](2\ell_{1};-2\ell_{1})$ sans résidu donnée dans la preuve du lemme~\ref{lem:geq1}. On peut alors couper cette surface le long d'un lien-selle associé au vecteur $v_{1}$. Par exemple, on coupe la différentielle de la figure~\ref{fig:a-a,pasderes} le long du lien-selle dénoté par~$1$. Pour tous les autres pôles on prend une partie polaire triviale d'ordre~$2\ell_{i}$ associée à $(v_{1};v_{1})$. La surface obtenue en recollant les segments de manière cyclique, comme sur la figure~\ref{fig:nonresg1bis}, possède les invariants désirés.
\begin{figure}[htb]
\begin{tikzpicture}


\begin{scope}[xshift=-4.5cm]
      
    \fill[fill=black!10] (0,0) ellipse (2cm and .7cm);

   \draw (-1,0) coordinate (a) -- node [below] {$2$} node [above] {$1$}  (0,0) coordinate (b);
 \draw (0,0) -- (1.5,0) coordinate[pos=.5] (c);
  \draw[dotted] (1.5,0) -- (2,0);
 \fill (a)  circle (2pt);
\fill[] (b) circle (2pt);
\node[above] at (c) {$b$};
\node[below] at (c) {$a$};

    \end{scope}

\begin{scope}[xshift=0cm]
    \fill[fill=black!10] (0,0) ellipse (2cm and .7cm);
   \draw (-1,0) coordinate (a) -- node [above,rotate=180] {$2$} node [above] {$3$}  (0,0) coordinate (b);
 \draw (0,0) -- (1.5,0) coordinate[pos=.5] (c);
  \draw[dotted] (1.5,0) -- (2,0);
 \fill (a)  circle (2pt);
\fill[] (b) circle (2pt);
\node[above] at (c) {$a$};
\node[below] at (c) {$b$};

\end{scope}

%

\begin{scope}[xshift=4cm]
    \fill[fill=black!10] (-.5,0) ellipse (1cm and .5cm);
   \draw (0,0) coordinate (a1) -- node [below] {$1$} node [above,xscale=-1] {$3$}  (-1,0) coordinate (a2);
  \foreach \i in {1,2}
  \fill (a\i) circle (2pt);

\end{scope}

\end{tikzpicture}
\caption{Différentielle de $\Omega^{2}\moduli[1](10;-4,-6)$ dont le résidu est nul} 
\label{fig:nonresg1bis}
\end{figure}
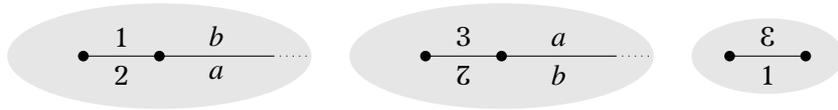
\smallskip
\par
\paragraph{\bf Les strates non connexes.}
Notons que si la strate possède un pôle d'ordre~$-4$, alors elle est connexe. On supposera donc que les pôles sont d'ordres $-2\ell_{i}\leq -6$. 
Si tous les pôles sont d'ordre $-6$, alors on procède de la façon suivante. Si $p=2$, la construction ci-dessus donne une différentielle quadratique avec un nombre de rotation égal à $1$. En effet, les indices des lacets sont respectivement $3$ et $7$. Supposons maintenant qu'il y a $p\geq 3$ pôles d'ordre~$-6$. 
On choisit le type des parties polaires de telle façon que l'indice du lacet qui les traverse est multiple de $3$ pour obtenir le nombre de rotation $3$ et n'est pas multiple sinon. Comme chaque pôle peut contribuer à l'indice par $2$ ou $4$ il est facile de vérifier que cela est toujours possible.

Les cas où il y a au moins trois pôles d'ordres identiques ou deux pôles d'ordres distincts sont obtenus de façon analogue. Nous allons donc nous concentrer sur le cas où il y a deux pôles d'ordres $-2\ell \leq -8$. Dans ce cas, pour les nombres de rotations $\rho < \ell$, la construction précédente fonctionne avec les adaptations nécessaires. Pour  $\rho = \ell$, on prend trois vecteurs $v_{1}=v_{2}=v_{3}$ et on associe au premier pôle $\ell-1$ parties polaires d'ordre $4$ ouvertes à droite associées à $(v_{1};v_{2})$, $(\emptyset;v_{2})$, $(v_{3};\emptyset)$ et les autres associées à $(\emptyset;\emptyset)$. On colle ces parties polaires de manière cyclique de sorte que l'indice entre les milieux des segments $v_{2}$ est $\ell$ et entre les milieux des segments $v_{1}$ et $v_{3}$ est $\nim<\ell$. Pour le second pôle, on considère la partie polaire d'ordre $2\ell$ et de type $(\rho-i)/2$ associée à $(v_{3};v_{1})$. En collant les segments de même label, on obtient une différentielle quadratique avec les invariants souhaités. Cette construction est représentée dans la figure~\ref{fig:nonresgdeuxpoles} dans le cas $\ell=5$.
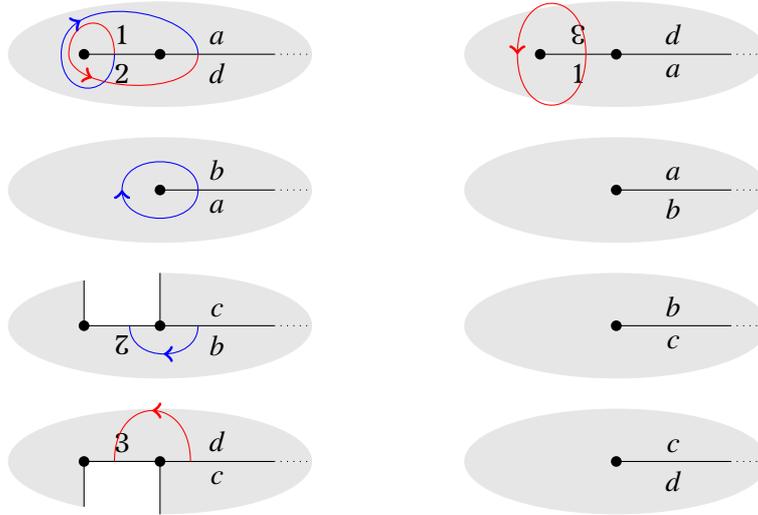
\begin{figure}[htb]
\begin{tikzpicture}[scale=1,decoration={
    markings,
    mark=at position 0.5 with {\arrow[very thick]{>}}}]


\begin{scope}[xshift=0cm]
      
    \fill[fill=black!10] (0,0) ellipse (2cm and .7cm);

   \draw (-1,0) coordinate (a) -- node [below] {$2$} node [above] {$1$}  (0,0) coordinate (b);
 \draw (0,0) -- (1.5,0) coordinate[pos=.5] (c);
  \draw[dotted] (1.5,0) -- (2,0);
 \fill (a)  circle (2pt);
\fill[] (b) circle (2pt);
\node[above] at (c) {$a$};
\node[below] at (c) {$d$};

\draw[postaction={decorate},red] (-.6,0) .. controls ++(90:.6)  and ++(90:.5) .. (-1.2,0)  .. controls         ++(-90:.5) and ++(-90:.6) .. (.5,0);
  \draw[postaction={decorate},blue] (-.6,0) .. controls ++(-90:.6)  and ++(-90:.6) .. (-1.3,0)  .. controls         ++(90:.9) and ++(90:.6) .. (.5,0);
    \end{scope}

\begin{scope}[yshift=-1.8cm]
    \fill[fill=black!10] (0,0) ellipse (2cm and .7cm);
 \draw (0,0) -- (1.5,0) coordinate[pos=.5] (c);
  \draw[dotted] (1.5,0) -- (2,0);
\fill[] (0,0) circle (2pt);
\node[above] at (c) {$b$};
\node[below] at (c) {$a$};

\draw[postaction={decorate},blue] (.5,0) .. controls ++(-90:.5)  and ++(-90:.5) .. (-.5,0)  .. controls         ++(90:.5) and ++(90:.5) .. (.5,0);
\end{scope}

\begin{scope}[yshift=-3.6cm]
    \fill[fill=black!10] (0,0) ellipse (2cm and .7cm);
   \draw (-1,0) coordinate (a) -- node [above,rotate=180] {$2$}   (0,0) coordinate (b);
   \fill[white] (a) -- (b) -- ++(0,.8) --++(-1,0) -- cycle;
   \draw (a) -- ++(0,.6);
\draw (b) -- ++(0,.7);
\draw (a) -- (b);
 \draw (0,0) -- (1.5,0) coordinate[pos=.5] (c);
  \draw[dotted] (1.5,0) -- (2,0);
 \fill (a)  circle (2pt);
\fill[] (b) circle (2pt);
\node[above] at (c) {$c$};
\node[below] at (c) {$b$};

\draw[postaction={decorate},blue] (.5,0) .. controls ++(-90:.5)  and ++(-90:.5) .. (-.4,0);
\end{scope}

\begin{scope}[yshift=-5.4cm]
      \fill[fill=black!10] (0,0) ellipse (2cm and .7cm);
   \draw (-1,0) coordinate (a) --  node [above] {$3$}  (0,0) coordinate (b);
   \fill[white] (a) -- (b) -- ++(0,-.8) --++(-1,0) -- cycle;
   \draw (a) --(b);
   \draw (a) -- ++(0,-.6);
\draw (b) -- ++(0,-.7);
 \draw (0,0) -- (1.5,0) coordinate[pos=.5] (c);
  \draw[dotted] (1.5,0) -- (2,0);
 \fill (a)  circle (2pt);
\fill[] (b) circle (2pt);
\node[above] at (c) {$d$};
\node[below] at (c) {$c$};

\draw[postaction={decorate},red] (.4,0) .. controls ++(90:.9)  and ++(90:.9) .. (-.6,0);
\end{scope}


\begin{scope}[xshift=6cm]
    \fill[fill=black!10] (0,0) ellipse (2cm and .7cm);
   \draw (-1,0) coordinate (a) -- node [below] {$1$} node [below,rotate=180] {$3$}  (0,0) coordinate (b);
 \draw (0,0) -- (1.5,0) coordinate[pos=.5] (c);
  \draw[dotted] (1.5,0) -- (2,0);
 \fill (a)  circle (2pt);
\fill[] (b) circle (2pt);
\node[above] at (c) {$d$};
\node[below] at (c) {$a$};

\draw[postaction={decorate},red] (-.4,0) .. controls ++(90:.9)  and ++(90:.9) .. (-1.3,0)  .. controls         ++(-90:.9) and ++(-90:.9) .. (-.4,0);
\end{scope}

\begin{scope}[xshift=6cm,yshift=-1.8cm]
       \fill[fill=black!10] (0,0) ellipse (2cm and .7cm);
 \draw (0,0) -- (1.5,0) coordinate[pos=.5] (c);
  \draw[dotted] (1.5,0) -- (2,0);
\fill[] (0,0) circle (2pt);
\node[above] at (c) {$a$};
\node[below] at (c) {$b$};
\end{scope}

\begin{scope}[xshift=6cm,yshift=-3.6cm]
        \fill[fill=black!10] (0,0) ellipse (2cm and .7cm);
 \draw (0,0) -- (1.5,0) coordinate[pos=.5] (c);
  \draw[dotted] (1.5,0) -- (2,0);
\fill[] (0,0) circle (2pt);
\node[above] at (c) {$b$};
\node[below] at (c) {$c$};
\end{scope}

\begin{scope}[xshift=6cm,yshift=-5.4cm]
        \fill[fill=black!10] (0,0) ellipse (2cm and .7cm);
 \draw (0,0) -- (1.5,0) coordinate[pos=.5] (c);
  \draw[dotted] (1.5,0) -- (2,0);
\fill[] (0,0) circle (2pt);
\node[above] at (c) {$c$};
\node[below] at (c) {$d$};
\end{scope}
\end{tikzpicture}
\caption{Différentielle de $\Omega^{2}\moduli[1](20;-10,-10)$ dont les résidus sont nuls et de nombre de rotation $5$} 
\label{fig:nonresgdeuxpoles}
\end{figure}
\smallskip
\par
Nous traitons maintenant les strates $\Omega^{2}\moduli[1](a;-2\ell_{1},\dots,-2\ell_{p};\rec[-2][s])$, telles que $s\neq0$ ou $s=0$ et avec au moins un résidu non nul. Dans le cas où le résidu d'un pôle est nul, alors nous supposerons que le résidu de $P_{1}$ est nul.
Considérons les pôles $P_{i}$ avec $2\leq i\leq p'$ qui possèdent un résidu $R_{i}$ non nul. Nous associons à $P_{i}$ la partie polaire non triviale d'ordre~$2\ell_{i}$ associée à $(r_{i};\emptyset)$ où $r_{i}$ est une racine carrée de $R_{i}$ de partie réelle positive (ou de partie imaginaire positive si la partie réelle est nulle). Considérons maintenant les pôles~$P_{j}$ avec $j > p'$ ayant un résidu nul. Nous associons la partie polaire triviale d'ordre $2\ell_{j}$ associée à $(r_{i_{j}};r_{i_{j}})$ pour un résidu quadratique $R_{i_{j}}\neq0$. Puis nous collons le segment~$r_{i}$ du domaine positif de~$P_{i}$ au segment~$r_{i_{j}}$ du domaine basique négatif de~$P_{j}$. 
\par
Enfin, pour le pôle $P_{1}$ nous procédons à la construction suivante. Notons que la somme des~$r_{i}$ est non nulle.  Nous supposerons que les $r_{i}$ sont ordonnés par argument croissant. Nous prenons pour $P_{1}$ la partie polaire d'ordre type $2\ell_{1}$ et de type $(\rho-1)/2$ associée à $(v_{1},v_{2};-v_{2},-v_{1},r_{2},\dots,r_{p'})$ où les $v_{i}$ vérifient les égalités $v_{1}=v_{2}$ et  $r_{1}=2v_{1}+2v_{2}-\sum_{i\geq2} r_{i}$. 
\par
La différentielle est obtenue en identifiant par translation les bords $r_{i}$ des domaines polaires positifs aux segments $r_{i}$ de la partie polaire négative de $P_{1}$. Enfin, nous identifions par rotation d'angle $\pi$ et translation, le premier $v_{1}$ au premier~$v_{2}$ et le second $v_{1}$ au second~$v_{2}$. Cela donne une différentielle quadratique primitive avec les invariants souhaités. Un exemple est donné dans la figure~\ref{fig:nonresg1ter}. 
    \begin{figure}[htb]
\begin{tikzpicture}
\begin{scope}[xshift=-3cm]
\fill[fill=black!10] (0.5,0)coordinate (Q)  circle (1.2cm);
    \coordinate (a) at (0,0);
    \coordinate (b) at (1,0);

     \fill (a)  circle (2pt);
\fill[] (b) circle (2pt);
    \fill[white] (a) -- (b) -- ++(0,-1.2) --++(-1,0) -- cycle;
 \draw  (a) -- (b);
 \draw (a) -- ++(0,-1.1);
 \draw (b) -- ++(0,-1.1);

\node[above] at (Q) {$r_{2}$};
    \end{scope}

\begin{scope}[xshift=3cm]
\fill[fill=black!10] (0.5,0)coordinate (Q)  circle (1.2cm);
    \coordinate (a) at (0,0);
    \coordinate (b) at (1,0);

     \fill (a)  circle (2pt);
\fill[] (b) circle (2pt);
    \fill[white] (a) -- (b) -- ++(0,-1.2) --++(-1,0) -- cycle;
 \draw  (a) -- (b);
 \draw (a) -- ++(0,-1.1);
 \draw (b) -- ++(0,-1.1);

\node[above] at (Q) {$r_{3}$};
    \end{scope}

%

\begin{scope}[xshift=.5cm,yshift=-.3cm]
\fill[fill=black!10] (0,0)coordinate (Q)  circle (1.5cm);
    \coordinate (a) at (-1,0);
    \coordinate (b) at (0,0);
    \coordinate (c) at (1,0);
    \coordinate (e) at (-.5,0);
    \coordinate (f) at (.5,0);

\fill (a)  circle (2pt);
\fill[] (b) circle (2pt);
\fill (c)  circle (2pt);
   \fill (-.4,0)  arc (0:180:2pt); 
   \fill (.6,0)  arc (0:180:2pt); 

\draw (a) -- (b)coordinate[pos=.5] (g1) --(c)coordinate[pos=.5] (g2) -- (f)coordinate[pos=.5] (g3) -- (b)coordinate[pos=.5] (g4) -- (e)coordinate[pos=.5] (g5)-- (a)coordinate[pos=.5] (g6);

\node[below] at (g1) {$r_{2}$};
\node[below] at (g2) {$r _{3}$};
\node[above] at (g3) {$v_{2}$};
\node[above] at (g4) {$v_{1}$};
\node[below,rotate=180] at (g5) {$v_{2}$};
\node[below,rotate=180] at (g6) {$v_{1}$};

\draw[dotted] (b) -- ++(0,1.5);
\draw[dotted] (c) -- ++(0,-1);
\end{scope}

\end{tikzpicture}
\caption{Différentielle de $\Omega^{2}\moduli[1](12;\rec[-4][3])$ dont les résidus sont $(0,1,1)$} \label{fig:nonresg1ter}
\end{figure}
\end{proof}

Nous traitons maintenant le cas des strates avec deux zéros et des pôles d'ordre~$-4$.
\begin{lem}\label{lem:geq1ter}
Pour tout $p\geq1$ et tout $(a_{1},a_{2})\neq(2p+1,2p-1)$, les images des applications résiduelles $\appresk[1][2](a_{1},a_{2};\rec[-4][p])$ et $\appresk[2][2](4(p+1);\rec[-4][p])$ contiennent l'origine.
\end{lem}

\begin{proof}
Pour commencer la figure~\ref{fig:exceptquadra} fournit une différentielle quadratique primitive ayant un résidu nul dans chacune des strates $\Omega^{2}\moduli[1](5,-1;-4)$ et $\Omega^{2}\moduli[1](2,2;-4)$ (en gris).

 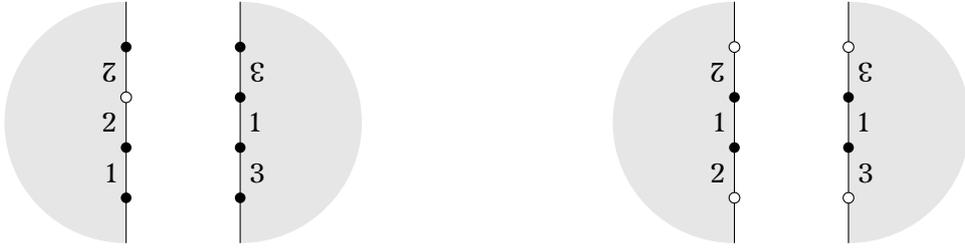
\begin{figure}[htb]
\begin{tikzpicture}
\begin{scope}[xshift=-4cm]

     \foreach \i in {1,2,...,4}
  \coordinate (a\i) at (0,2*\i/3); 
   \fill[black!10] (a1) -- (a4)-- ++(0,.6) arc (90:270:1.6) -- cycle;
 
       \foreach \i in {1,2,4}
   \fill (a\i)  circle (2pt);

        \foreach \i in {1,2,...,4}
  \coordinate (b\i) at (1.5,2*\i/3); 
     \fill[black!10] (b1) -- (b4)-- ++(0,.6) arc (90:-90:1.6) -- cycle;
       \foreach \i in {1,2,...,4}
   \fill (b\i)  circle (2pt);

\draw (a1)--+(0,-.6);
\draw (b1)--+(0,-.6);
\draw (a4)--+(0,.6);
\draw (b4)--+(0,.6);

\draw (a1)-- (a2) coordinate[pos=.5] (c1) -- (a3) coordinate[pos=.5] (c2) -- (a4) coordinate[pos=.5] (c3);
\draw (b1)-- (b2) coordinate[pos=.5] (d1) -- (b3) coordinate[pos=.5] (d2) -- (b4) coordinate[pos=.5] (d3);

      \fill[white] (a3)  circle (2pt);
      \draw (a3)  circle (2pt);
      
\node[left] at (c1) {$1$};
\node[left] at (c2) {$2$};
\node[right,rotate=180] at (c3) {$2$};
\node[right] at (d1) {$3$};
\node[right] at (d2) {$1$};
\node[left,rotate=180] at (d3) {$3$};

\end{scope}

\begin{scope}[xshift=4cm]
     \foreach \i in {1,2,...,4}
  \coordinate (a\i) at (0,2*\i/3); 
     \fill[black!10] (a1) -- (a4)-- ++(0,.6) arc (90:270:1.6) -- cycle;

       \foreach \i in {2,3}
   \fill (a\i)  circle (2pt);

        \foreach \i in {1,2,...,4}
  \coordinate (b\i) at (1.5,2*\i/3); 
       \fill[black!10] (b1) -- (b4)-- ++(0,.6) arc (90:-90:1.6) -- cycle;

       \foreach \i in {2,3}
   \fill (b\i)  circle (2pt);

\draw (a1)--+(0,-.6);
\draw (b1)--+(0,-.6);
\draw (a4)--+(0,.6);
\draw (b4)--+(0,.6);

\draw (a1)-- (a2) coordinate[pos=.5] (c1) -- (a3) coordinate[pos=.5] (c2) -- (a4) coordinate[pos=.5] (c3);
\draw (b1)-- (b2) coordinate[pos=.5] (d1) -- (b3) coordinate[pos=.5] (d2) -- (b4) coordinate[pos=.5] (d3);

       \foreach \i in {1,4}
      \fill[white] (a\i)  circle (2pt);
             \foreach \i in {1,4}
      \draw (a\i)  circle (2pt);
             \foreach \i in {1,4}
      \fill[white] (b\i)  circle (2pt);
             \foreach \i in {1,4}
      \draw (b\i)  circle (2pt);
      
\node[left] at (c2) {$1$};
\node[left] at (c1) {$2$};
\node[right,rotate=180] at (c3) {$2$};
\node[right] at (d1) {$3$};
\node[right] at (d2) {$1$};
\node[left,rotate=180] at (d3) {$3$};
\end{scope}
\end{tikzpicture}
\caption{Différentielles des strates $\Omega^{2}\moduli[1](5,-1;-4)$ et $\Omega^{2}\moduli[1](2,2;-4)$ dont le résidu est nul en gris. Différentielles de $\Omega^{2}\moduli[1](5,-1;-2,-2)$ et $\Omega^{2}\moduli[1](2,2;-2,-2)$ dont le résidu est $(1,1)$ en blanc.}
\label{fig:exceptquadra}
\end{figure}
\par
Nous prouvons maintenant le résultat pour les strates $\Omega^{2}\moduli[1](a_{1},a_{2};\rec[-4][p])$ satisfaisant $(a_{1},a_{2})\neq(2p+1,2p-1)$ par récurrence sur le nombre $p$ de pôles. L'hypothèse de récurrence est la suivante. Il existe une différentielle quadratique primitive dont les résidus sont nuls dans toutes les strates de cette forme avec $p-1$ pôles telles que si $a_{1}=-1$, alors il existe un lien-selle fermé reliant le zéro d'ordre $a_{2}$ à lui-même et dans tous les cas il existe un lien-selle entre les deux zéros. L'hypothèse de récurrence est satisfaite pour $p=1$ par les différentielles quadratiques représentées dans la figure~\ref{fig:exceptquadra}. Nous allons construire des différentielles satisfaisant à l'hypothèse de récurrence avec $p$ pôles.
\par
Si $(a_{1},a_{2})=(4p+1,-1)$, coupons la différentielle quadratique de $\Omega^{2}\moduli[1](4p-3,-1;\rec[-4][p-1])$ donnée par l'hypothèse le long du lien-selle entre le zéro d'ordre~$4p-3$. Prenons une partie polaire d'ordre~$4$ associée à $(v;v)$ où $v$ est la période de ce lien-selle. La surface formée en collant les bords de ces surfaces par translation est dans $\Omega^{2}\moduli[1](4p+1,-1;\rec[-4][p])$. Ses résidus sont nuls, elle possède un lien-selle fermé reliant le zéro d'ordre $4p+1$ à lui-même et un autre entre les deux zéros.
\par
Si $(a_{1},a_{2})\neq(4p+1,-1),(2p+1,2p-1)$, la construction du paragraphe précédent en partant d'une différentielle de $\Omega^{2}\moduli[1](a_{1}-2,a_{2}-2;\rec[-4][p-1])$ et du lien-selle entre les deux zéros donne une différentielle quadratique ayant les propriétés souhaitées.
\par
Une différentielle quadratique dans la strate $\Omega^{2}\moduli[2](4(p+1);\rec[-4][p])$ sans résidus aux pôles est donnée de la façon suivante. Le lemme~\ref{lem:geq1bis} fournit une différentielle primitive de la strate $\Omega^{2}\moduli[1](4(p+1);\rec[-4][p];-2,-2)$ telle que les résidus aux pôles d'ordre $-4$ sont nuls et les résidus quadratiques aux pôles doubles sont égaux entre eux. Nous formons une $2$-différentielle entrelacée en collant les deux pôles d'ordre $-2$ ensemble. Cette différentielle quadratique entrelacée peut être lissée sans changer les résidus aux pôles d'ordre $-4$ (voir proposition~\ref{lem:lisspolessimples}) pour donner une différentielle quadratique ayant les invariants souhaités. 
\end{proof}

Nous montrons maintenant que les applications résiduelles de toutes les composantes non exceptionnelles sont surjectives.
\begin{lem}\label{lem:geq1qua}
Les applications résiduelles des composantes non exceptionnelles de $\quadomoduli(\mu)$ avec $g\geq1$ sont surjectives.
\end{lem}

\begin{proof}
Si $\mu$ possède un unique zéro on obtient les différentielles dans les composantes souhaitées par couture d'anses à partir des différentielles quadratiques de genre $1$ ayant un unique zéro.
De plus, ces différentielles de genre $1$ peuvent être prises primitives sauf dans le cas des composantes hyperelliptiques. La surjectivité de l'application résiduelle est une conséquence de la surjectivité des applications résiduelles de ces strates (voir lemme~\ref{lem:geq1} et \ref{lem:geq1bis}), à l’exception des strates de la forme $\Omega^{2}\moduli[g](4p+4g-4;\rec[-4][p])$, des composantes hyperelliptiques et éventuellement de certaines composantes des strates $\quadomoduli(6+4g;-6)$ et $\quadomoduli(12+4g;-6,-6)$. De plus dans les cas non-hyperelliptiques, on obtient le complémentaire de l'origine de cette façon.

Dans les cas non-hyperelliptiques la configuration de résidus $(0,\dots,0)$ est obtenue en collant des anses à partir des strates $\Omega^{2}\moduli[2](4(p+1);\rec[-4][p])$: le lemme~\ref{lem:geq1ter} montre qu'elles contiennent l'origine.

Par le Théorème~1.3 de \cite{chge}, les strates primitives $\quadomoduli(12+4g;-6,-6)$ sont connexes et donc il suffit de partir de la composante $\quadomoduli[1]^{1}(12;-6,-6)$. La strate $\quadomoduli[2](10;-6)$ possède une composante hyperelliptique et une composante non-hyperelliptique. Dans le cas non-hyperelliptique, on colle une anse de nombre de rotation $1$ à une différentielle de la composante $\quadomoduli[1]^{3}(6;-6)$. On obtient la surjectivité dans ce cas. Le cas des composantes non-hyperelliptiques de $\quadomoduli(6+4g;-6)$ de genres supérieurs s'obtient en collant des anses aux différentielles de $\quadomoduli[2](10;-6)$.

Pour les composantes hyperelliptiques, on sait par le lemme~\ref{lem:nocomphyp} que $\mu=(2m_{1},-2m_{2})$. On part d'une différentielle de genre $1$ de nombre de rotation $m_{2}$ de la strate $\quadomoduli[1](2m_{2},-2m_{2})$. Comme $m_{2}$ est impair, l'application résiduelle de ces strates est surjective.  On colle des anses de nombres de rotations maximaux. Notons que si les différentielles que l'on colle sont des carrés de différentielles abéliennes, le théorème~1.1  de  \cite{getaab} montre qu'il existe une telle différentielle dont le résidu est nul. On obtient donc la surjectivité de l'application résiduelle de ces composantes.
\smallskip
\par
Considérons les strates $\quadomoduli(a_{1},\dots,a_{n};-c_{i};-b_{j};\rec[-2][s])$ avec $n\geq2$ zéros. La surjectivité de l'application résiduelle est obtenue en éclatant l'unique zéro des différentielles quadratiques de la strate $\quadomoduli(\sum a_{i};-c_{i};-b_{j};\rec[-2][s])$ (voir proposition~\ref{prop:eclatZero}), sauf éventuellement pour $\Omega^{2}\moduli[1](a_{1},\dots,a_{n};\rec[-4][p])$, $\Omega^{2}\moduli[1]^{1}(a_{1},\dots,a_{n};-6)$ et $\Omega^{2}\moduli[1]^{3}(a_{1},\dots,a_{n};-6,-6)$ avec $n\geq2$. Dans le  cas $\Omega^{2}\moduli[1](a_{1},a_{2};\rec[-4][p])$ la surjectivité a été démontrée dans le lemme~\ref{lem:geq1ter}. Notons que la seule strate $\Omega^{2}\moduli[1](a_{1},a_{2};-6)$ qui possède une composante de nombre de rotation égal à $3$ est $\Omega^{2}\moduli[1](3,3;-6)$.
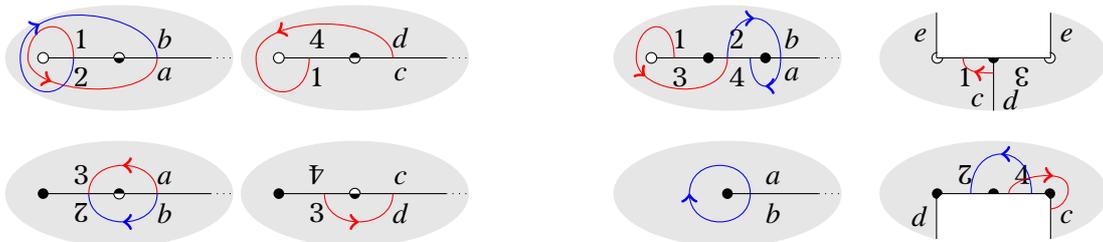
\begin{figure}[htb]
\begin{tikzpicture}[scale=1,decoration={
    markings,
    mark=at position 0.5 with {\arrow[very thick]{>}}}]
%
\begin{scope}[xshift=0cm]
\fill[fill=black!10] (0,0) ellipse (1.5cm and .7cm);

\draw (-1,0) coordinate (a) -- node [below] {$2$} node [above] {$1$}  (0,0) coordinate (b);
\draw (0,0) -- (1.2,0) coordinate[pos=.5] (c);
\draw[dotted] (1.2,0) -- (1.5,0);
 \filldraw[fill=white] (a)  circle (2pt);
\fill[] (b) circle (2pt);
   \filldraw[fill=white] (.07,0)  arc (0:180:2pt); 
\node[above] at (c) {$b$};
\node[below] at (c) {$a$};

 \draw[postaction={decorate},red] (-.6,0) .. controls ++(90:.6)  and ++(90:.5) .. (-1.2,0)  .. controls         ++(-90:.5) and ++(-90:.6) .. (.5,0);
  \draw[postaction={decorate},blue] (-.6,0) .. controls ++(-90:.6)  and ++(-90:.6) .. (-1.3,0)  .. controls         ++(90:.9) and ++(90:.6) .. (.5,0);
    \end{scope}

\begin{scope}[yshift=-1.8cm]
    \fill[fill=black!10] (0,0) ellipse (1.5cm and .7cm);
   \draw (-1,0) coordinate (a) -- node [above,rotate=180] {$2$} node [above] {$3$}  (0,0) coordinate (b);
 \draw (0,0) -- (1.2,0) coordinate[pos=.5] (c);
  \draw[dotted] (1.2,0) -- (1.5,0);
 \fill (a)  circle (2pt);
\fill[] (b) circle (2pt);
   \filldraw[fill=white] (.07,0)  arc (0:-180:2pt); 
\node[above] at (c) {$a$};
\node[below] at (c) {$b$};

\draw[postaction={decorate},blue] (.5,0) .. controls ++(-90:.5)  and ++(-90:.5) .. (-.4,0);
\draw[postaction={decorate},red] (.5,0) .. controls ++(90:.5)  and ++(90:.5) .. (-.4,0);
\end{scope}

%

\begin{scope}[xshift=3.1cm]
      
    \fill[fill=black!10] (0,0) ellipse (1.5cm and .7cm);

   \draw (-1,0) coordinate (a) -- node [below] {$1$} node [above] {$4$}  (0,0) coordinate (b);
 \draw (0,0) -- (1.2,0) coordinate[pos=.5] (c);
  \draw[dotted] (1.2,0) -- (1.5,0);
 \filldraw[fill=white] (a)  circle (2pt);
\fill[] (b) circle (2pt);
   \filldraw[fill=white] (.07,0)  arc (0:-180:2pt); 
\node[above] at (c) {$d$};
\node[below] at (c) {$c$};

  \draw[postaction={decorate},red] (.5,0) .. controls ++(90:.6)  and ++(90:.6) .. (-1.3,0)  .. controls         ++(-90:.6) and ++(-90:.6) .. (-.6,0);
    \end{scope}

\begin{scope}[xshift=3.1cm,yshift=-1.8cm]
    \fill[fill=black!10] (0,0) ellipse (1.5cm and .7cm);
   \draw (-1,0) coordinate (a) -- node [below,rotate=180] {$4$} node [below] {$3$}  (0,0) coordinate (b);
 \draw (0,0) -- (1.2,0) coordinate[pos=.5] (c);
  \draw[dotted] (1.2,0) -- (1.5,0);
 \fill (a)  circle (2pt);
\fill[] (b) circle (2pt);
   \filldraw[fill=white] (.07,0)  arc (0:180:2pt); 

\node[above] at (c) {$c$};
\node[below] at (c) {$d$};

\draw[postaction={decorate},red] (-.4,0) .. controls ++(-90:.5)  and ++(-90:.5) .. (.5,0);
\end{scope}


\begin{scope}[xshift=8cm]
\fill[fill=black!10] (0,0) ellipse (1.5cm and .7cm);

\draw (-1,0) coordinate (a) -- node [below] {$3$} node [above] {$1$}  (-.25,0) coordinate (b)-- node [below] {$4$} node [above] {$2$}  (0.5,0) coordinate (c);
\draw (c) -- (1.2,0) coordinate[pos=.5] (d);
\draw[dotted] (1.2,0) -- (1.5,0);
 \filldraw[fill=white] (a)  circle (2pt);
\fill[] (b) circle (2pt);
\fill[] (c) circle (2pt);
\node[above] at (d) {$b$};
\node[below] at (d) {$a$};

 \draw[postaction={decorate},red] (-.7,0) .. controls ++(90:.6)  and ++(90:.5) .. (-1.2,0)  .. controls         ++(-90:.5) and ++(-90:.6) .. (0,0);
 \draw[postaction={decorate},blue] (0,0) .. controls ++(90:.7)  and ++(90:.7) .. (.7,0);
\draw[postaction={decorate},blue] (.7,0) .. controls ++(-90:.5)  and ++(-90:.5) .. (.3,0);


\begin{scope}[yshift=-1.8cm]
    \fill[fill=black!10] (0,0) ellipse (1.5cm and .7cm);
 \draw (0,0) -- (1.2,0) coordinate[pos=.5] (c);
  \draw[dotted] (1.2,0) -- (1.5,0);
\fill[]  (0,0) circle (2pt);
\node[above] at (c) {$a$};
\node[below] at (c) {$b$};

\draw[postaction={decorate},blue] (.3,0) .. controls ++(-90:.5)  and ++(-90:.5) .. (-.5,0)  .. controls         ++(90:.5) and ++(90:.5) .. (.3,0);
\end{scope}


\begin{scope}[xshift=3.5cm]
    \fill[fill=black!10] (0,0) ellipse (1.5cm and .7cm);
    \fill[white] (-.75,0) coordinate (a)-- (.75,0) coordinate (c)-- ++(0,.75)  --++(-1.5,0) -- cycle;
\filldraw[fill=white] (-.67,0)  arc (0:-270:2pt); 
\filldraw[fill=white] (.67,0)  arc (-180:90:2pt); 
\fill(.07,0)  arc (0:-180:2pt); 
   \draw  (a) -- node [below] {$1$}   (0,0) coordinate (b) --  node [above,rotate=180] {$3$} (c);
 \draw (b) -- (0,-.7) coordinate[pos=.8] (d);
\node[right] at (d) {$d$};
\node[left] at (d) {$c$};
\draw (c) -- node[right] {$e$} ++(0,.6) ;
\draw (a) -- node[left] {$e$} ++(0,.6);

\draw[postaction={decorate},red] (0,-.2) .. controls ++(180:.2)  and ++(-90:.2) .. (-.4,0);
    \end{scope}

\begin{scope}[xshift=3.5cm,yshift=-1.8cm]
       \fill[fill=black!10] (0,0) ellipse (1.5cm and .7cm);
    \fill[white] (-.75,0) coordinate (a)-- (.75,0) coordinate (c)-- ++(0,-.75)  --++(-1.5,0) -- cycle;
\fill (-.67,0)  arc (0:270:2pt); 
\fill (.67,0)  arc (180:-90:2pt); 
\fill(.07,0)  arc (0:180:2pt); 
   \draw  (a) -- node [below,rotate=180] {$2$}   (0,0) coordinate (b) --  node [above] {$4$} (c);
\draw (c) -- node[right] {$c$} ++(0,-.6) ;
\draw (a) -- node[left] {$d$} ++(0,-.6);

\draw[postaction={decorate},red] (.2,0) .. controls ++(90:.2)  and ++(150:.2) .. (.9,.2)  .. controls         ++(-30:.2) and ++(0:.2) .. (.75,-.2);
\draw[postaction={decorate},blue] (.5,0) .. controls ++(90:.7)  and ++(90:.7) .. (-.3,0);
\end{scope}
\end{scope}

\end{tikzpicture}
\caption{Différentielle de résidus nuls dans $\Omega^{2}\moduli[1]^{3}(6,6;-6,-6)$  à gauche et dans $\Omega^{2}\moduli[1]^{3}(3,9;-6,-6)$ à droite} 
\label{fig:nonres66}
\end{figure}

Dans le cas $\Omega^{2}\moduli[1](a_{1},a_{2};-6,-6)$, notons que la seule strate qui possède une composante de nombre de rotation égal à $3$ satisfait $(a_{1},a_{2})=(3,9)$ ou $(a_{1},a_{2})=(6,6)$. Une différentielle dans chacune de ces composantes avec un résidu nul est représentée dans la figure~\ref{fig:nonres66}.
Dans tous les cas, pour $n\geq3$, il suffit d'éclater l'un des deux zéros de ces différentielles.
\end{proof}

Le  lemme suivant montre que l'origine n'appartient pas à l'image de l'application résiduelle des composantes exceptionnelles listées dans~(3).
\begin{lem}\label{lem:compexept}
Il n'existe pas de différentielle quadratique dans les composantes $\Omega^{2}\moduli[1]^{1}(6;-6)$,  $\Omega^{2}\moduli[1]^{1}(3,3;-6)$ et  $\Omega^{2}\moduli[1]^{3}(12;-6,-6)$ dont les résidus sont nuls.
\end{lem}

\begin{proof}
Supposons qu'il existe une différentielle dans l'une de ces composantes dont tous les résidus sont nuls. Dans ce cas la proposition~\ref{prop:coeurdege} implique qu'il existe une telle différentielle $\xi$ dont tous les liens-selles sont horizontaux. La proposition~6.2 de \cite{chge} implique que $\xi$ est obtenue en recollant des demi-plans dont le bord est formé de segments horizontaux. Nous faisons maintenant une analyse pour chaque composante.
\smallskip
\par
\paragraph{\bf Le cas de la strate $\Omega^{2}\moduli[1]^{1}(3,3;-6)$.}
La différentielle $\xi$ est obtenue en recollant $4$ demi-plans dont les bords sont formés d'un total de $6$ segments horizontaux identifiés deux à deux. Nous numérotons les demi-plans de manière cyclique.
Comme tous les liens-selles bordent le domaine polaire, deux liens-selles relient les deux zéros et un lien-selle est fermé.  Nous notons  les segments associés aux deux premiers liens-selles par $a$ et $b$, et les segments associés au dernier par $c$.

Supposons que les quatre segments $a$ et $b$ sont dans des demi-plans de même parité. Une base de l'homologie est donnée par deux lacets qui relient respectivement les $a$ et les segments $b$ entre eux. Comme le nombre de rotation est $1$, on peut supposer que les segments~$a$ appartiennent au même demi-plan. Afin d'obtenir une singularité d'angle $5\pi$, il faut que les quatre segments $a$ et $b$ soient dans le même demi-plan. Comme le résidu est nul les segments $c$ sont dans les demi-plans de l'autre parité. En faisant tendre la longueur de $a$ vers zéro et en agrandissant $b$ de la même quantité, on obtient une différentielle de $\Omega^{2}\moduli[1]^{1}(6;-6)$ dont le résidu est nul, ce qui est absurde.

Supposons que les segments $a$ soient dans des demi-plans de même parité (disons pair) et les segments $b$ dans des demi-plans de parité distincte. Comme le résidu est nul les segments~$c$ sont dans les demi-plans impairs. On peut alors faire tendre la longueur des segments $b$ afin d'obtenir une différentielle de $\Omega^{2}\moduli[1]^{1}(6;-6)$ dont le résidu est nul.

Supposons que les segments $a$ et les segments $b$ soient dans deux demi-plans de parités distinctes. Comme le résidu est nul, les segments $c$ sont aussi dans des demi-plans de parité distinctes. On obtient donc le carré d'une différentielle abélienne.

Enfin, supposons que les segments $a$ sont dans des demi-plans d'une même parité (disons paire) et les segments $b$ dans des demi-plans tous deux de l'autre parité (disons impaire). Si les segments $c$ sont de la même parité, on peut fusionner les deux zéros comme précédemment. Si les segments $c$ sont dans des demi-plans de parité distincte, alors l'unique configuration possible pour de tels segments donne une différentielle dont les zéros sont d'ordres $(2,4)$. Cela conclut la preuve de ce cas. 
\smallskip
\par
\paragraph{\bf Le cas de la strate $\Omega^{2}\moduli[1]^{1}(6;-6)$.}
Supposons qu'il existe une différentielle dans $\Omega^{2}\moduli[1]^{1}(6;-6)$ dont le résidu au pôle est nul. Par la proposition~\ref{prop:eclatZero}, on pourrait éclater son zéro en deux zéros d'ordre~$3$ sans changer le résidu au pôle. Cela contredit le fait qu'il n'existe pas une telle différentielle dans $\Omega^{2}\moduli[1]^{1}(3,3;-6)$.
\smallskip
\par
\paragraph{\bf Le cas de la strate $\Omega^{2}\moduli[1]^{3}(12;-6,-6)$.}
La différentielle $\xi$ est obtenue en recollant $8$ demi-plans, partitionnés en deux ensembles de $4$ demi-plans collés ensembles, formant deux domaines polaires. Les bords sont formés d'un total de $6$ segments horizontaux. Comme les résidus sont nuls, les domaines polaires sont bordés soit par $2$ et $4$ segments respectivement, soit les deux par $3$ segments.
\par
Dans le premier cas, considérons le domaine bordé par deux segments $a$ et $b$. Notons que~$a$ et $b$ ne sont pas identifiés l'un à l'autre car sinon la surface serait non connexe. Comme le résidu est nul, les segments $a$ et $b$ sont dans des demi-plans de parité distinctes.  Maintenant l'autre domaine polaire est bordé des segments $a$, $b$ et de deux segments $c$. Si $a$ et $b$ sont dans des demi-plans de parité différentes alors les deux $c$ sont dans des demi-plans de parités différentes. On a donc le carré d'une différentielle quadratique. Si les $a$ et $b$ sont dans le même demi-plan, alors comme le résidu est nul les deux $c$ sont dans les demi-plans de l'autre parité. Donc cette différentielle possède au moins deux singularités. Enfin si $a$ et $b$ sont dans deux demi-plans distincts mais de même parité, alors les $c$ sont chacun dans un demi-plan de l'autre parité. La différentielle ainsi formée possède les invariants souhaités à l'exception du nombre de rotation qui est égal à $1$.
\par
Dans le second cas, soit un segment de $a$, $b$ et $c$ borde chacun des domaines polaires, soit un domaine est bordé par les deux segments $a$ et un $c$ et l'autre par les deux segments $b$ et l'autre $c$. Dans le premier cas, comme les résidus sont nuls, on a deux segments au bord du premier domaine polaire qui sont dans des demi-plans de même parité et l'autre dans un de parité différente. Cette propriété est aussi satisfaite dans l'autre domaine polaire. On obtient donc le carré d'une différentielle abélienne.
Dans le second cas, le bord du premier domaine est formé de deux segments $a$ et un $c$ et le second domaine polaire de deux segments $b$ et le second $c$. Comme le résidu est nul, les segments $a$ sont dans des demi-plans de même parité et $c$ de l'autre. De plus, les deux $a$ sont dans deux demi-plans distincts, car sinon il y aurait un pôle simple. Les mêmes considérations s'appliquent à l'autre pôle. Toutes les différentielles qui peuvent être réalisées de cette façon possèdent au moins deux zéros.
\end{proof}

\subsection{Différentielles dont tous les pôles sont doubles}
\label{sec:pluri2}
Dans cette section, nous considérons les strates $\quadomoduli(a_{1},\dots,a_{n};\rec[-2][s])$ avec $g\geq1$. 
Comme le fait que la configuration $(1,\dots,1)$ n'est pas dans l'image de l'application résiduelle des strates exceptionnelles listées dans~(2) a été montré dans la section~\ref{sec:EXCEPT}, nous nous concentrons sur les constructions.

Nous commençons par traiter le cas des strates de genre~$1$ avec un unique zéro. Nous montrons dans le lemme~\ref{lem:g=1residugen} que l'application résiduelle contient toutes les configurations de résidus quadratiques sauf si tous les résidus sont colinéaires. Puis nous montrons dans le lemme~\ref{lem:g=1quadgen} que la seule exception en genre $1$ peut être les résidus quadratiques proportionnels à $(1,\dots,1)$  dans $\Omega^{2}\moduli[1](2s;\rec[-2][s])$ avec $s$ pair. Puis nous montrons dans le lemme~\ref{lem:g=1quadspe} que $(1,\dots,1)$ est dans l'image de l'application résiduelle des strates $\Omega^{2}\moduli[1](a_{1},a_{2};\rec[-2][s])$ avec $(a_{1},a_{2})\neq(s-1,s+1)$ et $s$ pair. Les composantes générales sont traitées dans le lemme~\ref{lem:ggentousres}.
\smallskip
\par
Nous commençons par traiter le cas des strates en genre un. Dans ce cas, les strates sont connexes. Le lemme suivant traite du cas d'un unique zéro. 
\begin{lem}\label{lem:g=1residugen}
L'application résiduelle de $\Omega^{2}\moduli[1](2s;\rec[-2][s])$ est surjective pour $s=1$ et contient les configurations de résidus $(R_{1},\dots,R_{s})$ telles qu'existent $R_{i}$ et $R_{j}$ vérifiant $\frac{R_{i}}{R_{j}}\notin\RR_{+}$ pour $s\geq2$. 
\end{lem} 

\begin{proof}
La figure~\ref{fig:a-a} montre que la strate $\quadomoduli[1](2;-2)$ est non vide. Nous supposerons donc que les strates ont $s\geq2$ pôles d'ordre~$-2$. Par hypothèse, il existe des racines carrées $r_{i}$ des $R_{i}$ qui génèrent $\CC$ comme $\RR$-espace vectoriel. Sans perte de généralité, on peut supposer que l'argument de chaque $r_{i}$ est dans $\left]-\pi,0\right]$.
Nous concaténons alors les $r_{i}$ par argument croissant. L'intérieur de ce segment brisé se trouve dans le demi-plan supérieur ouvert déterminé par la droite $(DE)$ où $D$ et $E$ sont respectivement les points initial et final de la concaténation.
Nous joignons $D$ à $E$ par quatre segments $v_{i}$ d'égale longueur de la façon suivante. Le segment brisé est formé en concaténant les~$v_{i}$ par $i$ croissant avec $v_{1}=v_{2}$, $v_{3}=v_{4}$ et l'angle au-dessus de l'intersection entre $v_{2}$ et $v_{3}$ est égal à~$\pi$.  Notons que ce segment brisé  se trouve dans le demi-plan inférieur à la droite~$(DE)$.

Nous collons maintenant des demi-cylindres infinis aux segments $r_{i}$ de ce polygone. La surface plate est obtenue en collant $v_{1}$ à $v_{3}$ et $v_{2}$ à $v_{4}$ par translation et rotation. On vérifie facilement que cette différentielle possède un unique zéro. De plus, la différentielle associée est primitive de genre un car on la déconnecte en coupant les courbes correspondant aux $v_{i}$ et une autre courbe fermée quelconque.
\end{proof}

\begin{lem}\label{lem:g=1quadgen}
L'application résiduelle de $\Omega^{2}\moduli[1](2s;\rec[-2][s])$ est surjective si $s$ est impair et contient $\espresk[1][2](2s;\rec[-2][s])\setminus \CC^{\ast}\cdot (1,\dots,1)$ si $s$ est pair.
\end{lem}

\begin{proof}
D'après le lemme~\ref{lem:g=1residugen}, il suffit de considérer le cas où les résidus sont $(r_{1}^{2},\dots,r_{s}^{2})$ avec $r_{i}>0$.

Dans le cas où $s$ est impair, nous classons les $r_{i}$ par ordre croissant. Coupons le segment de longueur $r_{1}$ en deux segments de longueur $\epsilon$ et $r_{1}-\epsilon$. Ensuite, nous coupons le segment de longueur $r_{2}$ en deux segments de longueur $r_{1}-\epsilon$ et  $r_{2}-r_{1}+\epsilon$ en identifiant le premier avec le second du découpage précédent. Nous procédons ainsi pour les $s-1$ premiers cylindres et obtenons en bout de chaîne un segment de longueur $(r_{s-1}-r_{s-2})+\dots+(r_{2}-r_{1})+\epsilon$. Le dernier cylindre est découpé en quatre segments. Le premier d'entre eux est identifié à celui que nous venons d'obtenir. Les trois autres segments se répartissent une longueur $(r_{s}-r_{s-1})+\dots+(r_{3}-r_{2})+r_{1}-\epsilon$ qui est strictement positive. Comme $\epsilon$ peut être choisi arbitrairement petit, nous pouvons répartir cette longueur en trois segments $l+\epsilon+l$ et identifier les deux paires de segments de longueur égale de façon à fermer la surface.

Dans le cas où $s$ est pair, nous découpons le segment de longueur $r_{1}$ en deux segments de longueur $\epsilon$ et $r_{1}-\epsilon$. Ensuite, nous coupons le segment de longueur $r_{2}$ en deux segments de longueur $\alpha$ et $r_{2}-\alpha$. Nous procédons ainsi avec tous les cylindres de largeur $r_{2}$ à $r_{s-1}$ de façon à obtenir une chaîne de $s-2$ cylindres dont les deux segments de bord sont de longueur $\alpha$ et $(r_{s-1}-r_{s-2})+\dots+(r_{3}-r_{2})+\alpha$. Enfin, nous découpons dans le dernier cylindre de largeur $r_{s}$ quatre segments correspondant aux bords restants. La seule condition que nous devons vérifier est que $(r_{s-1}-r_{s-2})+\dots+(r_{3}-r_{2})+r_{1}<r_{s}$, de façon à ce qu'il puisse exister un segment de longueur $\alpha$ strictement positive. L'hypothèse selon laquelle les résidus ne sont pas tous égaux est suffisante.
\end{proof}

Nous traitons maintenant les strates de genre un avec deux zéros.
\begin{lem}\label{lem:g=1quadspe}
L'application résiduelle $\appresk[1][2](a_{1},a_{2};\rec[-2][s])$ avec $s$ pair et $(a_{1},a_{2})\neq(s-1,s+1)$ est surjective.
\end{lem}

\begin{proof}
Par éclatement de zéro, le lemme~\ref{lem:g=1quadgen} implique qu'il suffit de prouver que $(1,\dots,1)$ est dans l'image de l'application résiduelle de ces strates. Nous prouvons ce résultat par récurrence sur le nombre (pair) de pôles d'ordre~$-2$. L'hypothèse de récurrence est que dans chaque strate avec  $s$ pair et $(a_{1},a_{2})\neq(s-1,s+1)$ il existe une différentielle dont les résidus quadratiques sont $(1,\dots,1)$ qui possède un lien-selle entre les deux zéros et un lien-selle fermé reliant le zéro d'ordre maximal à lui même. Ces liens-selles sont choisis non-horizontaux. 
\par
Pour $s=2$, il y les strates $\Omega^{2}\moduli[1](5,-1;-2,-2)$ et $\Omega^{2}\moduli[1](2,2;-2,-2)$ à considérer. Des différentielles quadratiques dans ces strates avec résidus quadratiques $(1,1)$ sont représentées dans la figure~\ref{fig:exceptquadra} en blanc.  L'existence des liens-selles comme ci-dessus est claire.
\par
Supposons par récurrence qu'il existe une différentielle de $\Omega^{2}\moduli[1](a_{1},a_{2};\rec[-2][s])$ avec résidus quadratiques $(1,\dots,1)$ et $(a_{1},a_{2})\neq(s-1,s+1)$. Si $a_{1}\geq a_{2}$, nous construisons des différentielles quadratiques dans les strates $\Omega^{2}\moduli[1](a_{1}+2,a_{2}+2;\rec[-2][s+2])$ et $\Omega^{2}\moduli[1](a_{1},a_{2}+4;\rec[-2][s+2])$ satisfaisant à l'hypothèse de récurrence.
\par
Pour ajouter $4$ à l'ordre du zéro d'ordre maximal, nous coupons la différentielle le long du lien-selle fermé $\gamma$ connectant ce zéro à lui même. Nous prenons un parallélogramme dont l'un des segments est donné par la période de $\gamma$ et les autres arêtes sont les vecteurs~$1$. Puis nous collons des demi-cylindres infinis de circonférence $1$ aux segments de longueur~$1$. Enfin nous collons les autres segments du parallélogramme aux bords de la surface. La différentielle obtenue à partir de la différentielle à gauche de la figure~\ref{fig:exceptquadra} est représentée à gauche de la figure~\ref{fig:exceptquadrabis}. Pour ajouter~$2$ à l'ordre des deux zéros, il suffit de faire la même construction en coupant un lien-selle entre les deux zéros. Cette construction est représentée à droite de la figure~\ref{fig:exceptquadrabis}.

 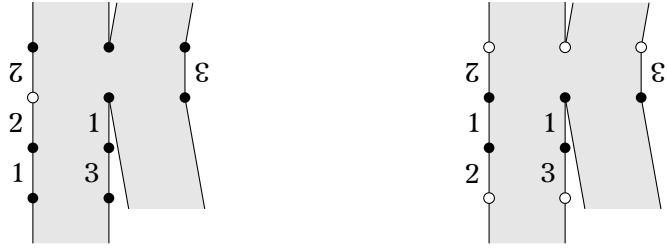
\begin{figure}[htb]
\begin{tikzpicture}
\begin{scope}[xshift=-3cm]

     \foreach \i in {1,2,...,4}
  \coordinate (a\i) at (0,2*\i/3); 
        \foreach \i in {1,2,...,4}
  \coordinate (b\i) at (1,2*\i/3); 
  \coordinate (e1) at (2,8/3);
\coordinate (e2) at (2,2);
   
    \fill[fill=black!10] (a1)   -- (a4) -- ++(0,.6) --++(1,0) -- (b4) -- ++(80:.6)-- ++(1,0) -- (e1) -- (e2)  -- ++(-80:1.5) -- ++(-1,0) -- (b3) -- (b1)-- ++(0,-.6) --++(-1,0) --cycle;

       \foreach \i in {1,2,4}
   \fill (a\i)  circle (2pt);
       \foreach \i in {1,2,...,4}
   \fill (b\i)  circle (2pt);

\draw (a1)--+(0,-.6);
\draw (b1)--+(0,-.6);
\draw (a4)--+(0,.6);
\draw (b4)--+(0,.6);

   \fill (e1)  circle (2pt);
      \fill (e2)  circle (2pt);

\draw (a1)-- (a2) coordinate[pos=.5] (c1) -- (a3) coordinate[pos=.5] (c2) -- (a4) coordinate[pos=.5] (c3);
\draw (b1)-- (b2) coordinate[pos=.5] (d1) -- (b3) coordinate[pos=.5] (d2);
\draw (e2) -- (e1)  coordinate[pos=.5] (d3);

 \fill[white] (a3)  circle (2pt);
\draw (a3)  circle (2pt);

\draw (b3)--+(-80:1.5);
\draw (e2)--+(-80:1.5);
\draw (e1)--+(80:.6);
\draw (b4)--+(80:.6);
      
\node[left] at (c1) {$1$};
\node[left] at (c2) {$2$};
\node[right,rotate=180] at (c3) {$2$};
\node[left] at (d1) {$3$};
\node[left] at (d2) {$1$};
\node[left,rotate=180] at (d3) {$3$};

\end{scope}

\begin{scope}[xshift=3cm]

     \foreach \i in {1,2,...,4}
  \coordinate (a\i) at (0,2*\i/3); 
        \foreach \i in {1,2,...,4}
  \coordinate (b\i) at (1,2*\i/3); 
   \coordinate (e1) at (2,8/3);
\coordinate (e2) at (2,2);
   
       \fill[fill=black!10] (a1)   -- (a4) -- ++(0,.6) --++(1,0) -- (b4) -- ++(80:.6)-- ++(1,0) -- (e1) -- (e2)  -- ++(-80:1.5) -- ++(-1,0) -- (b3) -- (b1)-- ++(0,-.6) --++(-1,0) --cycle;
       
       \foreach \i in {2,3}
   \fill (a\i)  circle (2pt);
       \foreach \i in {2,3}
   \fill (b\i)  circle (2pt);

\draw (a1)--+(0,-.6);
\draw (b1)--+(0,-.6);
\draw (a4)--+(0,.6);
\draw (b4)--+(0,.6);

      \fill (e2)  circle (2pt);

\draw (a1)-- (a2) coordinate[pos=.5] (c1) -- (a3) coordinate[pos=.5] (c2) -- (a4) coordinate[pos=.5] (c3);
\draw (b1)-- (b2) coordinate[pos=.5] (d1) -- (b3) coordinate[pos=.5] (d2);
\draw (e2) -- (e1)  coordinate[pos=.5] (d3);

\draw (b3)--+(-80:1.5);
\draw (e2)--+(-80:1.5);
\draw (e1)--+(80:.6);
\draw (b4)--+(80:.6);

       \foreach \i in {1,4}
      \fill[white] (a\i)  circle (2pt);
             \foreach \i in {1,4}
      \draw (a\i)  circle (2pt);
             \foreach \i in {1,4}
      \fill[white] (b\i)  circle (2pt);
             \foreach \i in {1,4}
      \draw (b\i)  circle (2pt);

   \fill[white] (e1)  circle (2pt);
   \draw (e1) circle (2pt);
      
\node[left] at (c2) {$1$};
\node[left] at (c1) {$2$};
\node[right,rotate=180] at (c3) {$2$};
\node[left] at (d1) {$3$};
\node[left] at (d2) {$1$};
\node[left,rotate=180] at (d3) {$3$};

\end{scope}
\end{tikzpicture}
\caption{Différentielle dans la strate $\Omega^{2}\moduli[1](9,-1;\rec[-2][4])$ à gauche et $\Omega^{2}\moduli[1](4,4;\rec[-2][4])$ à droite dont les résidus sont $(1,1,1,1)$}
\label{fig:exceptquadrabis}
\end{figure}
\par
Pour conclure, il suffit de remarquer que toutes les strates  $\Omega^{2}\moduli[1](a_{1},a_{2};\rec[-2][s])$ satisfaisant $(a_{1},a_{2})\neq(s-1,s+1)$ peuvent s'obtenir en ajoutant $2$ à l'ordre des deux zéros ou $4$ à l'ordre du zéro d'ordre maximal.
\end{proof}

Nous traitons maintenant les strates générales sous la condition que tous les pôles sont doubles.
\begin{lem}\label{lem:ggentousres}
L'application résiduelle de chaque composante de $\quadomoduli(a_{1},\dots,a_{n};\rec[-2][s])$, distinctes des strates $\Omega^{2}\moduli[1](2s;\rec[-2][s])$ et $\Omega^{2}\moduli[1](s+1,s-1;\rec[-2][s])$ avec $s$ pair, est surjective.
\end{lem}

\begin{proof}
Si $g=1$, alors nous avons montré le résultat pour au plus deux zéros dans les lemmes précédents. Pour les strates avec $n\geq3$ zéros, nous utilisons l'éclatement des zéros à partir des strates ayant deux zéros.

Si $g\geq2$ et $s\geq2$, il suffit de montrer que l'image de l'application résiduelle des strates de la forme $\Omega^{2}\moduli[2](2s+4;\rec[-2][s])$ avec $s$ pair contient $(1,\dots,1)$. En effet, si c'est le cas l'éclatement des zéros et la couture d'anse impliqueront le résultat pour toutes les strates.
Pour cela, prenons une différentielle quadratique de la strate $\Omega^{2}\moduli[1](2s+4;\rec[-2][s+2])$ dont les résidus quadratiques sont $(1,\dots,1,-1,-1)$ (cette différentielle existe en vertu du lemme~\ref{lem:g=1residugen}). Nous formons une différentielle quadratique entrelacée en collant les deux pôles avec les résidus quadratiques~$-1$. D'après la proposition~\ref{lem:lisspolessimples}, cette différentielle entrelacée est lissable et la différentielle obtenue par lissage possède les propriétés  souhaitées.  De plus, en vertu de la proposition~6.3 de \cite{chge}, toutes les composantes non-hyperelliptiques connexes peuvent être obtenues de cette façon.

Dans le cas où $g\geq2$ et $s=1$ il existe des composantes hyperelliptiques et non-hyperelliptiques lorsqu'il n'y a qu'un zéro ou deux d'ordres égaux. S'il n'y a qu'un zéro, les deux composantes s'obtiennent par couture d'anse à partir de la strate $\quadomoduli[1](2,-2)$.  On en déduit le résultat dans les deux cas, puis par éclatement des zéros le cas où $n\geq2$.
\end{proof}

\smallskip
\par
\section{Applications}
\label{sec:appli}

Dans cette section finale, nous montrons comment nos résultats permettent d'obtenir facilement le fait que les strates sont (presque) toutes non vides. De plus, nous montrons que le nombre de cylindres disjoints sur une strate est bornée et que cette borne est atteinte.

\subsection{Différentielles quadratiques d'aire finie}
\label{sec:plurifini}

Dans ce paragraphe nous donnons une nouvelle preuve de la caractérisation de \cite{masm} des strates de différentielles quadratiques vides énoncée dans la proposition~\ref{prop:stratesvides}. Ce résultat dit que les strates primitives $\quadomoduli(a_{1},\dots,a_{n})$ en genre $g\geq1$ sont vides si et seulement si $\mu=\emptyset$ ou $\mu=(-1,1)$ en genre $1$ ou $\mu=(4)$ et $\mu=(1,3)$ en genre~$2$.
\smallskip
\par
Toutes les différentielles quadratiques de type $\mu = \emptyset$ sont la puissance d'une différentielle abélienne de $\omoduli[1](\emptyset)$. De plus, la strate $\quadomoduli[1](-1,1)$ est isomorphe à la strate de $1$-formes ayant seulement un pôle et un zéro simples. Le théorème d'Abel implique que cette dernière est vide. Pour les deux strates vides de genre $2$, il suffit de remarquer que le théorème de Riemann-Roch implique que les différentielles quadratiques sont sommes de carrés de différentielles abéliennes. 
\smallskip
\par
Nous montrons que les autres strates sont non vides. 
Soit $\quadomoduli[g](a_{1},\dots,a_{n})$ une strate de genre $g\geq 2$ avec $a_{i}>-2$. Si $(a_{1},\dots,a_{n})$ est différent de $(4g-4)$ et $(2g-1,2g-3)$, alors, d'après le théorème~\ref{thm:geq1}, la strate $\quadomoduli[1](a_{1},\dots a_{n};\rec[-4][g-1])$ contient une différentielle primitive $(X_{0},\omega_{0})$ dont tous les résidus quadratiques sont nuls. On obtient une différentielle quadratique entrelacée en attachant le carré d'une différentielle abélienne sur un tore à chaque pôle de $X_{0}$.
Cette différentielle entrelacée est lissable par le lemme~\ref{lem:lissdeuxcomp}. Les lissages sont clairement des différentielles quadratiques primitives. Donc ces strates sont non vides.

Enfin il reste le cas des strates $\Omega^{2}\moduli[g](4g-4)$ et $\Omega^{2}\moduli[g](2g-1,2g-3)$. En utilisant l'éclatement de zéros, il suffit de montrer que la strate $\Omega^{2}\moduli[g](4g-4)$ est non vide pour tout $g\geq3$. Nous savons que $\Omega^{2}\moduli[2](4(g-1);\rec[-4][g-2])$ contient une différentielle primitive à résidus quadratiques nuls aux pôles (voir Théorème 1.1). On obtient une différentielle quadratique entrelacée en attachant à chaque pôle le carré d'une différentielle de $\omoduli[1](\emptyset)$. Les différentielles quadratiques obtenues par lissage sont dans $\Omega^{2}\moduli[g](4g-4)$ comme souhaité.

\subsection{Cylindres dans les différentielles quadratiques d'aire finie}
\label{sec:cylindres}

Un {\em cylindre} dans une surface plate est une famille continue maximale de géodésiques fermées parallèles. Quand il est d'aire finie, son bord est formé de liens-selles. Soit $\quadomoduli(a_{1},\dots,a_{n})$ une strate de différentielles quadratiques (primitives) avec $\nim$ zéros d'ordres impairs et $\npa$ zéros d'ordres pairs.
Nous allons montrer la proposition~\ref{prop:cylindresquad} qui assure que le nombre maximal de cylindres disjoints dans une différentielle de cette strate est $g+\npa+\tfrac{\nim}{2}-1$ et que cette borne est optimale.
\smallskip
\par
Soit $\xi$ une différentielle quadratique avec $t$ cylindres disjoints. En faisant tendre la hauteur des cylindres vers l'infini, on obtient (après normalisation) une différentielle quadratique entrelacée dont tous les \noeuds sont d'ordre $-2$. Dans chaque composante irréductible, il y a au moins un zéro (ou un pôle simple) et la somme totale des ordres des singularités est paire. Puisque la dégénérescence n'ajoute que des pôles doubles, chaque composante irréductible contient un nombre pair de zéros d'ordre impair. Le nombre $c$ de composantes irréductibles vaut au plus $\npa+\tfrac{\nim}{2}$. Le genre d'un graphe est la différence entre le nombre d'arêtes et le nombre de sommets moins un. Il est clair que le genre du graphe dual de la courbe stable vaut au plus~$g$. Autrement dit $t-(c-1) \leq g$ et donc $t \leq c+g-1 \leq g+\npa+\tfrac{\nim}{2}-1$.
\smallskip
\par
Dans chacune des strates, nous allons construire une différentielle réalisant le nombre maximal de cylindres disjoints. Nous formerons une différentielle entrelacée dont le lissage par le lemme~\ref{lem:lisspolessimples} donnera une différentielle avec les propriétés souhaitées. Dans chaque cas, les $\npa+\tfrac{\nim}{2}$ composantes irréductibles contiendront soit un unique zéro d'ordre pair, soit une paire de singularités d'ordres impairs.
Notons que la valence d'un sommet avec un unique zéro pair est  $\geq3$ tandis que celle d'un sommet avec une paire d'impairs est $\geq1$. Le cas où la valence est $1$ est réalisé lorsque les deux singularités sont des pôles simples.
\par
Traitons d'abord le cas des strates de genre zéro. Considérons un zéro pair d'ordre $a$. On lui associe le  carré d'une différentielle abélienne de $\omoduli[0](\frac{a}{2};\rec[-1][2+\frac{a}{2}])$. Pour une paire de singularités d'ordres impairs $a_{1},a_{2}$ de somme $a$, on lui associe une différentielle de $\omoduli[0](a_{1},a_{2};\rec[-2][2+\frac{a}{2}])$.  Nous traiterons ces différentielles de genre zéro comme les composantes irréductibles de valence $2+\frac{a}{2}$ de la différentielle entrelacée. Il est bien connu que toute distribution de valences de somme $2s-2$ est réalisable par un arbre à~$s$ sommets. Nous pouvons donc assembler les composantes irréductibles selon l'arbre. Comme le graphe est un arbre, quitte à multiplier les différentielles de chaque composante par des constantes, nous pouvons supposer que la somme des résidus sur chaque arête est nulle. Nous obtenons ainsi une différentielle entrelacée qui, après lissage, donne une différentielle dans $\quadomoduli[0](\mu)$ réalisant une famille de $\npa+\tfrac{\nim}{2} -1$ cylindres disjoints.
\par
Le cas général se démontre par récurrence sur le genre. On supposera la borne réalisée dans toute strate primitive de genre au plus $g$. Considérons une strate $\quadomoduli[g+1](a_{1},\dots,a_{n})$. Si l'un de ces ordres, disons $a_{1}$, vérifie $a_{1} \geq 3$ et $a_{1} \neq 4$, alors la borne est réalisée dans la strate $\quadomoduli(a_{1}-4,a_{2},\dots,a_{n})$ si elle est non vide. On obtient une différentielle avec un cylindre de plus dans la strate originale par couture d'anse. Par la proposition~\ref{prop:stratesvides}, les seules strates que l'on ne peut atteindre ainsi sont $\quadomoduli[2](-1,5)$, $\quadomoduli[3](1,7)$ et $\quadomoduli[3](8)$. Le point (ii) du théorème~1.2 de \cite{getaab} permet d'obtenir une différentielle abélienne dans $\quadomoduli[0](4,\rec[-1][6])$ dont les résidus $(r,1,1,-r,-1,-1)$. Identifions par paires les pôles dont les résidus sont respectivement $(A,-A)$, $(1,1)$ et $(-1,-1)$. Le carré de cette différentielle entrelacée est lissable et son lissage est une différentielle quadratique primitive avec trois cylindres. Le théorème~\ref{thm:geq0quad2} permet de conclure de manière analogue pour les deux autres strates.
\par
Il nous reste à traiter le cas où les ordres des singularités sont dans $\lbrace-1, 1, 2,4 \rbrace$. Sur une géodésique d'un cylindre déjà existant, il est facile d'insérer un parallélogramme dont une paire de côtés opposés sera collée sur un nouveau cylindre. On ajoute ainsi un zéro d'ordre~$4$ et le nombre de cylindres disjoints augmente de $2$. Cette chirurgie permet de traiter tous les cas ayant un 
zéro d'ordre $4$ sauf $\quadomoduli[2](-1,1,4)$ et $\quadomoduli[3](4,4)$. 
Dans le premier cas, il s'agit de construire une différentielle de $\quadomoduli[2](-1,1,4)$ ayant trois cylindres. Le point~(ii) du théorème~1.2 de \cite{getaab} donne l'existence d'une différentielle abélienne de résidus abéliens $(r,r,-r+\epsilon,-r-\epsilon)$ dont le carré est dans $\quadomoduli[0](4;\rec[-2][4])$. Ensuite, le théorème~\ref{thm:geq0quad2} permet de réaliser la configuration de résidus quadratiques $(r-\epsilon)^{2},(r+\epsilon)^{2}$ dans la strate $\quadomoduli[0](-1,1;-2,-2)$. L'identification naturelle permet d'obtenir une différentielle primitive de $\quadomoduli[2](-1,1,4)$ avec trois cylindres disjoints.
Dans le second cas, il s'agit de construire une différentielle de $\quadomoduli[3](4,4)$ ayant quatre cylindres. Partant de deux différentielles qui sont le carré d'une différentielle abélienne de $\quadomoduli[0](2;\rec[-1][4])$ de résidus abéliens $(r,r,-r+\epsilon,-r-\epsilon)$, le bon choix de recollement fournit la différentielle adéquate.
\par
Traitons enfin le cas où les ordres des singularités sont $-1$, $1$ et $2$. Étant donnée une décomposition $(a_{1},\dots,a_{n})$ de $4g-4$, nous remplaçons $g+1$ occurrences de $2$ ou de paires $(1,1)$ par une paire $(-1,-1)$. On obtient alors une décomposition de $-4$. Pour la strate de genre zéro correspondante, nous considérons, comme dans le paragraphe dédié, un arbre avec les valences adéquates telles que chaque paire $(-1,-1)$ obtenue précédemment corresponde à un sommet. Cela donne une différentielle entrelacée réalisant la borne. Pour chaque composante irréductible contenant une paire de pôles simples obtenue à partir d'un zéro double, nous remplaçons la différentielle de $\quadomoduli[0](-1,-1;-2)$ de résidu $r^{2}$ par le carré d'une différentielle de $\quadomoduli[0](1;-1,-1,-1)$ de résidus abéliens $(r,-\frac{r}{2},-\frac{r}{2})$. Puis on identifie les deux cylindres de circonférence $-\frac{r}{2}$ pour augmenter le genre et le nombre de cylindres de un. Dans l'autre cas, nous utilisons une différentielle quadratique de $\quadomoduli[0](1,1;-2,-2,-2)$ réalisant des résidus quadratiques $(r^{2}, r_{1}^{2}, r_{1}^{2})$ avec $r$ et $r_{1}$ deux réels non commensurables. Nous obtenons ainsi les différentielles réalisant la borne dans les cas restants.

\printbibliography

\end{document}